\documentclass[10pt]{amsart}
\usepackage[USenglish]{babel}
\usepackage{amsmath, amsfonts, amssymb, amsthm, amscd}
\usepackage{bbm, color, csquotes, dsfont, enumitem, epsf, graphicx, indentfirst, latexsym, mathtools, rotating, scalerel, tikz, tikz-cd, verbatim}
\usepackage[final, breaklinks=true, colorlinks, citecolor=green, linkcolor=blue, pdfusetitle]{hyperref}

\usepackage[T1]{fontenc}

\usepackage[abbrev]{amsrefs}

\usepackage{lmodern}

\newcommand{\bm}[1]{\boldsymbol{#1}}

\newtheorem{theorem}{Theorem}[section]

\newtheorem{lemma}[theorem]{Lemma}
\newtheorem{proposition}[theorem]{Proposition}
\newtheorem{corollary}[theorem]{Corollary}

\theoremstyle{definition}
\newtheorem{definition}[theorem]{Definition}
\newtheorem{notation}[theorem]{Notation}
\newtheorem{example}[theorem]{Example}

\theoremstyle{remark}
\newtheorem{remark}[theorem]{Remark}

\numberwithin{equation}{section}

\newcommand{\R}{\mathbb{R}}
\newcommand{\Z}{\mathbb{Z}}

\newcommand{\C}{\mathbb{C}}

\newcommand{\Q}{\mathbb{Q}}

\newcommand{\de}{\delta}

\newcommand{\ep}{\varepsilon}

\newcommand{\g}{\gamma}

\newcommand{\cO}{{\mathcal O}}

\newcommand{\cR}{{\mathcal R}}
\newcommand{\cS}{{\mathcal S}}

\newcommand{\cU}{{\mathcal U}}
\newcommand{\cV}{{\mathcal V}}
\newcommand{\cW}{{\mathcal W}}
\newcommand{\cX}{{\mathcal X}}
\newcommand{\cY}{{\mathcal Y}}
\newcommand{\cZ}{{\mathcal Z}}

\newcommand{\CM}{\overline{\mathcal{M}}}

\newcommand{\ww}{\omega}

\newcommand{\delbar}{\overline{\partial}}
\newcommand{\delbarj}{\overline{\partial}_J}
\newcommand{\delbarinv}{\smash{\overline{\partial}}\vphantom{\partial}^{-1}}

\newcommand{\hdelbar}{
	\mathchoice
	{% \displaystyle
		\hat{\overline{\partial}}
	}
	{% \textstyle
		\scalerel*{\hat{\overline{\partial}}}{\hat{M}}
	}
	{%	\scriptstyle
		\scalerel*{\hat{\overline{\partial}}}{\hat{M}}
	}
	{%	\scriptscriptstyle
		\scalerel*{\hat{\overline{\partial}}}{\hat{M}}
	}
}
\newcommand{\hdelbarj}{
	\mathchoice
	{% \displaystyle
		\hat{\overline{\partial}}_J
	}
	{% \textstyle
		\scalerel*{\hat{\overline{\partial}}_J}{\hat{M}_J}
	}
	{%	\scriptstyle
		\scalerel*{\hat{\overline{\partial}}_J}{\hat{M}_J}
	}
	{%	\scriptscriptstyle
		\scalerel*{\hat{\overline{\partial}}_J}{\hat{M}_J}
	}
}
\newcommand{\tdelbar}{
	\mathchoice
	{% \displaystyle
		\tilde{\overline{\partial}}
	}
	{% \textstyle
		\scalerel*{\tilde{\overline{\partial}}}{\hat{M}}
	}
	{}{}
}

\newcommand{\abs}[1]{\lvert#1\rvert}
\newcommand{\norm}[1]{\left\Vert#1\right\Vert}

\DeclareMathOperator{\dR}{dR}

\newcommand{\red}[1]{\textcolor{red}{#1}}

\newcommand{\ssc}{\text{sc}}

\newcommand{\dmlog}{\smash{\overline{\mathcal{M}}}\vphantom{\mathcal{M}}^{\text{log}}}

\newcommand{\supp}{\operatorname{supp}}
\newcommand{\id}{\operatorname{id}}

\newcommand{\DET}{\operatorname{DET}}
\newcommand{\Lin}{\operatorname{Lin}}
\newcommand{\coker}{\operatorname{coker}}

\title[Naturality of invariants and pulling back perturbations]{Naturality of polyfold invariants \\ and pulling back abstract perturbations}
\date{\today}
\thanks{Research partially supported by Project C5 of SFB/TRR 191 ``Symplectic Structures in Geometry, Algebra and Dynamics,'' funded by the DFG}

\author{Wolfgang Schmaltz}
\address{Mathematics Institute, Justus-Liebig University, D-35392 Gie{\ss}en, Germany}
\email{\href{mailto:wolfgang.schmaltz@math.uni-giessen.de}{wolfgang.schmaltz@math.uni-giessen.de}}
\urladdr{\url{https://sites.google.com/view/wolfgang-schmaltz/home}}

\begin{document}

%%%%%%%%%%%%%%%%%%%%%%%%%%%%%%%%%%%%%%%%%%%%%%%%%%
% Abstract %%%%%%%%%%%%%%%%%%%%%%%%%%%%%%%%%%%%%%%
%%%%%%%%%%%%%%%%%%%%%%%%%%%%%%%%%%%%%%%%%%%%%%%%%%

\begin{abstract}
	%as opposed to `modifications' change the lingo to `different approaches to modeling a moduli problem'
	%It is possible that one may...
	%It is a frequent occurrence that one may construct distinct polyfolds which model a given moduli space problem in subtly different ways.
	It is possible to construct distinct polyfolds which model a given moduli space problem in subtly different ways.
	These distinct polyfolds yield invariants which, a priori, we cannot assume are equivalent. 
	%Given two such distinct polyfolds which model a moduli space problem, we provide a general framework for comparing the resulting polyfold invariants.
	%This framework allows us to prove that the polyfold invariants are \emph{natural}, in the sense that under a mild hypothesis (the existence of an ``intermediary subbundle'' of a strong polyfold bundle), the polyfold invariants for such different models are equal.
	We provide a general framework for proving that polyfold invariants are \emph{natural}, in the sense that under a mild hypothesis (the existence of an ``intermediary subbundle'' of a strong polyfold bundle) the polyfold invariants for such different models will be equal.
	As an application, we show that the polyfold Gromov--Witten invariants are independent of all choices made in the construction of the Gromov--Witten polyfolds. 
	%(in particular, the invariants are independent of the choice of increasing sequence $(\delta_i)_{i\geq 0}\subset (0,2\pi)$).
	Furthermore, we show that the polyfold Gromov--Witten invariants are independent of the choice of exponential decay at the marked points.
	%Consequently, in future applications there is no difficulty in treating marked points as nodal points. 
	
	In addition, we consider the following problem.
	Given a map between polyfolds, we cannot naively consider the restriction of this map to the respective perturbed solution spaces.
	Under a mild topological hypothesis on the map, we show how to pullback abstract perturbations which then allows us to obtain a well-defined map between the perturbed solution spaces.
	As an application, we show that there exists a well-defined permutation map between the perturbed Gromov--Witten moduli spaces.
\end{abstract}

\maketitle

\tableofcontents

%%%%%%%%%%%%%%%%%%%%%%%%%%%%%%%%%%%%%%%%%%%%%%%%%%
% INTRODUCTION %%%%%%%%%%%%%%%%%%%%%%%%%%%%%%%%%%%
%%%%%%%%%%%%%%%%%%%%%%%%%%%%%%%%%%%%%%%%%%%%%%%%%%

\section{Introduction}

\subsection{Polyfold regularization}

Consider a compactified moduli space arising from the study of $J$-holomorphic curves in symplectic geometry.
A foundational problem is finding some way to give this moduli space enough structure to define invariants.
Polyfold theory, as developed by Hofer, Wysocki, and Zehnder, has been successful in providing a general abstract framework in which it is possible to ``regularize'' such a moduli space, yielding a perturbed moduli space which has sufficient additional structure.

\begin{theorem}[Polyfold regularization theorem, {\cite[Thm.~15.4, Cor.~15.1]{HWZbook}}]
	In some established cases, we can construct a polyfold $\cZ$ such that the compactified moduli space $\CM$ is equal to the zero set of a $\ssc$-smooth Fredholm section $\delbar$ of a strong polyfold bundle $\cW \to \cZ$, i.e., $\CM = \delbarinv (0) \subset \cZ$.
	
	We may then ``regularize'' the moduli space $\CM$ by means of an ``abstract perturbation.''
	The perturbed moduli space $\CM(p):=(\delbar+p)^{-1}(0)$ then has the structure of a compact oriented ``weighted branched orbifold.''
\end{theorem}

In the boundaryless case, such an approach has been successful in regularizing the Gromov--Witten moduli spaces (see \cite{HWZGW}).
A specialized approach has yielded a proof of the Arnold conjecture (see \cite{filippenko2018arnold}).
This approach is also being used in the pursuit of a well-defined symplectic field theory (see \cite{fish2018sft}).

For a suitably constructed abstract perturbation, the perturbed moduli space $\CM(p)$ has the structure of a compact oriented weighted branched orbifold, and therefore possesses sufficient structure to define the ``branched integration'' of differential forms.

\begin{theorem}[Polyfold invariants, {\cite[Cor.~15.2]{HWZbook}}]
	Let $\cO$ be an orbifold and consider a $\ssc$-smooth map $f: \cZ \to \cO$.
	We may define the \textbf{polyfold invariant} as the homomorphism obtained by pulling back a de Rahm cohomology class from the orbifold and taking the ``branched integral'' over a perturbed moduli space:
	\[
	H^*_{\dR} (\cO) \to \R, \qquad 	\ww 			\mapsto \int_{\CM(p)} f^*\ww.
	\]
	This homomorphism does not depend on the choice of abstract perturbation.
\end{theorem}

In particular, this is precisely the form for the polyfold Gromov--Witten invariants defined in \cite[Thm.~1.12]{HWZbook}.

\subsection{Naturality of polyfold invariants}

Given a compactified moduli space $\CM$, it is possible to model $\CM$ in subtly different ways.
That is, it is possible to construct distinct polyfolds $\cZ$ and $\cZ'$ which contain $\CM$ as a compact subset, $\CM \subset \cZ$ and $\CM \subset \cZ'$.
After regularization of the moduli space $\CM$ we obtain perturbed moduli spaces $\CM(p) \subset \cZ$ and $\CM(p') \subset \cZ'$ which have the structure of compact oriented weighted branched suborbifolds.
We obtain distinct polyfold invariants by taking the branched integral over these perturbed moduli spaces.
Thus, we find ourselves in the following situation: given a moduli space $\CM$ we can define polyfold invariants associated to the distinct polyfolds $\cZ$ and $\cZ'$ and which, a priori, we cannot assume are equivalent.
\emph{Therefore the polyfold invariants, which aspire to be agnostic of all possible choices, may depend on the subtle choices made in modeling a given moduli space.}

In this paper, we provide a general framework for studying and resolving this problem.
The first step is to find a third polyfold $\cY$ which models $\CM$ and which refines the different structures or choices made, in the sense that there are inclusion maps
	\[
	\cX' \hookleftarrow \cY \hookrightarrow \cX.
	\]
The problem then reduces to showing that the polyfold invariants for $\cY$ and $\cX$ are equal.
We consider a commutative diagram of inclusion maps between polyfolds and between strong polyfold bundles of the form:
	\[\begin{tikzcd}
	\cV \arrow[r, hook] \arrow[d, "\delbar_\cY \quad"'] & \cW \arrow[d, "\quad \delbar_\cZ"] &  \\
	\cY \arrow[r, hook] \arrow[u, bend left] & \cZ \arrow[u, bend right] & 
	\end{tikzcd}\]
in addition to a commutative diagram with target space the orbifold $\cO$:
	\[\begin{tikzcd}
	&  & \cO \\
	\cY \arrow[r, hook] \arrow[rru, "f_\cY"] & \cZ \arrow[ru, "f_\cZ"'] & 
	\end{tikzcd}
	\]
As outlined at the start of \S~\ref{subsec:intermediary-subbundles-naturality}, we assume that these inclusion maps satisfy a number of properties.
Although these hypothesis are somewhat lengthy at a glance, they will be natural from the point of view of our applications, and moreover reflect some commonalities of the practical construction of distinct polyfolds which model the same moduli space.
In these application, we furthermore note that the bundle $\cV$ is not the same as the pullback bundle of $\cW$, hence we may not use the methods of pulling back abstract perturbations of Theorem~\ref{thm:pullback-regular-perturbation}.

The most substantial hypothesis is the existence of an ``intermediary subbundle,'' a subset of the target strong polyfold bundle $\cR\subset \cW$ whose object space is fiberwise a (not necessarily complete) vector space and which satisfies some additional properties (see Definition~\ref{def:intermediate-subbundle}).

\begin{theorem}[Naturality of polyfold invariants]
	\label{thm:naturality-polyfold-invariants}
	Consider a compactified moduli space $\CM$ which is modeled by two polyfolds $\cY$ and $\cZ$, i.e., $\CM \subset \cY$ and $\CM \subset \cZ$.
	Suppose there is an inclusion map $\cY \hookrightarrow \cZ$.
	Moreover, assume we satisfy the hypothesis of the general framework described in \S~\ref{subsec:intermediary-subbundles-naturality}.
	
	Suppose there exists an intermediary subbundle $\cR \subset \cW$. Then the polyfold invariants for $\cY$ and $\cZ$ defined via the branched integral are equal.
	This means that, given a de Rahm cohomology class $\ww\in H^*_{\dR} (\cO)$ the branched integrals over the perturbed moduli spaces are equal,
	\[
	\int_{\CM(p)} f_\cY^* \ww = \int_{\CM(p')} f_\cZ^* \ww,
	\]
	for any regular abstract perturbations.
\end{theorem}

The proof ends up being somewhat involved as we encounter some substantial technical difficulties, which we sketch briefly.
Roughly, the existence of an intermediary subbundle allows the construction of abstract perturbations $p'$ of the strong polyfold bundle $\cW \to \cZ$ whose restrictions induce a well-defined abstract perturbation $p$ of the strong polyfold bundle $\cV \to \cY$.
This allows us to consider a well-defined restriction between perturbed moduli spaces, 
	\[\CM(p) \hookrightarrow \CM(p').\]
On the level of topological spaces, this restriction is a continuous bijection.
While we can achieve transversality for both perturbations, the abstract polyfold machinery is only able to ``control the compactness'' of the target perturbed moduli space, hence via usual methods we can only assume that $\CM(p')$ is a compact topological space.

Using only knowledge of the underlying topologies of both of these spaces, it is impossible to say anything more. 
The key to resolving this problem is understanding the additional structure that these spaces possess---the branched orbifold structure---and using this structure to prove an invariance of domain result for weighted branched orbifolds (see Lemma~\ref{lem:invariance-of-domain-branched-orbifolds}).
This result will allow us to assert that the above map is a homeomorphism---and therefore, $\CM(p)$ is also compact.

The second major difficulty comes from the fact that the restricted perturbation $p$ on the source space is not a ``regular'' perturbation (see Definition~\ref{def:regular-perturbation}).
This is problematic due to the fact that the present theory only guarantees the existence of a compact cobordism between abstract perturbations which are both assumed to be ``regular'' (see Theorem~\ref{thm:cobordism-between-regular-perturbations}).
In order to resolve this problem, we must generalize the abstract perturbation theory to allow for perturbation of $\ssc$-smooth proper ``Fredholm multisections'' (see \S~\ref{subsec:fredholm-multisections}).
This generalization enables us to construct a compact cobordism from the restricted perturbation $p$ to a regular perturbation (see Proposition~\ref{prop:cobordism-multisection-regular}).

\subsection{Application: Naturality of the polyfold Gromov--Witten invariants}

The construction of a Gromov--Witten polyfold structure requires choices, such as the choice of a cut-off function in the gluing constructions, choices of good uniformizing families of stable maps, choice of a locally finite refinement of a cover of M-polyfold charts, as well as the exponential gluing profile.
%and the choice of the strictly increasing sequence $(\delta_i)_{i\geq 0}\subset (0,2\pi)$.  

In addition to these choices, one must also choose a strictly increasing sequence $(\delta_i)_{i\geq 0} \subset (0,2\pi)$, i.e.,
	\[
	0<\delta_0 < \delta_1 < \cdots < 2\pi.
	\]
This sequence is used to define $\ssc$-Banach spaces which are then used to define the M-polyfold models of the Gromov--Witten polyfold $\cZ_{A,g,k}^{3,\delta_0}$ (see \cite[\S~2.4]{HWZGW}).

The following theorem states that, having fixed the exponential gluing profile and a strictly increasing sequence $(\delta_i)_{i\geq 0}\subset (0,2\pi)$, different choices lead to Morita equivalent polyfold structures. Hence the Gromov--Witten polyfold invariants are independent of such choices.

\begin{theorem}[{\cite[Thm.~3.37]{HWZGW}}]
	Having fixed the exponential gluing profile and a strictly increasing sequence $(\delta_i)_{i\geq 0}\subset (0,2\pi)$, the underlying topological space $\cZ_{A,g,k}^{3,\delta_0}$ possesses a  natural equivalence class of polyfold structures.
\end{theorem}

We can use Theorem~\ref{thm:naturality-polyfold-invariants} to show that the polyfold Gromov--Witten invariants are also independent of the choice of increasing sequence, and hence are natural in the sense that they do not depend on any choice made in the construction of the Gromov--Witten polyfolds.

\begin{corollary}[Naturality of the polyfold Gromov--Witten invariants]
	\label{cor:naturality-polyfold-gw-invariants}
	The polyfold Gromov--Witten invariants do not depend on the choice of an increasing sequence $(\delta_i)_{i\geq 0} \allowbreak \subset (0,2\pi)$.
\end{corollary}

We now consider the choice of puncture at the marked points
The underlying set of the Gromov--Witten polyfolds consist of stables curves.
As constructed in \cite{HWZbook}, these stable curves are required to satisfy exponential decay estimates on punctured neighborhoods of the nodal pairs.
In contrast, for these Gromov--Witten polyfolds no such decay is required at the marked points.

However, in some situations we would like to treat the marked points in the same way as the nodal points.
For example, this is true in the context of the splitting and genus reduction axioms, where we will wish to identify a pair of marked points with the same image with a nodal pair.
Allowing a puncture with exponential decay at a specified marked point is a global condition on a Gromov--Witten polyfold, and hence different choices of puncture at the marked points yield distinct Gromov--Witten polyfolds.

We again use Theorem~\ref{thm:naturality-polyfold-invariants} to show that the polyfold Gromov--Witten invariants are independent of such choice of puncture at the marked points.

\begin{corollary}
	\label{cor:punctures-equal}
	The polyfold Gromov--Witten invariants do not depend on the choice of puncture at the marked points.
\end{corollary}

\subsection{Pulling back abstract perturbations in polyfold theory}

Consider distinct moduli spaces $\CM$ and $\CM'$ which are modeled by polyfolds $\cY$ and $\cZ$, respectively.
Consider a naturally defined $\ssc$-smooth map between polyfolds $f: \cY \to \cZ$ which restricts to a map between moduli spaces $f|_{\CM} : \CM \to \CM'$.
In many situations we would like to study the geometry of this map and in order to establish algebraic relationships between the respective polyfold invariants.

However, without work, we cannot assume that this map will \emph{persist} after abstract perturbation.
Abstract perturbations are constructed using bump functions and choices of vectors in a strong polyfold bundle, which in general we cannot assume will be preserved by the $\ssc$-smooth map $f$.

To solve this problem, consider a pullback diagram of strong polyfold bundles as follows:
	\[\begin{tikzcd}
	f^* \cW \arrow[d, "f^*\delbar \quad"'] \arrow[r, "\text{proj}_2"'] & \cW \arrow[d, "\quad \delbar"] &  \\
	\cY \arrow[r, "f"'] \arrow[u, bend left] & \cZ. \arrow[u, bend right] & 
	\end{tikzcd}\]
The natural approach for obtaining a well-defined map between the perturbed moduli spaces is to take the pullback an abstract perturbation.
The main technical point is ensuring that we can control the compactness of the pullback perturbation.
This is achieved by a mild topological hypothesis on the map $f$, called the ``topological pullback condition'' (see Definition~\ref{topological-pullback-condition}).

\begin{theorem}
	\label{thm:pullback-regular-perturbation}
	Consider a $\ssc$-smooth map between polyfolds, $f: \cY \to \cZ$, and consider a pullback diagram of strong polyfold bundles as above.
	If $f$ satisfies the topological pullback condition then there exists a regular perturbation $p$ which pulls back to a regular perturbation $f^*p$.
	
	It follows that we can consider a well-defined restriction between perturbed moduli spaces,
		\[
		f|_{\CM(f^*p)} : \CM(f^*p) \to \CM (p).
		\]
\end{theorem}

This theorem follows from the more technically stated Theorem~\ref{thm:compatible-pullbacks}.

\subsection{Application: Permutation maps between perturbed Gromov--Witten moduli spaces}

Let $(Q,\ww)$ be a closed symplectic manifold, and fix a homology class $A \in H_2 (Q;\Z)$ and integers $g,\ k\geq 0$ such that $2g+k \geq 3$.
Fix a permutation $\sigma: \{1,\ldots, k\} \to \{ 1,\ldots, k \}$.
Consider the natural $\ssc$-diffeomorpism between Gromov--Witten polyfold defined by permuting the marked points,
	\[
	\sigma: \cZ_{A,g,k}\to \cZ_{A,g,k}.
	\]
For a fixed compatible almost complex structure $J$, this map has a well-defined restriction to the unperturbed Gromov--Witten moduli spaces
	\[
	\sigma|_{\CM_{A,g,k}(J)} : \CM_{A,g,k}(J) \to \CM_{A,g,k}(J).
	\]

As we have mentioned, abstract perturbations are constructed using bump functions and choic\-es of vectors in a strong polyfold bundle, which in general will not exhibit symmetry with regards to the labelings of the marked points.
As a result, given a stable curve $x\in \cZ_{A,g,k}$ which satisfies a perturbed equation $(\delbarj+p)(x)=0$ we cannot expect that $(\delbarj+p)(\sigma(x))=0$, as the perturbations are not symmetric with regards to the permutation $\sigma$.
Therefore, naively there does not exist a permutation map between perturbed Gromov--Witten moduli spaces.

However, since $\sigma: \cZ_{A,g,k}\to \cZ_{A,g,k}$ is a homeomorphism on the level of the underlying topological spaces, it is immediate that it satisfies the topological pullback condition, hence we immediately obtain the following corollary.

\begin{corollary}
	\label{cor:pullback-via-permutation}
	There exists a regular perturbation which pulls back to a regular perturbation via the permutation map $\sigma:\cZ_{A,g,k}\to \cZ_{A,g,k}$.
	Therefore, we can consider a well-defined permutation map between the perturbed Gromov--Witten moduli spaces,
		\[
		\sigma|_{\CM_{A,g,k}(\sigma^*p)} : \CM_{A,g,k}(\sigma^*p) \to \CM_{A,g,k} (p).
		\]
\end{corollary}

\subsection{Organization of the paper}

We give a self contained introduction to the basic abstract perturbation machinery of polyfold theory in \S~\ref{sec:abstract-perturbations-polyfold-theory}.
In \S~\ref{subsec:polyfolds-ep-groupoids} we review scale calculus, the definition of a polyfold as an ep-groupoid, and discuss the induced topology on subgroupoids and on branched suborbifolds.
In \S~\ref{subsec:abstract-perturbations} we discuss strong polyfold bundles, $\ssc$-smooth Fredholm sections and $\ssc^+$-multisection perturbations. In addition, we also discuss transverse perturbations, how to control the compactness of a perturbation, and questions of orientation.
In \S~\ref{subsec:branched-integral-polyfold-invariants} we consider $\ssc$-smooth differential forms, the definition of the branched integral on a weighted branched suborbifold, and how to define the polyfold invariants.

We provide a general framework for proving that the polyfold invariants are natural, and do not depend on the construction of a polyfold model for a given moduli space in \S~\ref{sec:naturality-polyfold-invariants}.
In \S~\ref{subsec:invariance-of-domain} we prove an invariance of domain result for branched suborbifolds, Lemma~\ref{lem:invariance-of-domain-branched-orbifolds}.
In \S~\ref{subsec:fredholm-multisections} we generalize the polyfold abstract perturbation theory to the case of a $\ssc$-smooth proper Fredholm multisection.
In \S~\ref{subsec:intermediary-subbundles-naturality} we provide the general framework, introduce the definition of an intermediary subbundle, and prove that the equality of polyfolds invariants in Theorem~\ref{thm:naturality-polyfold-invariants}.
In \S~\ref{subsec:independence-sequence} we apply Theorem~\ref{thm:naturality-polyfold-invariants} to show that the polyfold Gromov--Witten invariants are independent of the choice of increasing sequence.
In \S~\ref{subsec:independence-punctures} we apply Theorem~\ref{thm:naturality-polyfold-invariants} to show that the polyfold Gromov--Witten invariants are independent of the choice of puncture at the marked points.

We discuss how to pull back regular perturbations in \S~\ref{sec:pulling-back-abstract-perturbations}.
In \S~\ref{subsec:pullbacks-strong-polyfold-bundles} we define the pullback of a strong polyfold bundle and of a $\ssc^+$-multisection.
In \S~\ref{subsec:topological-pullback-condition-controlling-compactness} we introduce the topological pullback condition and show how it allows us to pullback a pair which controls compactness.
In \S~\ref{subsec:construction-regular-perturbation-which-pullback} we construct regular perturbations which pullback to regular perturbations, proving Theorem~\ref{thm:compatible-pullbacks}.
In \S~\ref{subsec:permutation-map} we apply Theorem~\ref{thm:compatible-pullbacks} to obtain a well-defined permutation map between the perturbed Gromov--Witten moduli spaces.

In Appendix~\ref{appx:local-surjectivity} we consider some basic properties of the linearized Cauchy--Riemann operator, which allow us to assert the simple fact that cokernel vectors can be chosen so that they vanish on small neighborhoods of the marked or nodal points.

%%%%%%%%%%%%%%%%%%%%%%%%%%%%%%%%%%%%%%%%%%%%%%%%%%%%%%%%%%%%%%%%%%%%%%%%%%%%%%%%%%%%%%%%%%%%%%%%%%%%
% ABSTRACT PERTURBATIONS IN POLYFOLD THEORY %%%%%%%%%%%%%%%%%%%%%%%%%%%%%%%%%%%%%%%%%%%%%%%%%%%%%%%%
%%%%%%%%%%%%%%%%%%%%%%%%%%%%%%%%%%%%%%%%%%%%%%%%%%%%%%%%%%%%%%%%%%%%%%%%%%%%%%%%%%%%%%%%%%%%%%%%%%%%

\section{Abstract perturbations in polyfold theory}
	\label{sec:abstract-perturbations-polyfold-theory}

In this section we recall and summarize the construction of abstract perturbations in polyfold theory, as developed by Hofer, Wysocki, and Zehnder.

\subsection{Polyfolds and ep-groupoids}
	\label{subsec:polyfolds-ep-groupoids}

We use the modern language of \'etale proper Lie groupoids to define polyfolds.  
The notion of orbifold was first introduced by Satake \cite{satake1956generalization}, with further descriptions in terms of groupoids and categories by Haefliger \cites{haefliger1971homotopy,haefliger1984groupoide,haefliger2001groupoids}, and Moerdijk \cites{moerdijk2002orbifolds,moerdijk2003introduction}.  
With this perspective, a polyfold may be viewed as a generalization of a (usually infinite-dimensional) orbifold, with additional structure.  This generalization of the \'etale proper Lie groupoid language to the polyfold context is due to Hofer, Wysocki, and Zehnder \cite{HWZ3}.
For full details in the present context, we will refer the reader to \cite{HWZbook} for the abstract definitions of ep-groupoids in the polyfold context.

\subsubsection[sc-Structures, M-polyfolds, and polyfold structures]{$\ssc$-Structures, M-polyfolds, and polyfold structures}

%We take a moment to summarize the hierarchy of structures which leads to the definition of a polyfold.

We begin by discussing the basic definitions of ``scale calculus'' in polyfold theory.
Scale calculus is a generalization of classical functional analytic concepts, designed to address the classical failure of reparametrization actions to be differentiable (see \cite[Ex.~2.1.4]{ffgw2016polyfoldsfirstandsecondlook}).
Thus, scale calculus begins by generalizing notions of Banach spaces and of Fr\'echet differentiability in order to obtain scale structures where reparametrization will be a smooth action.

%In classical functional analysis, reparametrization actions fail to be differentiable (see \cite[Ex.~2.1.4]{ffgw2016polyfoldsfirstandsecondlook}).

\begin{definition}[{\cite[Def.~1.1]{HWZbook}}]
	A \textbf{$\ssc$-Banach space} consists of a Banach space $E$ together with a decreasing sequence of linear subspaces
	\[
	E=E_0\supset E_1 \supset \cdots \supset E_\infty := \cap_{i\geq 0} E_i
	\]
	such that the following two conditions are satisfied.
	\begin{enumerate}
		\item The inclusion operators $E_{m+1} \to E_m$ are compact.
		\item $E_\infty$ is dense in every $E_i$.
	\end{enumerate}
\end{definition}

\begin{definition}[{\cite[Def.~1.9]{HWZbook}}]
	\label{def:ssc-differentiability-ssc-Banach-spaces}
	A map $f:U\rightarrow U'$ between two open subsets of $\ssc$-Banach spaces $E$ and $E'$ is called a \textbf{$\ssc^0$-map}, if  $f(U_i)\subset U'_i$  for all $i\geq 0$ and if the induced maps  $f:U_i\rightarrow U'_i$ are continuous.  
	Furthermore, $f$ is called a \textbf{$\ssc^1$-map}, or of \textbf{class $\ssc^1$}, if the following conditions are satisfied.
	\begin{itemize}
		\item For every $x\in U_1$ there exists a bounded 
		linear  map $Df(x)\in  \mathcal{L}(E_0, E'_0)$ satisfying for $h\in
		E_1$, with $x+h\in U_1$,
		\[\frac{1}{\norm{h}_1}\norm{f(x+h)-f(x)-Df(x)h}_0\to 0\quad
		\text{as\ $\norm{h}_1\to 0$.}\mbox{}\\[4pt]\]
		\item  The tangent map  $Tf:TU\to TU'$,
		defined by
		\[Tf(x, h)=(f(x), Df(x)h),
		\]
		is a $\ssc^0$-map between the tangent spaces.
	\end{itemize}
\end{definition}

If $Tf:TU\to TU'$ is of class $\ssc^1$, then $f:U\to U'$ is called of class $\ssc^2$; inductively, the map $f:U\to E'$ is called of class $\ssc^k$ if the $\ssc^0$-map $T^{k-1}f:T^{k-1}U\to T^{k-1}E'$ is of class $\ssc^1$.  A map which is of class $\ssc^k$ for every $k$ is called  \textbf{$\ssc$-smooth}  or of  \textbf{class $\ssc^{\infty}$}.  The basic building block which allows us to check the $\ssc$-differentiability of maps is the chain rule.

\begin{proposition}[Chain rule, {\cite[Thm.~1.1]{HWZbook}}]
	Assume that $E$, $F$, and $G$ are $\ssc$-smooth Banach spaces and $U\subset E$ and $V\subset F$ are open sets.  Assume that $f:E\to F$, $g:V\to G$ are of class $\ssc^1$ and $f(U)=V$.  Then the composition $g\circ f :U\to G$ is of class $\ssc^1$ and the tangent maps satisfy
	\[
	T(g\circ f) = Tg \circ Tf.
	\]
\end{proposition}

%We recall that a local M-polyfold model is a pair $(O,E)$ where $E$ is an sc-Banach space, and $O$ is a $\ssc^\infty$-retract, i.e., the image of an sc-smooth map $r:U\to U$ defined on an open subset $U$ of $E$ and satisfying $r\circ r = r$.

\begin{definition}[{\cite[Defs.~2.1,~2.2]{HWZbook}}]
	Consider a $\ssc$-Banach space $E$ and consider an open subset $U\subset E$.  A $\ssc$-smooth map $r:U\to U$ is called a \textbf{$\ssc$-smooth retraction} on $U$ if $r\circ r = r$.
	A \textbf{local M-polyfold model (without boundary)} is a pair $(O,E)$ consisting of a $\ssc$-Banach space $E$ and a subset $O\subset E$ such that there exists a $\ssc$-smooth retraction $r:U \to U$ defined on an open subset $U\subset E$ such that $r(U)= O$.  We call $O$, equipped with the subspace topology $O\subset E$, a \textbf{$\ssc$-retract}.
\end{definition}

These definitions of $\ssc$-differentiability extend to local M-polyfolds models in the following way.
\begin{definition}[{\cite[Def.~2.4]{HWZbook}}]
	A map $f:O \to O'$ between two local M-polyfold models is of \textbf{class $\ssc^k$} if the composition  $f\circ r:U\to E'$ is of class $\ssc^k$ where $U\subset E$ is an open subset of the $\ssc$-Banach space $E$ and where $r:U\to U$ is a $\ssc$-smooth retraction onto $r(U)=O$. 
\end{definition}

\begin{comment}
{only include this if i want to elaborate more about reparametrization.
In Gromov--Witten theory the compactification phenomena consist of nodal curves; hence our local M-polyfold models are without boundary and the $i$th-noded curves appear as interior points in the local M-polyfold models of codimension $2i$ (see \cite[Rem.~5.3.2]{ffgw2016polyfoldsfirstandsecondlook}).
Other cases such as Hamiltonian Floer theory (as first introduced in \cite{floer1988unregularized}) and Symplectic Field Theory (as first introduced in \cite{egh2000introduction}) would require the inclusion of partial quadrants in order to deal with boundaries and corners (see \cite[Defs.~1.6,~2.2]{HWZbook} for details on local M-polyfold models with partial quadrants).}
\end{comment}

In the absence of isotropy, we may consider the following definition of an ``M-polyfold,'' short for a ``polyfold of manifold type.''
% that is, topological space which is locally homeomorphic to a local M-polyfold model and where the transition maps are all $\ssc$-smooth.

\begin{definition}[{\cite[Def.~2.8]{HWZbook}}]
	We say that a paracompact Hausdorff topological space $Z$ is an \textbf{M-polyfold} if every point $z\in Z$ has an open neighborhood %$\hat{O}$
	which is homeomorphic to a $\ssc$-retract $O$, and such that the induced transition maps between any two $\ssc$-retracts are $\ssc$-smooth.
\end{definition}

However, in almost all situations that arise isotropy is inevitable, and must be dealt with.
In this sense, polyfold behave like infinite-dimensional orbifolds, and so we introduce the language of ep-groupoids.

\begin{definition}[{\cite[Defs.~7.1,~7.3]{HWZbook}}]
	A \textbf{groupoid} $(Z,\bm{Z})$ is a small category consisting of a set of objects $Z$, a set of morphisms $\bm{Z}$ which are all invertible, and the five structure maps $(s,t,m,u,i)$ (the source, target, multiplication, unit, and inverse maps).
	An \textbf{ep-groupoid} is a groupoid $(Z,\bm{Z})$ such that the object set $Z$ and the morphism set $\bm{Z}$ are both M-polyfolds, and such that all the structure maps are $\ssc$-smooth maps which satisfy the following properties.
	\begin{itemize}
		\item \textbf{(\'etale).}  The source and target maps
		$s:\bm{Z}\to Z$ and $t:\bm{Z}\to Z$ are surjective local sc-diffeomorphisms.
		\item \textbf{(proper).}  For every point $z\in Z$, there exists an
		open neighborhood $V(z)$ so that the map
		$t:s^{-1}(\overline{V(z)})\rightarrow Z$ is a proper mapping.
	\end{itemize}
\end{definition}

For a fixed object $z\in Z$ we denote the \textbf{isotropy group of $z$} by
	\[
	\bm{G}(z) := \{	\phi \in \bm{Z} \mid s(\phi)=t(\phi = z)	\}.
	\]
By \cite[Prop.~7.4]{HWZbook}, the properness condition ensures that this is a finite group.
The \textbf{orbit space} of the ep-groupoid $(Z,\bm{Z})$,
	\[
	\abs{Z} := Z / \sim,
	\]
is the quotient of the set of objects $Z$ by the equivalence relation given by $z\sim z'$ if there exists a morphism $\phi\in \bm{Z}$ with $s(\phi)=z$ and $t(\phi)=z'$.  It is equipped with the quotient topology defined via the map 
	\begin{equation}\label{eq:quotient-map}
	\pi: Z\to\abs{Z}, \qquad z\mapsto \abs{z}.
	\end{equation}

% M-polyfolds:
% FirstLook	- metrizable topological space
% HWZGW 	- paracompact, Hausdorff topological space
% HWZbook	- paracompact, Hausdorff topological space

% Polyfolds:
% HWZ3	- second countable, paracompact topological space
% HWZGW	- paracompact, Hausdorff topological space
% HWZbook	- proposition shows that polyfold is locally metrizable, regular, Hausdorff
%			- moreover, if paracompact => metrizable

% manifolds are second countable, Hausdorff, locally Euclidean => paracompact
% to work with orbifolds, best to assume they are second countable, paracompact, Hausdorff
% so, for polyfolds, let's just assume they are second countable, paracompact, Hausdorff

\begin{definition}[{\cite[Def.~16.1]{HWZbook}}]
	Let $\cZ$ be a second countable, paracompact, Hausdorff topological space.  A \textbf{polyfold structure} on $\cZ$ consists of an ep-groupoid $(Z,\bm{Z})$ and a homeomorphism $\abs{Z}\simeq \cZ$.
\end{definition}

Defining an ep-groupoid involves making a choice of local structures.  Taking an equivalence class of ep-groupoids makes our differentiable structure choice independent.  The appropriate notion of equivalence in this category-theoretic context is a ``Morita equivalence class'' (see \cite[Def.~3.2]{HWZ3}).%We can thus give the definition of an orbifold or a polyfold.

\begin{definition}[{\cite[Def.~16.3]{HWZbook}}]
	A \textbf{polyfold} consists of a second countable, paracompact, Hausdorff topological space $\cZ$ together with a Morita equivalence class of polyfold structures $[(Z,\bm{Z})]$ on $\cZ$.
\end{definition}

Taking a Morita equivalence class of a given polyfold structure (in the case of polyfolds) is analogous to taking a maximal atlas for a given atlas (in the usual definition of manifolds).
Given distinct polyfold structures which define an orbifold or a polyfold, the method of proving they define the same Morita equivalence class is by demonstrating that both polyfold structures possess a common refinement.

The scales of a $\ssc$-Banach space induce a filtration on the local M-polyfold models, which is moreover preserved by the structure maps $s,t$.  Consequently, there is a well-defined filtration on the orbit space which hence induces a filtration
	\[
	\cZ = \cZ_0 \supset \cZ_1 \supset \cdots \supset \cZ_\infty = \cap_{k\geq 0} \cZ_k
	\]
on the underlying topological space $\cZ$.

\begin{notation}
	It is common to denote both the ep-groupoid ``$(Z,\bm{Z})$,'' and its object set ``$Z$,'' by the same letter ``$Z$.''	
	We will refer to the underlying set, the underlying topological space, or the polyfold by the letter ``$\cZ$.''
	We will always assume that a topological space $\cZ$ with a polyfold structure is necessarily second countable, paracompact, and Hausdorff.	
	Furthermore, we will write objects as ``$x\in Z$,'' morphisms as ``$\phi \in \bm{Z}$,'' and points as ``$[x]\in \cZ$'' (due to the identification $\abs{Z} \simeq \cZ$). We will write ``$\phi: x\to y$'' for a morphism $\phi \in \bm{Z}$ with $s(\phi)=x$ and $t(\phi)=y$.
\end{notation}

The local topology of a polyfold is related to the local isotropy groups, as demonstrated by the following proposition.

\begin{proposition}[Natural representation of $\bm{G}(x)$, {\cite[Thm.~7.1, Prop.~7.6]{HWZbook}}]
	\label{prop:natural-representation}
	Let be an ep-groupoid $(Z,\bm{Z})$.  Let $x\in Z$ with isotropy group $\bm{G}(x)$.  Then for every open neighborhood $V$ of $x$ there exists an open neighborhood $U\subset V$ of $x$, a group homomorphism $\Phi : \bm{G}(x)\rightarrow \text{Diff}_{\ssc}(U)$, $g\mapsto  \Phi (g)$,  and a $\ssc$-smooth map
	$\Gamma: \bm{G}(x)\times U\rightarrow \bm{Z}$ such that the following holds.
	\begin{enumerate}
		\item $\Gamma(g,x)=g$.
		\item $s(\Gamma(g,y))=y$ and $t(\Gamma(g,y))=\Phi (g)(y)$ for all $y\in U$ and $g\in \bm{G}(x)$.
		\item If $h: y\rightarrow z$ is a morphism between points in $U$, then there exists a unique element $g\in \bm{G}(x)$ satisfying $\Gamma(g,y)=h$, i.e., 
		\[
		\Gamma: \bm{G}(x)\times U\rightarrow \{\phi\in \bm{Z} \mid   \text{$s(\phi)$ and $t(\phi)\in U$}\}
		\]
		is a bijection.
	\end{enumerate}
	The data $(\Phi,\Gamma)$ is called the \textbf{natural representation} of $\bm{G}(x)$.
	Moreover, consider the following topological spaces:
	\begin{itemize}
		\item $\bm{G}(x) \backslash U$, equipped with quotient topology defined by the projection $U \to \bm{G}(x) \backslash U$,
		\item $U / \sim$, where $x \sim x'$ for $x, x' \in U$ if there exists a morphism $\phi \in \bm{Z}$ with $s(\phi)=x$ and $t(\phi)=x$, equipped with the quotient topology defined by the projection $U \to U / \sim$,
		\item $\abs{U}$, the image of $U$ under the map $Z\to \abs{Z}$, equipped with the subspace topology defined by the inclusion $\abs{U}\subset \abs{Z}$.
	\end{itemize}
	Then these spaces are all naturally homeomorphic.
\end{proposition}

\subsubsection{Maps between polyfolds}

Using category-theoretic language, we discuss the definition of map between polyfolds.

\begin{definition}
	A \textbf{$\ssc^k$ functor} between two polyfold structures
	\[
	\hat{f}:(Z_1,\bm{Z}_1) \to (Z_2,\bm{Z}_2)
	\]
	is a functor on groupoidal categories which moreover is a $\ssc^k$ map when considered on the object and morphism sets.
\end{definition}

A $\ssc^k$ functor between two polyfold structures $(Z_1,\bm{Z}_1)$, $(Z_2,\bm{Z}_2)$ with underlying topological spaces $\cZ_1$, $\cZ_2$ induces a continuous map on the orbit spaces $\abs{\hat{f}}:\abs{Z_1} \to \abs{Z_2}$, and hence also induces a continuous map $f : \cZ_1 \to \cZ_2$, as illustrated in the following commutative diagram.
	\[
	\begin{tikzcd}[row sep=small]
	\abs{Z_1} \arrow[d,phantom,"\rotatebox{90}{\(\simeq\)}"] \arrow[r,"\abs{\hat{f}}"] & \abs{Z_2} \arrow[d,phantom,"\rotatebox{90}{\(\simeq\)}"] \\
	\cZ_1 \arrow[r,"f"] & \cZ_2
	\end{tikzcd}
	\]

\begin{definition}
	Consider two topological spaces  $\cZ_1$, $\cZ_2$ with orbifold structures $(Z_1,\bm{Z}_1)$, $(Z_2,\bm{Z}_2)$. We define a \textbf{$\ssc^k$ map between polyfolds} as a continuous map 
	\[
	f: \cZ_1 \to \cZ_2
	\]
	between the underlying topological spaces of the polyfolds, for which there exists an associated $\ssc^k$ functor
	\[
	\hat{f}: (Z_1,\bm{Z}_1) \to (Z_2,\bm{Z}_2).
	\]
	such that $\abs{\hat{f}}$ induces $f$.
\end{definition}

\begin{remark}
	From an abstract point of view a stronger notion of map is needed.   This leads to the definition of \textit{generalized maps} between orbifold structures, following a category-theoretic localization procedure \cite[\S~2.3]{HWZ3}.  Following this, a precise notion of map between two polyfolds is defined using an appropriate equivalence class of a given generalized map between two given polyfold structures \cite[Def.~16.5]{HWZbook}.	
	With this in mind, taking an appropriate equivalence class of a given $\ssc^k$-functor between two given polyfold structures is sufficient for giving a well-defined map between two polyfolds.
\end{remark}

\subsubsection{Subgroupoids}

We state some essential facts about the topology of subgroupoids.

\begin{definition}	
	Let $(Z,\bm{Z})$ be an ep-groupoid.
	We say that a subset of the object set, $S\subset Z$, is \textbf{saturated} if $S = \pi^{-1} (\pi(S))$, where $\pi$ is the quotient map \eqref{eq:quotient-map}.
	We define a \textbf{subgroupoid} as the full subcategory $(S,\bm{S})$ associated to a saturated subset of the object set.
	%We define a \textbf{subgroupoid} $(S,\bm{S})$ associated to a saturated subset of the object set, $S\subset Z$, is a full subcategory.
	%set which is closed under morphisms, i.e., if $\phi \in \bm{Z}$ and $s(\phi) \in S$ then $t(\phi)\in S$ (equivalently if $t(\phi)\in S$ then $s(\phi)\in S$).
\end{definition}

A subgroupoid $(S,\bm{S})$ comes equipped with the subspace topology induced from the ep-groupoid $(Z,\bm{Z})$, in addition to the induced grading.  It does not come with a $\ssc$-smooth structure in general, so the \'etale condition no longer makes sense.  However, one may observe it inherits the following directly analogous properties.
\begin{itemize}
	\item The source and target maps are surjective local homeomorphisms which moreover respect the induced grading. We say that the source and target maps are \textbf{$\ssc^0$-homeomorphisms} and the subgroupoid $(S,\bm{S})$ is automatically \textbf{$\ssc^0$-\'etale}.
	\item For every point $x\in S$, there exists an open neighborhood $V(x)$ so that the map $t: s^{-1}(\overline{V(x)}) \to S$ is a proper mapping. (This can be shown from the definitions, using in addition that if $f:X\to Y$ is proper, then for any subset $V\subset Y$ the restriction $f|_{f^{-1}(V)}: f^{-1}(V)\to V$ is proper.)
\end{itemize}
Thus, a subgroupoid is automatically $\ssc^0$-\'etale in the above sense, as well as proper.

\begin{remark}
	\label{rmk:local-topology-subgroupoid}
	Let $U$ be an open subset of $S$.
	We may consider two topologies on $U$:
	\begin{itemize}
		\item $(U, \tau_S)$, where $\tau_S$ is the subspace topology induced from the inclusion $\cup_{i\in I} M_i \allowbreak \hookrightarrow S$,
		\item $(U, \tau_Z)$, where $\tau_Z$ is the subspace topology induced from the inclusion $\cup_{i\in I} M_i \allowbreak \hookrightarrow Z$.
	\end{itemize}
	Then these two topologies are identical. Moreover, $U\hookrightarrow S$ is a local homeomorphism.
\end{remark}

\begin{comment}
	Let $Y$ be a topological space.  Let $X$ be a subset of $Y$, and equip $X$ with the subspace topology.  Consider an open subset $A\subset X$, hence $A= X\cap U$ for some open subset $U$ of $Y$.
	We can equip $A$ with the subspace topology induced from the inclusion $A\hookrightarrow X$, or the subspace topology induced from the inclusion $A\hookrightarrow Y$.
	Then these two topologies are identical.  Moreover, $A\hookrightarrow X$ is a local homeomorphism.
\end{comment}

\begin{proposition}
	\label{prop:topology-subgroupoid}
	Consider the orbit space of a subgroupoid, $\abs{S}$.  There are two topologies on this space we may consider:
	\begin{itemize}
		\item the subspace topology $\tau_s$, induced from the inclusion $\abs{S}\subset \abs{Z}$,
		\item the quotient topology $\tau_q$, induced from the projection $S\to \abs{S}$.
	\end{itemize}
	These two topologies are identical.
\end{proposition}
\begin{proof}
We show that $\tau_s = \tau_q$.
	\begin{itemize}
		\item $\tau_s \subset \tau_q$
	\end{itemize}
	Suppose $U\subset \abs{S}$ and $U\in \tau_s$.  Then $U = V\cap \abs{S}$ for $V\subset \abs{Z}$ open.  By definition, $\pi^{-1} (V) \subset Z$ is open.  Moreover, $\pi^{-1} (U)  = \pi^{-1} (V) \cap \pi^{-1}(S) = \pi^{-1}(V)\cap S$.  Hence $\pi^{-1}(U)$ is open in $S$.  It follows from the definition of the quotient topology that $U\in \tau_q$.
	\begin{itemize}
		\item $\tau_q \subset \tau_s$
	\end{itemize}
	Suppose $U\subset \abs{S}$ and $U\in \tau_q$.  We will show for every $[x]\in U$ there exists a subset $B \subset \abs{S}$ such that $B \in \tau_s$ and $[x] \in B \subset U$.  It will then follow that $U\in\tau_s$, as desired.
	
	Let $x\in \pi^{-1}(U)$ be a representative of $[x]$.  There exists an open neighborhood $V(x) \subset Z$ equipped with the natural action by $\bm{G}(x)$ and such that $V(x) \cap S \subset \pi^{-1}(U)$.
	Observe that $\abs{V(x)\cap S} = \abs{V(x)} \cap \abs{S}$; this follows since $S$ is saturated.	
	
	Let $B:= \abs{V(x)} \cap \abs{S} \subset U$.  Then observe that $\abs{V(x)} \subset \abs{Z}$ is open,
	since the quotient map $\pi : Z \to \abs{Z}$ is an open map (see \cite[Prop.~7.1]{HWZbook}).
	Hence $B:=\abs{V(x)} \cap \abs{S}\subset \abs{S}$ is open in the subspace topology.  It follows that  $B \in \tau_s$ and $[x] \in B \subset U$, as desired.
\end{proof}

The following proposition is an analog of Proposition~\ref{prop:natural-representation} for subgroupoids.

\begin{proposition}[Induced representation of $\bm{G}(x)$ for a subgroupoid]
	\label{prop:natural-representation-subgroupoid}
	Let $(S,\bm{S})$ be a subgroupoid of an ep-groupoid $(Z,\bm{Z})$.  Let $x\in S$ with isotropy group $\bm{G}(x)$.  Then for every open neighborhood $V$ of $x$ there exists an open neighborhood $U\subset V$ of $x$, a group homomorphism $\Phi : \bm{G}(x)\rightarrow \text{Homeo}_{\ssc^0}(U)$, $g\mapsto  \Phi (g)$,  and a $\ssc^0$-map
	$\Gamma: \bm{G}(x)\times U\rightarrow \bm{S}$ such that the following holds.
	\begin{enumerate}
		\item $\Gamma(g,x)=g$,
		\item $s(\Gamma(g,y))=y$ and $t(\Gamma(g,y))=\Phi (g)(y)$ for all $y\in U$ and $g\in \bm{G}(x)$,
		\item if $h: y\rightarrow z$ is a morphism between points in $U$, then there exists a unique element $g\in \bm{G}(x)$ satisfying $\Gamma(g,y)=h$, i.e., 
		\[
		\Gamma: \bm{G}(x)\times U\rightarrow \{\phi\in \bm{Z} \mid   \text{$s(\phi)$ and $t(\phi)\in U$}\}
		\]
		is a bijection.
	\end{enumerate}
	Moreover, consider the following topological spaces:
	\begin{itemize}
		\item $\bm{G}(x) \backslash U$, equipped with quotient topology via the projection $U \to \bm{G}(x) \backslash U$,
		\item $U / \sim$, where $x \sim x'$ for $x, x' \in U$ if there exists a morphism $\phi : x \to x'$, equipped with the quotient topology via the projection $U \to U / \sim$,
		\item $\abs{U}$, the image of $U$ under the map $S\to \abs{S}$, equipped with the subspace topology,
		\item $\abs{U}$, the image of $U$ under the map $Z \to \abs{Z}$, equipped with the subspace topology.
	\end{itemize}
	Then these spaces are all naturally homeomorphic.
\end{proposition}

\subsubsection{Weighted branched suborbifolds}

View $\Q^+:= \Q \cap [0,\infty)$ as an ep-groupoid, having only the identities as morphisms. 
Consider a polyfold, consisting of a polyfold structure $(Z,\bm{Z})$ and an underlying topological space $\cZ$.
Consider a functor $\hat{\theta}: (Z,\bm{Z}) \to \Q^+$ which induces the function $\theta:=\abs{\hat{\theta}} :\cZ \to \Q^+$.
Observe that $\hat{\theta}$ defines a subgroupoid $(S,\bm{S})\subset (Z,\bm{Z})$ with object set
	\[
	S:= \supp (\hat{\theta}) = \{x\in Z\mid \hat{\theta}(x)>0 \}
	\]
and with underlying topological space
	\[
	\cS := \supp (\theta) = \{[x]\in \cZ \mid \theta([x])>0\}.
	\]
Moreover, $(S,\bm{S})$ is a full subcategory of $(Z,\bm{Z})$ whose object set is saturated, i.e., $S= \pi^{-1} (\pi(S))$ where $\pi : Z \to \abs{Z}, x\mapsto [x]$.

\begin{definition}[{\cite[Def.~9.1]{HWZbook}}]
	\label{def:weighed-branched-suborbifold}
	A \textbf{weighted branched suborbifold structure} consists of a subgroupoid $(S,\bm{S}) \subset (Z,\bm{Z})$ defined by a functor $\hat{\theta} : (Z,\bm{Z}) \to \Q^+$ as above which satisfies the following properties.
	\begin{enumerate}
		\item $\cS \subset \cZ_\infty$.
		\item Given an object $x\in S$, there exists an open neighborhood $U\subset Z$ of $x$ and a finite collection $M_i$, $i\in I$ of finite-dimensional submanifolds of $Z$ (in the sense of \cite[Def.~4.19]{HWZ2}) such that
		\[
		S \cap U= \bigcup_{i \in I}M_i.
		\]
		We require that the inclusion maps $\phi_i: M_i\hookrightarrow U$ are proper and are {topological embeddings,} and in addition we require that the submanifolds $M_i$ all have the same dimension. 
		The submanifolds $M_i$ are called \textbf{local branches} in $U$. \label{def:local-branches}
		\item There exist positive rational numbers $w_i$, $i\in I$, (called \textbf{weights}) such that if $y\in S \cap U$, then
		\[\hat{\theta}(y)=\sum_{\{i \in I \mid   y\in M_i\}} w_i.\]
	\end{enumerate}
	%This crucial fact is buried in the definition of a finite dimensional submanifold of an M-polyfold.
	We call ${(M_i)}_{i\in I}$ and ${(w_i)}_{i\in I}$ a \textbf{local branching structure}.
\end{definition}

By shrinking the open set $U$ we may assume that the local branches $M_i$ (equipped with the subspace topology induced from $U$) are homeomorphic to open subsets of $\R^n$.  Hence we may assume that a local branch is given by a subset $M_i\subset\R^n$ and an inclusion map $\phi_i : M_i\hookrightarrow U$ where $\phi_i$ is proper and a homeomorphism onto its image.

%refer to \cite[Theorem 9.1, Definition 9.8]{HWZbook} for a more formal definition.

\begin{definition}
	\label{def:local-orientation}
	Let $(S,\bm{S})$ be a weighted branched suborbifold structure. Consider an object $x\in S$ and a local branching structure $(M_i)_{i\in I}$, $(w_i)_{i\in I}$ at $x$.
	Suppose moreover that each local branch has an \textit{orientation}, denoted as $(M_i,o_i)$
	
	We define a \textbf{local orientation} at $x$ with respect to the local branching structure $(M_i)_{i\in I}$, $(w_i)_{i\in I}$ as the following finite formal sum of weighted oriented tangent planes:
	\[
	\sum_{\{i\in I \mid x\in M_i\}} w_i \cdot T_x (M_i,o_i).
	\]
\end{definition}

\begin{definition}
	\label{def:orientation}
	Let $(S,\bm{S})$ be a weighted branched suborbifold structure.
	We define an \textbf{orientation} on $(S,\bm{S})$ as a local orientation at every object $x\in S$ and local branching structure $(M_i)_{i\in I}$, $(w_i)_{i\in I}$ at $x$ such that the following holds.
	\begin{enumerate}
		\item We require that the local orientation is well-defined and does not depend on choice of local branching structure. Given an object $x\in S$, suppose we have:
		\begin{itemize}
			\item a local orientation at $x$ with respect to a local branching structure $(M_i)_{i\in I}$, $(w_i)_{i\in I}$,
			\item a local orientation at $x$ with respect to a local branching structure $(M'_j)_{j\in I'}$, $(w'_j)_{j\in I'}$.
		\end{itemize}
		We require the finite formal sums of weighted oriented tangent planes to be identical, i.e.,
		\[
		\sum_{\{i\in I \mid x\in M_i\}} w_i \cdot T_x (M_i,o_i)= \sum_{\{j\in I' \mid x\in M'_j\}} w'_j \cdot T_x (M'_j,o_j)
		\]
		\item We require morphism invariance of the local orientations. Given a morphism, $\phi : x \to y$ there exists a well-defined tangent map $T\phi : T_xZ \to T_yZ$. 
		Suppose we have:
		\begin{itemize}
			\item a local orientation at $x$ with respect to a local branching structure $(M_i)_{i\in I}$, $(w_i)_{i\in I}$,
			\item a local orientation at $y$ with respect to a local branching structure $(M'_j)_{j\in I'}$, $(w'_j)_{j\in I'}$.
		\end{itemize}		
		The image of a finite formal sum of weighted oriented tangent planes under this map is again a finite formal sum of weighted oriented tangent planes.
		We require invariance of the local orientations under this map, i.e.,
		\[
		\sum_{\{j\in I' \mid y\in M'_j\}} w'_j \cdot T_y (M'_j,o'_j) = \sum_{\{i\in I \mid x\in M_i\}} w_i \cdot T\phi_* (T_x (M_i,o_i)).
		\]
	\end{enumerate}
\end{definition}

A \textbf{weighted branched suborbifold structure with boundary} consists of a subgroupoid $(S,\bm{S})\subset (Z,\bm{Z})$ defined identically to Definition~\ref{def:weighed-branched-suborbifold} except we allow the possibility that the local branches are manifolds with boundary.
A \textbf{local orientation} at an object $x\in S$ is again defined as in Definition~\ref{def:local-orientation} as a finite formal sum determined by orientations of the local branches, and again an \textbf{orientation} is also defined similarly to Definition~\ref{def:orientation}.

\subsection{Abstract perturbations in polyfold theory}
	\label{subsec:abstract-perturbations}

Abstract perturbations in polyfold theory are a mixture of two different technologies:
\begin{enumerate}
	\item scale calculus generalizations of classical Fredholm theory, involving the development of analogs of Fredholm maps, compact perturbations, and the implicit function theorem for surjective Fredholm operators (originally developed in \cite{HWZ2});
	\item equivariant transversality through the use of ``multisections;'' due to the presence of nontrivial isotropy, it is generally impossible to obtain transversality through the use of single valued sections, and thus it is necessary to work with multisections (developed in \cite{cieliebak2003equivariant} and generalized to polyfold theory in \cite{HWZ3}).
\end{enumerate}

\subsubsection[Strong polyfold bundles and sc+-multisections]{Strong polyfold bundles and $\ssc^+$-multisections}
	\label{subsubsec:polyfold-abstract-perturbations}

In order to develop a Fredholm theory for polyfolds, it is necessary to formulate the notion of a ``strong polyfold bundle'' over a polyfold.
Let $P:W\to Z$ be a strong M-polyfold bundle (see \cite[Def.~2.26]{HWZbook}).
Recall that a fiber  $p^{-1} (y) = W_y$ over an object $y \in O_x$ carries the structure of a $\ssc$-Banach space. Furthermore $W$ is equipped with a double filtration $W_{m,k}$ for $0\leq k\leq m+1$, and the filtered spaces
	\begin{gather*}
	W[0]:= W_{0,0} \supset W_{1,1}\supset \cdots \supset W_{i,i} \supset \cdots,	\\
	W[1]:= W_{0,1} \supset W_{1,2}\supset \cdots \supset W_{i,i+1}\supset \cdots 
	\end{gather*}
are both M-polyfolds in their own rights.
With respect to these filtrations, the maps $P[0]: W[0]\to Z$ and $P[1]:W[1]\to Z$ are both $\ssc$-smooth.

\begin{proposition}[{\cite[Prop.~2.16]{HWZbook}}]
	\label{prop:pullback-bundle}
	Let $P: W \to Z$ be a strong M-polyfold bundle, and let $f: Y \to Z$ be a $\ssc$-smooth map between M-polyfolds. The pullback $f^* W := \{(y,w_x) \in Y \times W \mid f(y)=x=P(w_x) \}$ carries a natural structure of a strong M-polyfold bundle over the M-polyfold $Y$.
\end{proposition}

Let $(Z,\bm{Z})$ be a polyfold structure, and consider a strong M-polyfold bundle over the object space, $P:W\to Z.$
The source map $s:\bm{Z}\to Z$ is a local $\ssc$-diffeomorphism, and hence we may consider the fiber product
	\[
	\bm{Z} _s\times_P W = \{(\phi,w)\in \bm{Z}\times W	\mid	s(\phi)=P(w)	\}.
	\]
Via the above proposition, we can also view as $\bm{Z} _s\times_P W$ as the pullback bundle via $s$ over the morphism space $\bm{Z}$,
	\[\begin{tikzcd}
	\bm{Z} _s\times_P W \arrow[r] \arrow[d] & W \arrow[d] \\
	\bm{Z} \arrow[r, "s"] & Z.
	\end{tikzcd}\]

\begin{definition}[{\cite[Def.~8.4]{HWZbook}}]
	\label{def:strong-polyfold-bundle}
	A \textbf{strong polyfold bundle structure} $(W,\bm{W})$ over a polyfold structure $(Z,\bm{Z})$ consists of a strong M-polyfold bundle over the object M-polyfold $P:W\to Z$ together with a strong bundle map
	\[
	\mu : \bm{Z} _s\times_P W \to W
	\]
	which covers the target map $t:\bm{Z} \to Z$, such that the diagram
	\begin{center}
		\begin{tikzcd}
		\bm{Z} _s\times_P W \arrow[r, "\mu"] \arrow[d] & W \arrow[d] \\
		\bm{Z} \arrow[r, "t"] & Z
		\end{tikzcd}
	\end{center}
	commutes.
	Furthermore we require the following:
	\begin{enumerate}
		\item $\mu$ is a surjective local diffeomorphism and linear on fibers,
		\item $\mu(\id_x,w)= w$ for all $x\in Z$ and $w\in W_x$,
		\item $\mu(\phi \circ \g ,w)= \mu (\phi,\mu(\g,w))$ for all $\phi,\g\in\bm{Z}$ and $w\in W$ which satisfy
		\[
		s(\g) = P(w),\qquad t(\g) = s(\phi) = P(\mu(\g,w)).
		\]
	\end{enumerate}
\end{definition}

A strong polyfold bundle structure $(W,\bm{W})$ has polyfold structures in its own right: we may take $W$ as the object set with the grading $W_{i,i}$ or $W_{i,i+1}$, and define the morphism set by $\bm{W}:= \bm{Z} _s\times_P W$. Moreover, we have source and target maps $s,t: \bm{W} \to W$ defined as follows:
	\[
	s(\phi,w) := w, \qquad t(\phi,w) := \mu(\phi,w).
	\]
We have a natural smooth projection functor $\hat{P}: (W,\bm{W}) \to (Z,\bm{Z})$.

\begin{definition}
	A \textbf{strong polyfold bundle} consists of a topological space $\cW$ together with a Morita equivalence class of strong polyfold bundle structures $(W,\bm{W})$.
\end{definition}

The double filtration of the fibers is preserved by the structure maps, and hence the orbit space $\abs{W}$ is equipped with a double filtration
	\[
	\abs{W}_{m,k}, \quad \text{for } 0\leq m \ \text{and}\ 0\leq k\leq m+1.
	\]
We moreover obtain polyfolds $\cW[0]$ and $\cW[1]$ with the filtrations $\cW[0]_i := \cW_{i,i}$ and $\cW[1]_i := \cW_{i,i+1}$.  Unless specified, ``$\cW$'' refers to the first filtration, i.e., $\cW[0]$ and ``$P$'' refers to the projection map with respect to this filtration.

\begin{definition}[{\cite[Def.~12.1]{HWZbook}}]
	We define a \textbf{$\ssc$-smooth Fredholm section} of the strong polyfold bundle $P:\cW\to\cZ$ as a $\ssc$-smooth map between polyfolds $\delbar: \cZ\to \cW$ which satisfies $P\circ \delbar = \id_\cZ$ (where $\id_\cZ$ is the identity map on $\cZ$)
	%together with an associated $\ssc$-smooth functor between polyfold structures $\hdelbar: (Z,\bm{Z}) \to (W,\bm{W})$ where $\abs{\hdelbar}$ induces $\delbar$.
	We require that $\delbar$ is \textbf{regularizing}, meaning that if $[x] \in \cZ_m$ and $\delbar ([x]) \in \cW_{m,m+1}$ then $[x]\in \cZ_{m+1}$.
	Finally, we require that at every smooth object $x\in Z$ the germ $(\hdelbar,x)$ is a ``Fredholm germ'' (see \cite[Def.~3.7]{HWZbook}).
\end{definition}

\begin{definition}
	\label{def:unperturbed-solution-space}
	We define the \textbf{unperturbed solution space} of $\delbar$ as the set 
		\[
		\cS(\delbar) :=\{ [x]\in \cZ\mid \delbar([x]) = 0\} \subset \cZ,
		\]
	with topology given by the subspace topology induced from $\cZ$.
	The space $\cS(\delbar)$ has an associated subgroupoid structure $(S(\hdelbar), \bm{S}(\hdelbar))$ defined as follows:
\begin{itemize}
	\item (saturated) object set: $S(\hdelbar) := \{	x\in Z \mid \hdelbar(x)=0	\} \subset Z$,
	\item morphism set: $\bm{S} (\hdelbar) := \{ \phi \in \bm{Z}	\mid s(\phi) \in S(\hdelbar) \ (\text{equivalently } t(\phi)\in S(\hdelbar))	\} \subset	\bm{Z}$.
\end{itemize}
Both the object and morphism sets carry the subspace topology induced from the topologies on the object space $Z$ and morphism space $\bm{Z}$.
\end{definition}

We say that the Fredholm section $\delbar$ is \textbf{proper} if the unperturbed solution space $\cS(\delbar)$ is a compact topological space.

\begin{definition}[{\cite[Def.~2.24]{HWZbook}}]
	A \textbf{$\ssc^+$-section} is a $\ssc$-smooth map $s: Z \to W[1]$ which satisfies $P\circ s = \id_Z$
\end{definition}

The significance of this definition is captured in the fact that if $(\hdelbar,x)$ is a Fredholm germ and $s$ is a germ of a $\ssc^+$-section around $y$, then $(\hdelbar+s,x)$ remains a Fredholm germ.
This follows tautologically from the definition of a Fredholm germ (see the comment following \cite[Def.~2.44]{HWZGW}).
We may view the relationship of Fredholm sections and $\ssc^+$-sections in the current theory as the analogs of Fredholm and compact operators in classical functional analysis.

One can view a ``multisection'' as the rationally weighted characteristic function of an equivariant collection of locally defined single valued sections.
This is made precise in the following definition.

\begin{definition}
	\label{def:sc-multisection}
	We view $\Q^+:= \Q \cap [0,\infty)$ as an ep-groupoid, having only the identities as morphisms.
	A \textbf{$\ssc^+$-multisection} of a strong polyfold bundle $P:\cW\to \cZ$ consists of the following:
		\begin{itemize}
			\item a function $\Lambda:\cW \to \Q^+$,
			\item an associated functor $\hat{\Lambda}: W \to \Q^+$ where $\abs{\hat{\Lambda}}$ induces $\Lambda$,
		\end{itemize}
	such that at every $[x]\in \cZ$ there exists a \textbf{local section structure} defined as follows.
	Let $x\in Z$ be a representative of $[x]$ and let $U\subset Z$ be a $\bm{G}(x)$-invariant open neighborhood of $x$, and consider the restricted strong M-polyfold bundle $P: W|_U \to U$.
	Then there exist finitely many $\ssc^+$-sections $s_1,\ldots,s_k : U \to W_U$ (called \textbf{local sections}) with associated positive rational numbers $\sigma_1,\ldots ,\sigma_k \in \Q^+$ (called \textbf{weights}) which satisfy the following:
	\begin{enumerate}
		\item $\sum_{i=1}^k \sigma_i =1.$
		\item The restriction $\hat{\Lambda}|_{W|_{U}}: W_U \to \Q^+$ is related to the local sections and weights via the equation 
			\[
			\hat{\Lambda}|_{W|_{U}}(w)=\sum_{i\in \{1,\ldots, k \mid w=s_i(P(w))\}} \sigma_i
			\]
		where the empty sum has by definition the value $0$.
	\end{enumerate}	
\end{definition}

We define the \textbf{domain support} of $\Lambda$ as the subset of $\cZ$ given by
	\[
	\text{dom-supp}(\Lambda) := \text{cl}_\cZ (	\{	[x]\in\cZ \mid \exists [w]\in \cW_{[x]}\setminus\{0\} \text{ such that } \Lambda([w])>0	\}	).
	\]
\begin{comment}
The  \textbf{domain support} of $\Lambda$ is the subset of $Z$, defined by 
\[
\text{dom-supp}(\Lambda) = \text{cl}_Z (\{x\in Z \mid \text{there exists $w\in W_x\setminus\{0\}$ for which  $\Lambda(w)>0$}\}).
\]

The \textbf{support of $\hat{\Lambda}$} is the subset of $\supp(\hat{\Lambda})$ of $\cW$ defined by
\[
\supp(\hat{\Lambda}) = \{	[w]\in\cW \mid \hat{\Lambda}(w)>0	\}
\]
The \textbf{support of $\Lambda$} is the subset $\supp(\Lambda)$ of $W$ defined by
\[
\supp(\Lambda)=\{w\in W	\mid	\Lambda(w)>0\}.
\]
\end{comment}

\begin{definition}
	\label{def:perturbed-solution-space}
	Associated to a $\ssc$-smooth Fredholm section $\delbar$ and a $\ssc^+$-multisection $\Lambda$, 
	we define the \textbf{perturbed solution space} as the set
		\[
		\cS(\delbar,\Lambda) :=\{[x]\in\cZ \mid \Lambda(\delbar([x]))	>0	\}\subset \cZ
		\]
	with topology given by the subspace topology induced from $\cZ$. It is equipped with a \textbf{weight function} $\Lambda\circ \delbar:\cS(\delbar,\Lambda) \to \Q^+$.
	The space $\cS(\delbar,\Lambda)$ has an associated subgroupoid structure $(\cS(\hdelbar,\hat{\Lambda}), \bm{\cS}(\hdelbar,\hat{\Lambda}))$ with (saturated) object set
		\[
		S(\hdelbar,\hat{\Lambda}) :=\{x\in Z \mid \hat{\Lambda} ( \hdelbar(x))>0\} \subset Z
		\]
	and with morphism set given by 
		\[
		\bm{S} (\hdelbar, \hat{\Lambda}) := \{ \phi \in \bm{Z}	\mid s(\phi) \in S(\hdelbar,\hat{\Lambda})	\} \subset	\bm{Z}
		\]
	(we could equivalently require that $t(\phi)\in S(\hdelbar, \hat{\Lambda})$).
	It is equipped with a \textbf{weight functor} $\hat{\Lambda}\circ \hdelbar:(\cS(\hdelbar,\hat{\Lambda}), \bm{\cS}(\hdelbar,\hat{\Lambda})) \to \Q^+$.
\end{definition}

Note that the space $\cS(\hdelbar,\hat{\Lambda})$ or the subgroupoid $(\cS(\hdelbar,\hat{\Lambda}), \bm{\cS}(\hdelbar,\hat{\Lambda}))$ can be respectively encoded entirely via the weight function or the weight functor; such a description is closer to the language used in \cite{HWZ3} and \cite{HWZbook}.

\subsubsection{Transverse perturbations}

At a local level, it is easy to adapt the functional analytic construction of compact perturbations of Fredholm operators to M-polyfolds; the implicit function theorem for M-polyfolds \cite[Thm.~3.13]{HWZbook} then guarantees that the zero set of a transversal $\ssc$-Fredholm section has the structure of a finite-dimensional manifold.
It is somewhat more involved to adapt these constructions to the global level, as this requires using multisections to obtain equivariance.

\begin{definition}[{\cite[Def.~15.2]{HWZbook}}]
	\label{def:transversal-pair}
	Let $P:\cW\to \cZ$ be a strong polyfold bundle, $\delbar$ a $\ssc$-smooth Fredholm section, and $\Lambda$ a $\ssc^+$-multisection.
	
	Consider a point $[x]\in \cZ$.  We say $(\delbar,\Lambda)$ is \textbf{transversal at $[x]$} if, given a local $\ssc^+$-section structure for $\Lambda$ at a representative $x$, the linearized local expression 
		\[D(\hdelbar-s_i)(x):T_x Z \to W_x\]
	is surjective for all $i\in I$ with $\hdelbar(x)=s_i(x)$.  We say that $(\delbar,\Lambda)$ is \textbf{transversal} if it is transversal at every $[x] \in \cS(\delbar,\Lambda)$.
	%why? $[x] \in \text{dom-supp}(\Lambda)$.
	
	%If $x\in Z$ with $\Lambda\circ \delbar(x)>0$ we say that $(\delbar,\Lambda)$ is \textbf{transversal at $x$} provided $\mathsf{T}_{(\delbar,\Lambda)}(x)$ is a nonzero combination of surjective sc-Fredholm operators.
	
	%We shall say that $\mathsf{T}_{(f,\Lambda)}(x)$ is \textbf{onto} or \textbf{surjective} at a point $x$  in $\supp(\Lambda\circ f)$ provided every occurring sc-Fredholm operator with positive weight is onto.
\end{definition}

Given a $\ssc$-Fredholm section it is relatively easy to construct a transversal multisection (see the general position argument of \cite[Thm.~15.4]{HWZbook}); a key ingredient is \cite[Lem.~5.3]{HWZbook} which guarantees the existence of locally defined $\ssc^+$-sections which take on a prescribed value at a point.

\begin{theorem}[{\cite[Thm.~4.13]{HWZ3}}\footnote{The original statement of this theorem carries the additional requirement that the perturbed solution set $\cS(\delbar,\Lambda)$ is a compact set.  This requirement is unnecessary, and is not used in the proof.}]
	\label{thm:transversal-pairs-weighted-branched-suborbifolds} \label{thm:transversality}
	%Let $P:\cW\rightarrow \cZ$ be a strong polyfold bundle, $\delbar$ a $\ssc$-smooth Fredholm section, and $\Lambda$ a $\ssc^+$-multisection.
	If the pair  $(\delbar,\Lambda)$ is transversal, then the perturbed solution set $\cS(\delbar,\Lambda)$ carries in a natural way the structure of a weighted branched suborbifold.
\end{theorem}

\begin{remark}[Relationship between local section structures and local branching structures]
	\label{rmk:relationship-local-section-structures-local-branching-structures}
	Consider a weighted branched suborbifold $\cS(\delbar,\Lambda)$ defined by a transversal $\ssc$-smooth Fredholm section $\delbar$ and a $\ssc^+$-multisection $\Lambda$.
	The relationship between the local section structure for $\Lambda$ and the local branching structure can be described as follows.
	Consider a point $[x]\in\cS(\delbar,\Lambda)$, and let $U\subset Z$ be a $\bm{G}(x)$-invariant open neighborhood of a representative $x$.	
	Consider a local section structure for $\Lambda$ at $[x]$ consisting of $\ssc^+$-sections $s_i  : U \to W|_U$ and weights $w_i$ for $i\in I$.
	The implicit function theorem for M-polyfolds then implies that the sets
		\[
		M_i = (\hdelbar -s_i)^{-1}(0)
		\]
	define finite dimensional submanifolds which together with the weights $w_i$ give a local branching structure in $U$.
\end{remark}

\subsubsection{Pairs which control compactness}

Given a proper $\ssc$-smooth Fredholm section $\delbar$ and a $\ssc^+$-multisection $\Lambda$, we need some way to control the compactness of the resulting perturbed solution space $\cS(\delbar,\Lambda)$. This can be achieved by requiring that the perturbation $\Lambda$ is ``small'' in a suitable sense.

\begin{definition}[{\cite[Def.~12.2]{HWZbook}}]
	\label{def:auxiliary-norm}
	Let $P:\cW\to \cZ$ be a strong polyfold bundle. %(with $\hat{P}:W\to Z$ an associated strong polyfold bundle structure).	
	We define an \textbf{auxiliary norm} as a $\ssc^0$-map 
	\[N:\cW[1] \to [0,\infty)\]
	where we regard $[0,\infty)$ as a smooth manifold with the trivial ep-groupoid structure (i.e., a polyfold with finite-dimensional local models and trivial isotropy).
	It has an associated $\ssc^0$-functor $\hat{N}:W[1]\to [0,\infty)$ where as usual $\abs{\hat{N}}$ induces $N$.
	We require that $\hat{N}$ satisfies the following conditions.
	\begin{enumerate}
		\item The restriction of $\hat{N}$ to each fiber $W_x[1]$ is a complete norm. (Recall that for each $x\in Z$, the fiber $W_x[1]$ is a $\ssc$-Banach space.)
		\item \label{property-2-auxiliary-norm} If $\{h_k\}$ is a sequence in $W[1]$ such that $\{\hat{P}(h_k)\}$ converges in $Z$ to an object $x$, and if $\lim_{k\to \infty} \hat{N}(h_k) = 0$, then $\{h_k\}$ converges to $0_x \in W_x[1]$.
	\end{enumerate}
\end{definition}

%Let $\hat{P}:\cW\to \cZ$ be a strong polyfold bundle (with $P:W\to Z$ an associated strong polyfold bundle structure), $\hat{\Lambda}$ be a $\ssc^+$-multisection (with $\Lambda$ an associated $\ssc^+$-multisection functor), and $\hat{N}$ be an auxiliary norm (with $N$ an associated auxiliary norm functor).

Given a point $[x]\in\cZ$ we define the \textbf{pointwise norm} of $\Lambda$ at $[x]$ with respect to the auxiliary norm $N$ by
	\[
	N[\Lambda] ([x]) := \max \{	N([w])	\mid	[w]\in \cW[1], \Lambda ([w])>0, P([w])=[x]	\}
	\]
and moreover define the \textbf{norm} of $\Lambda$ with respect to $N$ by
	\[
	N[\Lambda] := \sup_{[x]\in\cZ} N[\Lambda] ([x]).
	\]

\begin{definition}[{\cite[Def.~15.4]{HWZbook}}]
	\label{def:pair-which-controls-compactness}
	Let $P:\cW\to \cZ$ be a strong polyfold bundle, let $\delbar$ be a $\ssc$-smooth proper Fredholm section, and let $N :\cW[1]\to [0,\infty)$ be an auxiliary norm.	
	Consider an open neighborhood $\cU$ of the unperturbed solution set $\cS(\delbar)\subset \cZ$. %(i.e., $\cU$ is open considered as a set of the underlying topological space $\cZ$)
	We say that the pair $(N,\cU)$ \textbf{controls the compactness} of $\delbar$ if the set 
	\[
	cl_\cZ \{[x]\in \cU \mid \delbar ([x]) \in \cW[1], N(\delbar([x]))\leq 1\} \subset \cZ
	\]
	is compact.
\end{definition}

\begin{remark}
	\label{rmk:shrink-neighborhood}
	We may always shrink the controlling neighborhood $\cU$ of the unperturbed solution set. To be precise, suppose that $(N, \cU)$ is a pair which controls compactness, and let $\cU'$ be an open set such that $\cS(\delbar)\subset\cU'\subset \cU$.  It is immediate from the above definition that the pair $(N, \cU')$ also controls compactness.
\end{remark}

Given a $\ssc$-smooth proper Fredholm section of a strong polyfold bundle, \cite[Prop.~2.27]{HWZ3} guarantees the existence of auxiliary norms.  The existence of a pair which control compactness then follows from \cite[Thm.~4.5]{HWZ3} which states that given an auxiliary norm $N$ there always exists an associated neighborhood $\cU$, such that the pair $(N, \cU)$ controls compactness.

\begin{theorem}[{\cite[Lem.~4.16]{HWZ3}}]
	\label{thm:compactness}
	Let $P:\cW\to \cZ$ be a strong polyfold bundle, let $\delbar$ be a $\ssc$-smooth proper Fredholm section, and let $(N, \cU)$ be a pair which controls compactness.
	
	Consider a $\ssc^+$-multisection $\Lambda$ and suppose it satisfies the following:
		\begin{itemize}
			\item $N[\Lambda] \leq 1$,
			\item $\text{dom-supp} (\Lambda) \subset \cU$.
		\end{itemize}
	Then the perturbed solution set $\cS(\delbar,\Lambda)$ is compact.
	We call such a $\ssc^+$-multisection \textbf{$(N,\cU)$-admissible} (compare with \cite[Def.~15.5]{HWZbook}).
\end{theorem}

\subsubsection{Determinant line bundles and orientations}

%The topic of orientations is intimidating, to say the least.
We do not try to give a full account of the polyfold theory on orientations (for that, we refer to \cite[\S~6]{HWZbook}).
However, to talk precisely about orientations in our main theorems it is necessary to give a brief summary of the main ideas and definitions.

\begin{definition}[{\cite[Defs.~6.3,~6.4]{HWZbook}}]
	Let $T: E \to F$ be a bounded linear Fredholm operator between real Banach spaces. The \textbf{determinant} of $T$ is the 1-dimensional real vector space
		\[\det T = \Lambda^{\max} (\ker T) \otimes \left(\Lambda^{\max} (\coker T)	\right)^*.\]
	An \textbf{orientation} of $T$ is a choice of orientation of the real line $\det T$.
\end{definition}

Let $P: W\to Z$ be a strong M-polyfold bundle, and let $\hdelbar: Z \to W$ be a $\ssc$-smooth Fredholm section.
In general, there is no intrinsic notion a linearization of the section $\hdelbar$ at smooth points $x \in Z_\infty$ if $\hdelbar(x) \neq 0$.
To deal with this, one chooses a locally defined $\ssc^+$-section $s$ such that $s(x)=\hdelbar(x)$; one may then consider the well-defined linearization $D(\hdelbar - s)(x) :T_xZ \to W_x$.
The \textbf{space of linearizations} of $\hdelbar$ at $x$ is then defined as the following subset of linear Fredholm operators from $T_x Z \to W_x:$
	\[\Lin(\hdelbar,x) := \{	D(\hdelbar - s)(x) + a \mid a: T_x Z \to W_x \text{ is a }\ssc^+\text{-operator}	\}.\]
It may be observed that $\Lin(\hdelbar,x)$ is a convex subset, and hence is contractible.

To each linearization we may associate its determinant; in doing so, we may consider the disjoint union
	\[\DET(\hdelbar,x) := \bigsqcup_{L\in \Lin(\hdelbar,x)} \{L\}\times \det(L).\]
A priori, this set does not have much structure, as although each determinant is a real line, locally the kernel and cokernel of the linearizations may vary in dimension.
However, with some work it is possible to prove the following proposition.

\begin{proposition}[{\cite[Prop.~6.11]{HWZbook}}]
	The set $\DET(\hdelbar,x)$ has the structure of a topological line bundle over $\Lin(\hdelbar,x)$.
	The base space $\Lin(\hdelbar,x)$ is contractible and hence $\DET(\hdelbar,x)$ has two possible orientations.
\end{proposition}

We may therefore define an \textbf{orientation of $\hdelbar$ at a smooth point $x\in Z_\infty$} as a choice of one of the two possible orientations for $\DET(\hdelbar,x)$. We denote such an orientation by $o_{(\hdelbar,x)}$.

As we vary the smooth points, we need some way to compare the orientations at each point.
Intuitively, the choice of an orientation at a point should automatically determine an orientation at all nearby points.
This intuition is made precise in the theory as a sort of ``local orientation propagation.''

\begin{theorem}[{\cite[Thm.~6.1]{HWZbook}}]
	Consider a smooth point $x\in Z_\infty$. There exists an open neighborhood $U\subset Z$ such that for any smooth point $y\in U$ and for any $\ssc$-smooth path $\phi: [0,1] \to Z$ with $\phi(0)=x,$ $\phi(1)=y$ there exists a well-defined ``local orientation propagation.'' This means that, given an orientation $o_{(\hdelbar,x)}$ of $\DET(\hdelbar,x)$ we can associate an orientation $\phi_* o_{(\hdelbar,x)}$ of $\DET(\hdelbar,y)$, and moreover this association does not depend on the choice of $\ssc$-smooth path.
\end{theorem}

We may therefore define an orientation of a Fredholm section as a fixed choice of orientation at all smooth points which is consistent with the local orientation propagation.

\begin{definition}[{\cite[Def.~6.11]{HWZbook}}]
	\label{def:oriented-Fredholm}
	Let $P: W\to Z$ be a strong M-polyfold bundle, and let $\hdelbar: Z \to W$ be a $\ssc$-smooth Fredholm section.
	We define an \textbf{orientation} of $\hdelbar$ as an association for every smooth point $x\in Z_\infty$ with an orientation $o_{(\hdelbar,x)}$ of the determinant $\DET(\hdelbar,x)$ and which is consistent with the local orientation propagation in the following sense.
	
	For any two smooth points $x, y \in Z_\infty$ and for any $\ssc$-smooth path $\phi:[0,1] \to Z$ with $\phi(0)=x,$ $\phi(1)=y$ the orientation $o_{(\hdelbar,y)}$ is the same as the pushforward orientation $\phi_* o_{(\hdelbar,x)}$ determined by the local orientation propagation. (Compare with \cite[Defs.~6.12,~6.13]{HWZbook}.)
\end{definition}

We end with an observation regarding how the above abstract discussion induces orientations on the perturbed solution spaces.
Consider an oriented $\ssc$-smooth Fredholm section $\hdelbar:Z\to W$ and consider a $\ssc^+$-section locally defined on a neighborhood $U$ of a point $x\in Z$, $s: U \to W|_U$.
Suppose that $\hdelbar(x) = s(x)$ and suppose that the linearization 
	\[D(\hdelbar - s)(x) :T_x Z \to W_x\]
is surjective. The implicit function theorem for M-polyfolds implies that $M:= (\hdelbar - s)^{-1}(0)$ has the structure of a finite-dimensional manifold.

A choice of orientation $o_{(\hdelbar,x)}$ of $\DET(\hdelbar,x)$ determines for any linearization $T\in \Lin (\hdelbar,x)$ a choice of orientation of $\det T$.
Then simply observe that $D(\hdelbar - s)(x) \in \Lin(\hdelbar,x)$, and since
	\[\det ((D(\hdelbar - s)(x)) = \Lambda^{\max} (\ker (D(\hdelbar - s)(x))) = \Lambda^{\max} (T_x M)\]
a choice of orientation for $\det ((D(\hdelbar - s)(x))$ automatically induces an orientation for $M$ at $x$.

\subsubsection{Regular perturbations and compact cobordism}

In order to define invariants, a perturbed solution set needs to be both transversally cut out, and compact.
We therefore introduce the following definition, given also in \cite[Cor.~15.1]{HWZbook}.

\begin{definition}
	\label{def:regular-perturbation}
	Let $P:\cW\to \cZ$ be a strong polyfold bundle, let $\delbar$ be a $\ssc$-smooth proper Fredholm section, and let $(N, \cU)$ be a pair which controls compactness.
	
	Suppose a $\ssc^+$-multisection $\Lambda$ satisfies both the requirements of Theorem~\ref{thm:transversality} and Theorem~\ref{thm:compactness}, i.e.,
	\begin{itemize}
		\item $(\delbar, \Lambda)$ is a transversal pair,
		\item $N[\Lambda] \leq 1$ and $\text{dom-supp} (\Lambda) \subset \cU$.
	\end{itemize}
	We then say $\Lambda$ is a \textbf{regular perturbation} of $\delbar$ with respect to the pair $(N,\cU)$.
\end{definition}

\begin{corollary}[{\cite[Cor.~15.1]{HWZbook}}]
	\label{prop:existence-regular-perturbations}
	%Let $P:\cW\to \cZ$ be a strong polyfold bundle, let $\delbar$ be a $\ssc$-smooth proper Fredholm section, and let $(N, \cU)$ be a pair which controls compactness.  
	
	There exist regular perturbations $\Lambda$ of $\delbar$ with respect to the pair $(N,\cU)$.	
	Theorems~\ref{thm:transversality} and \ref{thm:compactness} immediately imply that the perturbed solution space $\cS(\delbar,\Lambda)$ has the structure of a compact weighted branched suborbifold, with weight function given by $\Lambda\circ \delbar : \cS(\delbar,\Lambda) \to \Q^+$.
\end{corollary}

Compact weighted branched suborbifolds are suitable geometric spaces for defining invariants.
However, it remains to show that such invariants are independent of the choices used to define such a compact weighted branched orbifold, in particular, are independent of:
	\begin{itemize}
		\item the choice of regular perturbation,
		\item the choice of pair which controls compactness.
	\end{itemize}

\begin{theorem}[{\cite[Cor.~15.1]{HWZbook}}]
	\label{thm:cobordism-between-regular-perturbations}
	Let $P:\cW\to \cZ$ be a strong polyfold bundle, let $\delbar$ be a $\ssc$-smooth proper oriented Fredholm section, and let $(N_0, \cU_0)$, $(N_1,\cU_1)$ be two pairs which control compactness.  Suppose that $\Lambda_0$ is a regular perturbation of $\delbar$ with respect to the pair $(N_0,\cU_0)$, and likewise $\Lambda_1$ is a regular perturbation of $\delbar$ with respect to the pair $(N_1,\cU_1)$.  Consider the strong polyfold bundle $[0,1]\times \cW \to [0,1]\times \cZ$ and the $\ssc$-smooth proper oriented Fredholm section $\tdelbar$ defined by $(t,[z]) \mapsto (t,\delbar([z]))$.
	
	Then there exists a pair $(N,\cU)$ which controls the compactness of $\tdelbar$ and which satisfies the following:
	\begin{enumerate}
		\item the auxiliary norm $N:[0,1]\times \cW \to \Q^+$ restricts to $N_0$ on $\{0\}\times \cW$ and restricts to $N_1$ on $\{1\}\times \cW$,
		\item the open neighborhood $\cU$ of $\cS (\tdelbar)$ satisfies $\cU \cap (\{0\}\times \cZ )= \cU_0$ and $\cU \cap (\{1\}\times \cZ )= \cU_1$.
	\end{enumerate}
	
	In addition, there exists a regular perturbation $\tilde{\Lambda}$ of $\tdelbar$ with respect to the pair $(N,\cU)$, such that $\tilde{\Lambda}|_{\{0\}\times \cW}$ can be identified with $\Lambda_0$ and likewise $\tilde{\Lambda}|_{\{1\}\times \cW}$ can be identified with $\Lambda_1$.
	
	It follows that the perturbed solution set $\cS (\tdelbar, \tilde{\Lambda})$ has the structure of a compact weighted branched suborbifold, and is a cobordism between perturbed solution sets, in the sense that
		\[
		\partial \cS (\tdelbar, \tilde{\Lambda}) = -\cS (\delbar,\Lambda_0) \sqcup \cS (\delbar,\Lambda_1).
		\]
\end{theorem}

\subsection{The branched integral and polyfold invariants}
	\label{subsec:branched-integral-polyfold-invariants}

We now describe how to define the polyfold invariants through the use of the branched integral. The definition of the branched integral theory on compact oriented weighted branched suborbifolds was originally developed in \cite{HWZint}.

\begin{definition}[{\cite[Def.~4.9]{HWZbook}}]
	Let $\cZ$ be a polyfold with an associated polyfold structure $(Z,\bm{Z})$.
	The vector space of $\ssc$-differential $k$-forms $\Omega^k (Z)$ is the set of $\ssc$-smooth maps 
	\[\ww:\bigoplus^k_{n=1} TZ\rightarrow \R\]
	defined on the Whitney sum of the tangent of the object space, which are linear in each argument and skew-symmetric.
	Moreover, we require that the maps $\ww$ are morphism invariant in the following sense: for every morphism $\phi: x\to y$ in $\bm{Z}_1$ with tangent map $T\phi:T_xZ\rightarrow T_yZ$ we require that
	\[
	(T\phi)^*\ww_y=\ww_x.
	\]
\end{definition}

Recall the definition of $\cZ^i$ as the shifted polyfold with shifted polyfold structure $(Z^i,\bm{Z}^i)$.
Via the inclusion maps $\cZ^i \hookrightarrow \cZ$ we may pullback a $\ssc$-differential $k$-form $\ww$ in $\Omega^k(Z)$ to $\Omega^k(Z^i)$, obtaining a directed system
\[
\Omega^k(Z) \to \cdots \to \Omega^k(Z^i) \to \Omega^k(Z^{i+1}) \to \cdots,
\]
we denote by $\Omega^k_\infty (Z)$ the direct limit of this system.
As defined in \cite[p.~149]{HWZbook} there exists an \textbf{exterior derivative}
\[
d:\Omega^*(Z^{i+1}) \to \Omega^{* +1}(Z^i)
\]
such that the composition $d\circ d = 0$.
The exterior derivative commutes with the inclusion maps $Z^i \hookrightarrow Z^{i+1}$ and hence induces a map
\[
d:\Omega^*_\infty(Z) \to \Omega^{* +1}_\infty(Z)
\]
which also satisfies $d\circ d =0$.

\begin{theorem}[{\cite[Thm.~9.2]{HWZbook}}]
	\label{def:branched-integral}
	Let $\cZ$ be a polyfold with polyfold structure $(Z,\bm{Z})$ which admits $\ssc$-smooth partitions of unity.
	Given a $\ssc$-smooth differential form $\ww\in \Omega^n_\infty (Z)$ and an $n$-dimensional compact oriented weighted branched suborbifold $\cS\subset \cZ$.
	
	Then there exists a well-defined \textbf{branched integral}, denoted as
	\[\int_{\cS} \ww,\]
	which is partially characterized by the following property. 
	Consider a point $[x]\in \cS$ and a representative $x\in S$ with isotropy group $\bm{G}(x)$. Let $(M_i)_{i\in I}$, $(w_i)_{i\in I}$, $(o_i)_{i\in I}$ be a local branching structure at $x$ contained in a $\bm{G}(x)$-invariant open neighborhood $U\subset Z$ of $x$.
	Consider a $\ssc$-smooth differential form $\ww\in \Omega^n_\infty (Z)$ and suppose that $\abs{\supp \ww} \subset \abs{U}$.
	Then
	\[
	\int_{\cS} \ww = \frac{1}{\sharp \bm{G}^\text{eff}(x)} \left( \sum_{i\in I} w_i \cdot \int_{(M_i,o_i)} \ww\right)
	\]
	where $\sharp \bm{G}^\text{eff}(x)$ is the order of the effective isotropy group and $\int_{(M_i,o_i)} \ww$ is the usual integration of the differential $n$-form $\ww$ on the oriented $n$-dimensional manifold $M_i$.
\end{theorem}

\begin{theorem}[Stokes' theorem, {\cite[Thm.~9.4]{HWZbook}}]
	\label{thm:stokes}
	Let $\cZ$ be a polyfold with polyfold structure $(Z,\bm{Z})$ which admits $\ssc$-smooth partitions of unity.
	Let $\cS$ be an $n$-dimensional compact oriented weighted branched suborbifold, and let $\partial \cS$ be its boundary with induced weights and orientation.  Consider a $\ssc$-differential form $\omega\in \Omega^{n-1}_\infty (Z)$.
	Then 
	\[
	\int_{\cS} d\omega   = \int_{\partial \cS} \omega.
	\]
\end{theorem}

The next theorem follows the same reasoning used to prove \cite[Thm.~11.8]{HWZbook}.
\begin{theorem}[Change of variables]
	\label{thm:change-of-variables}
	%Let $\cZ_1$ and $\cZ_2$ be polyfolds, with associated polyfold structures $Z_1$ and $Z_2$.
	%Suppose that $f:\cZ_1\to \cZ_2$ is a $\ssc$-smooth map with an associated $\ssc$-smooth functor $\hat{f}:Z_1\to Z_2$.
	Let $\cS_i\subset \cZ_i$ be $n$-dimensional compact oriented weighted branched suborbifolds with weight functions $\vartheta_i:\cS_i \to \Q^+$ for $i=1,2$.  
	Let $(S_i,\bm{S_i})$ be the associated branched suborbifold structures with associated weight functors $\hat{\vartheta}_i: (S_i,\bm{S_i}) \to \Q^+$ for $i=1,2$.
	
	Let $g:\cZ_1 \to \cZ_2$ be a $\ssc$-smooth map between polyfolds, which has a well-defined restriction $g|_{\cS_1}	: \cS_1 \to \cS_2$ between the branched suborbifolds. In addition, assume the following:
	\begin{itemize}
		\item $g: \cS_1 \to \cS_2$ is a homeomorphism between the underlying topological spaces,
		\item $\hat{g}: S_1\to S_2$ is injective and an orientation preserving local homeomorphism,
		\item $g$ is weight preserving, i.e., $\vartheta_2\circ g=\vartheta_1$ and $\hat{\vartheta}_2 \circ \hat{g}=\hat{\vartheta}_1$.
	\end{itemize}
	
	Then given a $\ssc$-smooth differential form $\ww \in \Omega^n_\infty (Z_2)$,
	\[
	\int_{\cS_2} \ww = \int_{\cS_1} g^* \ww.
	\]
\end{theorem}

\begin{theorem}[Polyfold invariants as branched integrals, {\cite[Cor.~15.2]{HWZbook}}]
	Consider a $\ssc$-smooth map 
	\[
	f:\cZ \to \cO
	\]
	from a polyfold $\cZ$ to an orbifold $\cO$.
	We may define the \textbf{polyfold invariant} as the homomorphism obtained by pulling back a de Rahm cohomology class from the orbifold and taking the branched integral over a perturbed zero set:
	\[
	H^*_{\dR} (O) 	\to \R, \qquad \ww \mapsto \int_{\cS(p)} f^*\ww.
	\]
	By Theorem~\ref{thm:cobordism-between-regular-perturbations} and by Stokes' theorem~\ref{thm:stokes}, this homomorphism does not depend on the choice of abstract perturbation used to obtain the compact oriented weighted branched suborbifold $\cS(p)$.
\end{theorem}

%%%%%%%%%%%%%%%%%%%%%%%%%%%%%%%%%%%%%%%%%%%%%%%%%%
% NATURALITY OF POLYFOLD INVARIANTS %%%%%%%%%%%%%%
%%%%%%%%%%%%%%%%%%%%%%%%%%%%%%%%%%%%%%%%%%%%%%%%%%

\section{Naturality of polyfold invariants}
	\label{sec:naturality-polyfold-invariants}

In this section we establish the necessary theory for proving the naturality of polyfold invariants, culminating in Theorem~\ref{thm:naturality-polyfold-invariants} and in Corollaries~\ref{cor:naturality-polyfold-gw-invariants} and \ref{cor:punctures-equal}.

\subsection{Invariance of domain and branched suborbifolds}
	\label{subsec:invariance-of-domain}

In the process of considering the naturality of the polyfold invariants, we will encounter a smooth bijection between weighted branched suborbifolds,
	\[
	f: \cS_1 \to \cS_2,
	\]
where $\dim \cS_1 = \dim \cS_2$ and $\cS_2$ is a compact topological space.
We would like to show that this map is a homeomorphism.
%by Change of Variables \red{XXX} it would then follow that the branched integrals over $\cS_1$ and $\cS_2$ are equal. In particular, we need to show this map is a homeomorphism} \red{What do we actually need to show??}

However using only knowledge of the topologies of these spaces, it is impossible to show this.
The key to resolving this problem is understanding the branched suborbifold structure and how to use this additional structure to prove an invariance of domain result.
This result will allow us to assert that the above map is a homeomorphism.

Invariance of domain is a classical theorem of algebraic topology due to Brouwer, and was originally published in 1911.
\begin{theorem}[Invariance of domain, {\cite{brouwer1911beweis}}]
	\label{thm:invariance-of-domain}
	Let $U$ be an open subset of $\R^n$, and let $f: U\to \R^n$ be an injective continuous map.  Then $f$ is a homeomorphism between $U$ and $f(U)$.
\end{theorem}

This result can immediately be generalized to manifolds; let $M$ and $N$ be an $n$-dimensional manifolds and let $f: M\to N$ be an injective continuous map.  Then $f$ is a homeomorphism onto its image.  Moreover, if $f$ is bijective, it is a homeomorphism.  We seek to generalize this result to the branched suborbifolds of our current situation.

\subsubsection{Local topology of branched submanifolds}

%We begin with a study of the local topology of weighted branched submanifolds considered as subsets of an ambient M-polyfold.

As a starting definition, a \emph{branched manifold} is a topological space which is locally homeomorphic to a finite union of open subsets of $\R^n$.
However, such a broad definition of a branched manifold immediately raises the possibility of non-desirable topological properties. Consider the classic example of the \emph{line with two origins}---although this is a locally Euclidean and second-countable topological space, it is not Hausdorff.
In contrast, the branched submanifolds we study are embedded into open subsets of ambient M-polyfolds and have better behaved topologies.

\begin{lemma}
	\label{lem:topology-of-local-branching-structures}
	Let $U$ be a metrizable topological space.
	Let $M_i$, $i\in I$ be a finite collection of finite-dimensional manifolds together with inclusion maps $\phi_i: M_i \hookrightarrow U$.
	Assume moreover that each $\phi_i$ is proper and a topological embedding.
	
	Consider the set defined by the image of the inclusions, $\cup_{i \in I} \phi_i (M_i)$.
	There are two topologies we may consider on this set:
	\begin{itemize}
		\item $(\cup_{i \in I} \phi_i(M_i), \tau_s)$, where $\tau_s$ is the subspace topology induced from $U$
		\item $(\cup_{i \in I} \phi_i(M_i), \tau_q)$, where $\tau_q$ is the quotient topology induced by the map $\sqcup_{i\in I} \phi_i : \sqcup_{i\in I} M_i \to \cup_{i \in I} \phi(M_i)$.
	\end{itemize}
	These two topologies are identical.
\end{lemma}
\begin{proof}
	We show that $\tau_s = \tau_q$.
	\begin{itemize}
		\item $\tau_s \subset \tau_q$
	\end{itemize}
	Consider the following commutative diagram where $q$ is the quotient map, $\phi_i$ are the continuous inclusion maps $\phi_i: M_i \to U$, and $i$ is inclusion map.
	\begin{center}
		\begin{tikzcd}
		\bigsqcup_{i\in I} M_i \arrow[r, "\sqcup_{i\in I}\phi_i"] \arrow[d, "q"'] & U \\
		(\bigcup_{i\in I} \phi_i(M_i),\tau_q) \arrow[ru, "i",hook] \arrow[r, "\id"] & (\bigcup_{i \in I} \phi_i(M_i), \tau_s) \arrow[u, "i"', hook]
		\end{tikzcd}
	\end{center}
	Then by the characteristic property of the quotient topology, $\sqcup_{i\in I} \phi_i$ continuous implies $i:(\cup_{i \in I} \phi_i(M_i), \tau_q) \allowbreak \hookrightarrow \allowbreak U $ continuous.
	By the definition of the subspace topology, $i:(\cup_{i \in I} \phi_i(M_i), \tau_q) \hookrightarrow U $ is continuous.
	By the characteristic property of the subspace topology, $i:(\cup_{i \in I} \phi_i(M_i), \tau_q) \hookrightarrow U $ continuous implies %$i:(\cup_{i \in I}\phi_i(M_i), \tau_s) \hookrightarrow U $ wow what a typo huh
	$\id : (\cup_{i \in I} \phi_i(M_i), \tau_q) \to (\cup_{i \in I} \phi_i(M_i), \tau_s)$
	is continuous.
	
	\begin{itemize}
		\item $\tau_q \subset \tau_s$
	\end{itemize}
	By assumption $U$ is a metrizable space; hence it is also a regular topological space.
	The assumption that each $\phi_i$ is a topological embedding and is proper implies moreover that the images $\phi_i(M_i)\subset U$ are closed in the subspace topology; to see this note that in metric spaces, sequential compactness is equivalent to compactness, and then use properness.
	
	Suppose $V\subset \cup_{i \in I} \phi_i(M_i)$ and $V\in \tau_q$.  
	We will show for every $x\in V$ there exists a subset $B \subset \cup_{i \in I} \phi_i(M_i)$ such that $B \in \tau_s$ and $x \in B \subset V$.  This implies that $V\in\tau_s$, as desired.
	
	By the definition of the quotient topology, the set $q^{-1} (V) \subset \sqcup_{i\in I} M_i$ is open and hence $q^{-1} (V) \cap M_i$ is open in the topology on $M_i$.	Consider $x$ as a point in $U$ via the set inclusion $\cup_{i\in I} \phi_i(M_i) \subset U$, since $\phi_i : M_i\to U$ is an injection it follows that $q^{-1} (x) = \{x_{i_1}, \ldots, x_{i_k} \}$ where $x_{i_l} \in M_{i_l}$ for a nonempty subset $\{i_1,\ldots , i_k \} \subset I$.
	
	Let $B_\epsilon (x)\subset U$ be an $\epsilon$-ball at $x$.
	Since $\phi_{i_l}$ is a topological embedding it follows that the sets $\phi^{-1}_{i_l} (B_\epsilon(x))$ give a neighborhood basis for $M_{i_l}$ at the point $x_{i_l}$.
	Therefore, we may take $\epsilon$ small enough that $\phi^{-1}_{i_l} (B_\epsilon(x)) \subset q^{-1} (V) \cap M_{i_l}$ for all $i_l \in \{i_1,\ldots , i_k \}$.	
	Since $U$ is a regular topological space, and since $x$ and $\phi_j(M_j)$ are disjoint closed subsets of $U$ for $j \in I \setminus \{i_1,\ldots , i_k \}$, we can find disjoint open neighborhoods that separate $x$ and $\phi_j(M_j)$.
	This moreover implies that we may take $\epsilon$ small enough that $\phi^{-1}_j (B_\epsilon (x)) = \emptyset$ for all $j \in I \setminus \{i_1,\ldots , i_k \}$.
	For such an $\epsilon$, it follows that $\phi_i^{-1} (B_\ep(x)) \subset q^{-1}(V) \cap M_i$ for all $i\in I$.
		
	The desired set is then given by
		\[
		B:= B_\epsilon (x) \cap \bigcup_{i \in I} \phi_i (M_i);
		\]
	it is an open set in the subspace topology on $\cup_{i \in I} \phi_i(M_i)$.
	By construction, $q^{-1}(B) = \sqcup_{i\in I} \phi_i^{-1} (B_\ep(x)) \subset \sqcup_{i\in I} q^{-1}(V) \cap M_i = q^{-1} (V)$, therefore $B\subset V$ as desired.
	%To see this, let $y\in B$, then there exists $\hat{y}\in \bigsqcup_{i\in I} M_i$ such that $q(\hat{y})=y$ and $\hat{y} \in \phi^{-1}_{i_l} (B_\epsilon(x)) \subset q^{-1} (V) \cap M_{i_l}$, hence $y \in V$. 
	
	%By construction, $\sqcup f^{-1}_i(B) \subset q^{-1}(V) $. and as required $B \subset V$.  Write first thing tomorrow.
	
	%Very important!  Uses the following facts - inclusion maps of submanifolds are proper, which implies they have closed image (true for metric spaces, where sequential compactness is equivalent to compactness).  Then use some basic facts about the topology of M-polyfolds (in particular, these are regular spaces.)
\end{proof}

An open subset of an M-polyfold with the subspace topology is a metrizable topological space, and hence the above lemma applies to the branched suborbifolds of Definition~\ref{def:weighed-branched-suborbifold}

\begin{lemma}
	\label{lem:local-homeo-m-polyfold}
	Let $S_i$ be $n$-dimensional branched submanifolds of M-polyfolds $Z_i$ for $i=1,2$.
	Consider an injective continuous map between these two M-polyfolds,
	$\hat{f}: Z_1 \hookrightarrow Z_2,$
	and suppose that there is a well-defined restriction to the branched submanifolds,
	$\hat{f}|_{S_1}:S_1 \hookrightarrow S_2.$
	
	For every $x\in S_1$ with $y:= \hat{f}(x) \in S_2$, suppose that there exist local branching structures $(M_i)_{i\in I}$ at $x$ and $(M'_j)_{j\in I}$ at $y$ which have the same index set $I$. 
	Moreover, assume that $\hat{f}$ has a well-defined restriction to the individual local branches for each index $i\in I$ as follows:
	\[
	\hat{f}|_{M_i} : M_i \hookrightarrow M'_i.
	\]
	Then $\hat{f}|_{S_1}$ is a local homeomorphism between $S_1$ and $S_2$. Since we have assumed that $\hat{f}$ is injective, it follows that $\hat{f}|_{S_1}$ is also a homeomorphism onto its image.
\end{lemma}
\begin{proof}
	Let $x\in S_1$ which maps to $\hat{f}(x)\in S_2$.  By assumption, there exists a local branching structure $(M_i)_{i\in I}$ in a neighborhood $O_x$ of $x$, and there exists a local branching structure $(M'_j)_{j\in J}$ in a neighborhood $O_{\hat{f}(x)}$ of $\hat{f}(x)$ such that the index sets are the same, $I=J$, and $\hat{f}$ restricts to a injective continuous map between each branch, i.e.,
	\[\hat{f}|_{M_i} : M_i \to M'_i.\]
	
	We may invoke invariance of domain \ref{thm:invariance-of-domain} to see that the restricted maps $\hat{f}|_{M_i}$ are homeomorphisms onto their images.
	Observe that the open balls $B_\epsilon(\hat{f}(x))\subset O_{\hat{f}(x)}$ give a neighborhood basis for $M_i'$ at $\hat{f}(x)$ for all $i\in I$.  It follows that $\hat{f}^{-1} (B_\epsilon (\hat{f}(x))) \subset O_x$ give a neighborhood basis for $M_i$ at $x$ for all $i\in I$.  For $\epsilon$ small enough, the restricted maps 
	\begin{equation}\label{eq:restriction-to-local-branches}
	\hat{f}|_{M_i \cap \hat{f}^{-1} (B_\epsilon (\hat{f}(x)))} : M_i \cap \hat{f}^{-1} (B_\epsilon (\hat{f}(x))) \to M_i' \cap B_\epsilon(\hat{f}(x))
	\end{equation}
	are homeomorphisms for all $i\in I$.
	
	Define a neighborhood of $x$ by $U_x = \hat{f}^{-1} (B_\epsilon (\hat{f}(x)))$; then $N_i := M_i \cap \hat{f}^{-1} (B_\epsilon(\hat{f}(x)))$ give local branches in $U_x$.  Define a neighborhood of $\hat{f}(x)$ by $U_{\hat{f}(x)}:=B_\epsilon(\hat{f}(x))$; then $N_i' := M_i' \cap B_\epsilon(\hat{f}(x))$ give local branches in $U_{\hat{f}(x)}$.  We can now rewrite \eqref{eq:restriction-to-local-branches} more simply as
	\[
	\hat{f}|_{N_i} : N_i \to N'_i.
	\]
	and note again that the maps $\hat{f}|_{N_i}$ are homeomorphisms for all $i\in I$.  Hence the map $\sqcup_{i\in I} (\hat{f}|_{N_i}): \sqcup_{i\in I} N_i \to \sqcup_{i\in I} N'_i$ is also a homeomorphism.
	
	Consider the following commutative diagram of maps.
	\begin{center}
		\begin{tikzcd}
		\bigsqcup_{i\in I} N_i \arrow[d, "q"] \arrow[r, "\sqcup (\hat{f}|_{N_i})"] & \bigsqcup_{i\in I} N'_i \arrow[d, "q'"] \\
		(\cup_{i\in I} N_i, \tau_q) \arrow[r, "\hat{f}|_{\cup N_i}"', hook] & (\cup_{i\in I} N'_i, \tau_{q'})
		\end{tikzcd}
	\end{center}
	We assert that the map $\hat{f}|_{\cup N_i} : (\cup_{i \in I} N_i,\tau_q) \hookrightarrow (\cup_{i\in I} N'_i,\tau_{q'})$ is a homeomorphism.
	Indeed, by assumption $\hat{f}|_{\cup N_i}$ is injective.  We can use the fact that $\sqcup (\hat{f}|_{N_i})$ is a bijection to see that $\hat{f}|_{\cup N_i}$ must also be surjective.
	It is easy to check that $\hat{f}|_{\cup N_i}$ is continuous with respect to the quotient topologies $\tau_q$ and $\tau_{q'}$.
	Furthermore, $\hat{f}|_{\cup N_i}$ is an open map.
	To see this, let $U\subset (\cup_{i \in I} N_i,\tau_q)$ be an open set.  Then $q^{-1} (U) \subset \sqcup_{i\in I} N_i$ is open by the definition of the quotient topology.
	Since $\sqcup (\hat{f}|_{N_i})$ is a homeomorphism, $(\sqcup (\hat{f}|_{N_i})) (q^{-1}(U))$ is open.
	Commutativity of the diagram and the fact that both $\sqcup (\hat{f}|_{N_i})$ and $\hat{f}|_{\cup N_i}$ are bijections implies that $(\sqcup \hat{f}|_{N_i}) (q^{-1}(U)) = q'^{-1} (\hat{f}|_{\cup N_i} (U))$.
	It therefore follows that $\hat{f}|_{\cup N_I} (U)$ is open by the definition of the quotient topology.
	
	By Lemma~\ref{lem:topology-of-local-branching-structures}, the fact that $\hat{f}|_{\cup N_i} : (\cup_{i \in I} N_i,\tau_q) \hookrightarrow (\cup_{i\in I} N'_i,\tau_{q'})$ is a homeomorphism implies that $\hat{f}|_{\cup N_i} : (\cup_{i \in I} N_i,\tau_s) \hookrightarrow (\cup_{i\in I} N'_i,\tau_s)$ is a homeomorphism.
	Note that $\cup_{i \in I} N_i \subset S_1$ and $\cup_{i \in I} N'_i \subset S_2$ are both open subsets.
	By Remark~\ref{rmk:local-topology-subgroupoid}, the inclusion maps $(\cup_{i \in I} N_i,\tau_s)\hookrightarrow S_1$ and $(\cup_{i\in I} N'_i,\tau_s)\hookrightarrow S_2$ are both local homeomorphisms.  We now see that the map $\hat{f} : S_1 \to S_2$ is a local homeomorphism on an open neighborhood of the point $x\in S_1$.  Since $x\in S_1$ was arbitrary, and since $\hat{f}$ is injective, we can conclude $\hat{f}$, considered on the object sets, is a local homeomorphism.  It then follows from the \'etale property that $\hat{f}$, considered on the morphism sets, is a local homeomorphism.  This proves the claim.
\end{proof}

\begin{lemma}
	\label{lem:invariance-of-domain-branched-orbifolds}
	Let $\cS_i$ be an $n$-dimensional branched suborbifold of a polyfold $\cZ_i$ for $i=1,2$.
	Consider an injective continuous map between these two polyfolds, $f: \cZ_1 \hookrightarrow \cZ_2$,
	and which has an associated functor $\hat{f}: (Z_1,\bm{Z_1}) \hookrightarrow (Z_2,\bm{Z_2})$, which is injective and continuous with respect to the object and morphism sets.
	In addition, assume that the functor $\hat{f}$ is fully faithful.
	Suppose that $f$ has a well-defined restriction to the branched suborbifolds $f|_{\cS_1}:\cS_1 \hookrightarrow \cS_2$; it follows that $\hat{f}$ restricts to a well-defined functor between the subgroupoids $\hat{f}|_{S_1} : (S_1,\bm{S}_1) \to (S_2,\bm{S}_2)$.
	
	Assume that for every $x\in S_1$ with $y:= \hat{f}(x) \in S_2$, there exist local branching structures $M_i$, $i\in I$ at $x$ and $M'_j$, $j\in I$ at $y$ which have the same index set $I$. 
	Moreover, assume that $\hat{f}$ has a well-defined restriction to the individual local branches for each index $i\in I$ as follows:
	\[
	\hat{f}|_{M_i} : M_i \hookrightarrow M'_i.
	\]
	
	Then the restriction $f|_{\cS_1}:\cS_1 \hookrightarrow \cS_2$ is a local homeomorphism. In particular, if $f|_{\cS_1}$ is a bijection, then it is a homeomorphism.
\end{lemma}
\begin{proof}	
	Let $[x] \in \cS_1$ and let $f([x]) \in \cS_2$.  Let $x$ be a representative of $[x]$, hence $\hat{f}(x)$ is a representative of $f([x])$.  
	From the proof of Lemma~\ref{lem:local-homeo-m-polyfold}, we have seen that there exists a local branching structure $(N_i)_{i\in I}$ at $x$ and a local branching structure $(N'_i)_{i\in I}$ at $\hat{f}(x)$ such that $\hat{f}|_{\cup N_i} : (\cup_{i \in I} N_i,\tau_s) \hookrightarrow (\cup_{i\in I} N'_i,\tau_s)$ is a homeomorphism.
	
	The proof now follows the same reasoning as Lemma~\ref{lem:local-homeo-m-polyfold}.
	Consider the following commutative diagram of maps.
	\begin{center}
		\begin{tikzcd}
		\bigcup_{i\in I} N_i \arrow[d, "q"] \arrow[r, "\hat{f}|_{\cup N_i}"] & \bigcup_{i\in I} N'_i \arrow[d, "q'"] \\
		(\abs{\cup_{i\in I} N_i},\tau_q)  \arrow[r, "f|_{\abs{\cup N_i}}"', hook] & (\abs{\cup_{i\in I} N'_i}, \tau_{q'})
		\end{tikzcd}
	\end{center}
	We assert that the map $f|_{\abs{\cup N_i}}$ is a homeomorphism.
	Indeed, by assumption $f|_{\abs{\cup N_i}}$ is injective.  We can use the fact that $\hat{f}|_{\cup N_i}$ is a bijection to see that $f|_{\abs{\cup N_i}}$ must also be surjective.
	By assumption, $f$ is continuous and therefore the restriction $f|_{\abs{\cup N_i}}$ is continuous.
	Furthermore, $f|_{\abs{\cup N_i}}$ is an open map.
	To see this, let $U\subset \abs{\cup_{i\in I} N_i}$ be an open set.
	Then $q^{-1}(U)\subset \cup_{i \in I} N_i$ is open by the definition of the quotient topology.
	Since $\hat{f}|_{\cup N_i}$ is a homeomorphism,  $(\hat{f}|_{\cup N_i})(q^{-1}(U)) \subset \cup_{i \in I} N'_i$ is open.  Commutativity of the diagram and the fact that both $\hat{f}|_{\cup N_i}$ and $f|_{\abs{\cup N_i}}$ are bijections implies that $(\hat{f}|_{\cup N_i}) (q^{-1}(U)) = q'^{-1} (f|_{\abs{\cup N_i}}(U))$.
	It therefore follows that $f|_{\abs{\cup N_i}}(U)$ is open by the definition of the quotient topology.
	
	Proposition~\ref{prop:natural-representation-subgroupoid} implies that the inclusion maps $\abs{\cup_{i\in I} N_i} \hookrightarrow \cS_1$ and $\abs{\cup_{i\in I} N'_i} \hookrightarrow \cS_2$ are local homeomorphisms.
	We now see that the map $f|_{\cS_1} : \cS_1 \to \cS_2$ is a local homeomorphism on an open neighborhood of the point $[x]\in \cS_1$.  Since $[x]\in \cS_1$ was arbitrary it follows that $f|_{\cS_1}$ is a local homeomorphism.
	It moreover follows that if $f|_{\cS_1}$ is bijective, it is a homeomorphism.  This proves the claim.
\end{proof}

\subsection{Fredholm multisections and abstract perturbations}
	\label{subsec:fredholm-multisections}

In this subsection we generalize the polyfold abstract perturbation theory from Fredholm sections to Fredholm multisections.
This involves minor modifications to the definitions and theorems originally developed in \cite{HWZ3} and which we recalled in \S~\ref{subsec:abstract-perturbations}.
This generalization is developed with a specific goal in mind, which is the proof of Theorem~\ref{thm:naturality}.

\begin{definition}
	Let $\cW\to \cZ$ be a strong polyfold bundle.  We define a \textbf{$\ssc$-smooth Fredholm multisection} as 
	\begin{enumerate}
		\item a function $F:\cW \to \Q^+$,
		\item an associated functor $\hat{F}: W \to \Q^+$ where $\abs{\hat{F}}$ induces $F$,
	\end{enumerate}
	such that at ever $[x]\in \cZ$ there exists a \textbf{local Fredholm section structure} defined as follows.
	Let $x\in Z$ be a representative of $[x]$ and let $U\subset Z$ be a $\bm{G}(x)$-invariant open neighborhood of $x$, and consider the restricted strong M-polyfold bundle $P: W|_U \to U$.
	Then there exist finitely many $\ssc$-Fredholm sections $f_1,\ldots,f_k : U \to W|_U$ %(see \cite[Definition 3.8]{HWZbook}) 
	with associated positive rational numbers $\sigma_1,\ldots ,\sigma_k \in \Q^+$ which satisfy the following:
	\begin{enumerate}
		\item $\sum_{i=1}^k \sigma_i =1.$
		\item The restriction $\hat{F}|_{W|_U}: W|_U \to \Q^+$ is related to the local sections and weights via the equation 
			\[
			\hat{F}|_{W|_U}(w)=\sum_{\{i\in \{1,\ldots, k\} \mid w=f_i(p(w))\}} \sigma_i
			\]
		where the empty sum has by definition the value $0$.
	\end{enumerate}
\end{definition}

We say that the Fredholm multisection $F$ is \textbf{proper} if the unperturbed solution set
	\[
	\cS (F) := \{ [z]\in \cZ \mid F(0_{[x]})	>0	\} \subset \cZ
	\]
is a compact topological space.
(Notice that the condition $F(0_{[x]})>0$ is equivalent to the condition that $f_i(x) = 0$ for some $i \in I$ for a given representative $x$ and a local Fredholm section structure $(f_i)_{i\in I}$, $(\sigma_i)_{i\in I}$ at $x$.)
Furthermore, we can define a weight function on the unperturbed solution set, $\cS(F) \to \Q^+$, by $[z] \mapsto F(0_{[x]})$.

\begin{example}
	For the applications we have in mind, the $\ssc$-smooth Fredholm multisections are obtained as a pair $(\delbar,\Lambda)$ consisting of:
	\begin{itemize}
		\item a $\ssc$-smooth Fredholm section $\delbar:\cZ\to\cW$,
		\item a $\ssc^+$-multisection $\Lambda :\cW \to \Q^+$.
	\end{itemize}
	Given a point $[x]\in\cZ$, we define a local Fredholm section structure for $(\delbar,\Lambda)$ at $[x]$ as follows.
	Let $x\in Z$ be a representative of $[x]$ and let $U\subset Z$ be a $\bm{G}(x)$-invariant open neighborhood of $x$, and consider the restricted strong M-polyfold bundle $P: W|_U \to U$.
	Consider the $\ssc$-smooth Fredholm section $\hdelbar: U \to W|_U$, and let $(s_i)_{i\in I}$, $(\sigma_i)_{i\in I}$ be a local section structure for $\Lambda$ at $x$.
	
	Then the local Fredholm section structure is given by $f_i := \hdelbar - s_i$ with associated weight $\sigma_i$. It follows from \cite[Thm.~3.2]{HWZbook} that such an $f_i$ is in fact a $\ssc$-smooth Fredholm section.
	We may then define the functor $\hat{F}$ locally via the equation
		\[
		\hat{F}|_{W|_U}(w)=\sum_{i\in \{1,\ldots, k \mid w=f_i(p(w))\}} \sigma_i
		\]
	where the empty sum has by definition the value $0$. It is evident this extends to a well-defined functor $\hat{F}: (W,\bm{W}) \to \Q^+$.
	Finally, observe the perturbed solution set $\cS (\delbar,\Lambda)$ associated to the pair $(\delbar,\Lambda)$ is the same as the unperturbed solution set $\cS (F)$ associated to the Fredholm multisection $F$, i.e.,
		\[
		\{ [z]\in \cZ \mid \Lambda(\delbar([x]))	>0	\} = \{ [z]\in \cZ \mid F(0_{[x]})	>0	\}.
		\]
\end{example}

\subsubsection{Transverse perturbations of Fredholm multisections}

We can immediately adapt the main definitions and results of \S~\ref{subsec:abstract-perturbations}; there is no difficulty in generalizing the construction of transverse perturbations to Fredholm multisections.

\begin{definition}
	Associated to a $\ssc$-smooth Fredholm multisection $\delbar$ and a $\ssc^+$-mul\-ti\-sec\-tion $\Gamma$, 
	we define the \textbf{perturbed solution space} as the set
	\[
	\cS(F,\Lambda) :=\{[z]\in\cZ \mid (F\oplus\Gamma) (0_{[z]})	>0	\}\subset \cZ
	\]
	with topology given by the subspace topology induced from $\cZ$. It is equipped with the weight function $\cS(F,\Gamma) \to \Q^+,$ $[z]\mapsto (F\oplus\Gamma) (0_{[z]}).$
\end{definition}

Along the same lines as Definition~\ref{def:transversal-pair}, we can formulate what it means for a Fredholm multisection and a $\ssc^+$-multisection to be transversal.

\begin{definition}
	\label{def:fredholm-multisection-transversal}
	Let $P:\cW\to \cZ$ be a strong polyfold bundle, $F$ a $\ssc$-smooth Fredholm multisection, and $\Gamma$ a $\ssc^+$-multisection.
	
	Consider a point $[x]\in \cZ$.  We say $(F,\Gamma)$ is \textbf{transversal at $[x]$} if, given a local Fredholm section structure for $F$ at $[x]$ and given a local $\ssc^+$-section structure for $\Gamma$ at $[x]$, then the linearized local expression 
		\[
		D(f_i-s_j)(x):T_x Z \to W_x
		\]
	is surjective for all $i\in I$, $j\in J$ with $f_i(x)=s_j(x)$.  We say that $(F,\Gamma)$ is \textbf{transversal} if it is transversal at every $[x] \in \cS(F,\Gamma)$.
\end{definition}

Consider our example of a Fredholm multisection $(\delbar,\Lambda)$ consisting of a Fredholm section and a $\ssc^+$-multisection $\Lambda$, and let $\Gamma$ be an additional $\ssc^+$-multisection. Then the sum $\Lambda \oplus \Gamma: \cW \to \Q^+$ is a $\ssc^+$-multisection, with local section structure given by $s_i +r_j$ where $(s_i)_{i\in I}$ is a local section structure for $\Lambda$ and $(r_j)_{j\in J}$ is a local section structure for $\Gamma$.
We may now observe that the pair $(\delbar, \Lambda\oplus \Gamma)$ consisting of the Fredholm section $\delbar$ and the $\ssc^+$-multisection $\Lambda\oplus \Gamma$ is transversal in the sense of Definition~\ref{def:transversal-pair} if an only if the pair $((\delbar,\Lambda), \Gamma)$ consisting of the Fredholm multisection $(\delbar,\Lambda)$ and the $\ssc^+$-multisection $\Gamma$ is transversal in the sense of the above Definition~\ref{def:fredholm-multisection-transversal}.

We have an analog of Theorem~\ref{thm:transversal-pairs-weighted-branched-suborbifolds}.

\begin{proposition}
	\label{prop:fredholm-multisection-transversal-pairs-weighted-branched-suborbifolds}
	Let $P:\cW\rightarrow \cZ$ be a strong polyfold bundle, $F$ a $\ssc$-smooth Fredholm multisection, and $\Gamma$ a $\ssc^+$-multisection.
	If the pair  $(F,\Lambda)$ is transversal, then the perturbed solution set $\cS(F,\Gamma)$ carries in a natural way the structure of a weighted branched suborbifold.
\end{proposition}

\subsubsection{Controlling compactness of Fredholm multisections}

In contrast to construction of transverse perturbations of Fredholm multisections, where no modification of the underlying definitions or ideas was required, it is somewhat more involved to show how to control the compactness of Fredholm multisections.
It is necessary to refer to the earlier work contained in \cite[\S~4.2]{HWZ3} in order to obtain complete results in our current situation.

\begin{definition}
	Consider a Fredholm multisection and a point $[z] \in \cZ$.
	Let $(f_i)_{i\in I}$ be a local section structure for $F$ at a representative $z$.
	Let $N:\cW[1]\to [0,\infty)$ be an auxiliary norm with associated $\ssc^0$-functor $\hat{N}:W[1]\to [0,\infty)$; as in \cite[p.~434]{HWZbook}, we may extend $N$ to all of $\cW$ by defining $N([w]):= +\infty$ for $[w]\in \cW[0] \setminus \cW[1]$, and likewise extend $\hat{N}$ to all of $W$.
	We define the \textbf{min norm of the Fredholm multisection} $F$ at $[z]$ by the equation
	\[
	N_{\min} (F) [z] := \min_{i\in I} \{	\hat{N} (	f_i(z)	)	\}.
	\]
	% N_{\min} (\delbar,\Lambda) [z] := \min_{i\in I} \{	\hat{N} (	(\hdelbar-s_i)(z)	)	\}
\end{definition}

\begin{definition}
	Let $P:\cW\to \cZ$ be a strong polyfold bundle, let $F$ be a $\ssc$-smooth proper Fredholm multisection, and let $N :\cW\to [0,\infty)$ be an extended auxiliary norm.
	
	Consider an open neighborhood $\cU$ of the unperturbed solution set $\mathcal{S}(F)\subset \cZ$.
	%(i.e., $\cU$ is open considered as a set of the underlying topological space $\cZ$).
	We say that the pair $(N,\cU)$ \textbf{controls the compactness} of $F$ provided the set 
	\[
	cl_\cZ \{[x]\in \cU \mid N_{\min} (F) [x]\leq 1\} \subset \cZ
	\]
	is compact.
	% cl_\cZ \{[x]\in \cU \mid \hdelbar ([x]) \in \cW[1], N_{\min} (\delbar,\Lambda) [x]\leq 1\} \subset \cZ
\end{definition}

\begin{proposition}[Analog of {\cite[Thm.~4.5]{HWZ3}}]
	Let $P:\cW\to \cZ$ be a strong polyfold bundle, let $F$ be a $\ssc$-smooth proper Fredholm multisection, and let $N :\cW[1]\to [0,\infty)$ be an auxiliary norm.  Then there exists an open neighborhood $\cU$ of the unperturbed solution set $\mathcal{S}(F)$ such that the pair $(N,\cU)$ controls the compactness of $F$.
\end{proposition}

\begin{proposition}[Analog of {\cite[Lem.~4.16]{HWZ3}}]
	\label{prop:fredholm-multisection-compactness}
	Let $P:\cW\to \cZ$ be a strong polyfold bundle, let $F$ be a $\ssc$-smooth proper Fredholm multisection, and let $(N,\cU)$ be a pair which controls compactness.
	If a $\ssc^+$-multisection $\Gamma$ satisfies $N[\Gamma] \leq 1$ and $\text{dom-supp} (\Gamma) \subset \cU$, then the perturbed solution set $\mathcal{S}(F,\Gamma)$ is compact.%(considered with the subspace topology induced from the underlying topological space $\cZ$).
\end{proposition}

\subsubsection{Regular perturbations and compact cobordism}

Let $P:\cW\to \cZ$ be a strong polyfold bundle, let $F$ be a $\ssc$-smooth proper Fredholm multisection, and let $(N, \cU)$ be a pair which controls compactness.
We say a $\ssc^+$-multisection $\Gamma$ is a \textbf{regular perturbation} of $F$ with respect to the pair $(N,\cU)$ if it satisfies the following:
	\begin{itemize}
		\item $(F, \Gamma)$ is a transversal pair,
		\item $N[\Gamma] \leq 1$ and $\text{dom-supp} (\Gamma) \subset \cU$.
	\end{itemize}
As in \cite[Cor.~15.1]{HWZbook}, one can prove that there exist regular perturbations $\Gamma$ of $F$ with respect to the pair $(N,\cU)$.
It follows from Proposition~\ref{prop:fredholm-multisection-transversal-pairs-weighted-branched-suborbifolds} and Proposition~\ref{prop:fredholm-multisection-compactness} that the perturbed solution space $\cS(F,\Gamma)$ has the structure of a compact weighted branched suborbifold, with weight function given by $\cS(F,\Gamma) \to \Q^+,$ $[z]\mapsto (F\oplus \Gamma)(0_{[z]})$.

Furthermore, as in \cite[Cor.~15.1]{HWZbook} one can prove the existence of a compact cobordism between perturbed solution sets of regular perturbations.

\subsubsection{Cobordism from a transversal Fredholm multisection to a regular perturbation}
	\label{subsubsec:cobordism-multisection-regular}

Having developed the above generalization to Fredholm multisections, we are finally in a position to state the desired specialized result, Proposition~\ref{prop:cobordism-multisection-regular}.

Consider a strong polyfold bundle $P:\cW \to \cZ$ and a $\ssc$-smooth proper Fredholm section $\delbar$.
Suppose that $(N_0,\cU_0)$ is a pair which controls the compactness of $\delbar$.
Consider a $\ssc^+$-multisection $\Lambda$ and suppose that $(\delbar,\Lambda)$ is a transversal pair.
(Note that we do not assume that $\Lambda$ is admissible to a pair which controls compactness.)
Now, consider the strong polyfold bundle $\cW\times[0,1] \to \cZ\times[0,1]$, and consider a $\ssc$-smooth Fredholm multisection $(\tdelbar,\tilde{\Lambda})$ defined as follows:
	\begin{itemize}
		\item $\tdelbar$ is the $\ssc$-smooth Fredholm section defined by $([z],s)\mapsto (\delbar([z]),s)$,
		\item $\tilde{\Lambda}$ is the $\ssc^+$-multisection defined for $s\neq 0$ by
		$([w],s)\mapsto \Lambda(1/s \cdot [w])$ and for $s=0$ by
			\[
			([w],0)	\mapsto 
			\begin{cases}
				1, &\text{if } [w]=[0], \\
				0, &\text{if } [w]\neq [0],
			\end{cases}
			\]
		and
		whose local section structure at an object $(x,s)$ is defined by $O_x\times [0,1] \to W \times [0,1]; (x,s) \mapsto (s\cdot s_i(x), s)$ (where $(s_i)$ is the original local section structure for $\Lambda$ at the object $x\in Z$).
	\end{itemize}
Moreover, let us assume that the Fredholm multisection $(\tdelbar,\tilde{\Lambda})$ is proper, i.e., the solution set $\cS(\tdelbar,\tilde{\Lambda})$ is compact.

Observe that the topological boundary of $\cS(\tdelbar,\tilde{\Lambda})$ is given by the following set:
	\[
	\partial \cS(\tdelbar,\tilde{\Lambda}) = \cS(\delbar) \sqcup \cS(\delbar,\Lambda).
	\]
By the assumption that $(\delbar, \Lambda)$ is a transversal pair $\cS(\delbar,\Lambda)$ is a weighted branched orbifold; moreover it is a closed subset of $\cS(\tdelbar, \tilde{\Lambda})$ and is therefore compact.
We emphasize that since $\Lambda$ is not admissible to a pair which controls compactness, it is not a regular perturbation (see Definition~\ref{def:regular-perturbation}) and hence cannot be used to define polyfold invariants.
%In addition, since $\delbar$ is not assumed to be transverse, $\cS(\delbar)$ does not have the structure of a weighted branched suborbifold. 
We can almost consider $\cS(\tdelbar,\tilde{\Lambda})$ as a compact cobordism, except $\delbar$ is not assumed to be transverse and hence $\cS(\delbar)$ is not assumed to have the structure of a weighted branched suborbifold.

The following proposition demonstrates how to perturb the solution space $\cS(\tdelbar,\tilde{\Lambda})$ in order to obtain a compact cobordism between $\cS(\delbar,\Lambda)$ and a perturbed solution space $\cS(\delbar,\Gamma_0)$ where $\Gamma_0$ is a \emph{regular} perturbation.

\begin{proposition}
	\label{prop:cobordism-multisection-regular}
    Suppose that $\Gamma_0$ is a regular perturbation of $\delbar$ with respect to the pair $(N_0,\cU_0)$,
	There exists a pair $(N,\cU)$ which controls the compactness of the Fredholm multisection $(\tdelbar,\tilde{\Lambda})$ and which satisfies the following:
	\begin{itemize}
		\item the auxiliary norm $N: \cW\times [0,1] \to \Q^+$ restricts to $N_0$ on $\cW\times \{0\}$,
		\item the open neighborhood $\cU$ of $\cS(\tdelbar,\tilde{\Lambda})$ satisfies $\cU \cap (\cZ\times \{0\}) = \cU_0$.
	\end{itemize}
	Moreover, there exists a regular perturbation $\Gamma$ of $(\tdelbar,\tilde{\Lambda})$ with respect to the pair $(N,\cU)$ such that $\Gamma|_{\cW\times \{0\}}$ can be identified with $\Gamma_0$ and such that $\Gamma|_{\cW\times\{1\}} \equiv 0$.
\end{proposition}

The proof of this proposition follows the same reasoning used to prove Theorem~\ref{thm:cobordism-between-regular-perturbations}, noting in addition that  we do not need to perturb in a neighborhood of $\cZ\times\{1\}$, as by assumption $(\delbar,\Lambda)$ is a transversal pair.

\subsection{Intermediary subbundles and naturality of polyfold invariants}
	\label{subsec:intermediary-subbundles-naturality}

Consider a commutative diagram as follows,
\begin{equation}\label{eq:commutative-diagram-naturality}
	\begin{tikzcd}
	\cW_1 \arrow[r, "\iota_\cW"', hook] \arrow[d, "\delbar_1\quad "'] & \cW_2 \arrow[d, "\quad \delbar_2"] &  \\
	\cZ_1 \arrow[r, "\iota_\cZ"', hook] \arrow[u, bend left] & \cZ_2 \arrow[u, bend right] & 
	\end{tikzcd}
\end{equation}
where:
\begin{itemize}
	\item $\cW_i \to \cZ_i$ are strong polyfold bundles for $i=1,2$.
	\item $\delbar_i$ are $\ssc$-smooth proper oriented Fredholm sections of the same index for $i=1,2$.
	\item $\iota_\cZ :\cZ_1 \hookrightarrow \cZ_2$ is a $\ssc$-smooth injective map, and the associated functor between polyfold structures $\hat{\iota}_\cZ : (Z_1,\bm{Z}_1) \hookrightarrow (Z_2,\bm{Z}_2)$ is fully faithful and is also an injection on both the object and the morphism sets.
	\item $\iota_\cW:\cW_1\hookrightarrow\cW_2$ is a $\ssc$-smooth injective map, and the associated functor between polyfold strong bundle structures $\hat{\iota}_\cW :(W_1,\bm{W}_1) \hookrightarrow (W_2,\bm{W}_2)$ is fully faithful, and is also an injection on both the object and the morphism sets.  Moreover, $\hat{\iota}$ is a bundle map (i.e., restricts to a linear map on the fibers).
	\item $\cS(\delbar_2) \subset \text{Im} (\iota_\cZ)$.
\end{itemize}

In order to deal with orientations, consider the following.
Consider a smooth object $x\in (Z_1)_\infty$ which maps to $y:= \hat{\iota}_\cZ \in (Z_2)_\infty$.
Consider a locally defined $\ssc^+$-section $s:' U \to W_2$ defined on an open neighborhood $U\subset Z_2$ of $y$, which satisfies $s'(y) = \hdelbar_2 (y)$.
Assume that this $\ssc^+$-section has a well-defined restriction $s'|_{U \cap \hat{\iota}_\cZ (Z_1)} : U \cap \hat{\iota}_\cZ (Z_1) \to \hat{\iota}_\cW (W_1)$, which induces a $\ssc^+$-section $s: \hat{\iota}_\cZ^{-1}(U) \to W_1$ which moreover satisfies $s(x) = \hdelbar_1 (x)$.
We therefore have a commutative diagram.
	\[
	\begin{tikzcd}
	T_x W_1 \arrow[r, "D\hat{\iota}_\cW"'] & T_y W_2 &  \\
	T_x Z_1 \arrow[r, "D\hat{\iota}_\cZ"'] \arrow[u, "D(\hdelbar_1-s)(x)"] & T_y Z_2 \arrow[u, "D(\hdelbar_2-s')(y)"'] & 
	\end{tikzcd}
	\]
Consider the following maps: $D\hat{\iota}_\cZ: \ker (D(\hdelbar_1-s)(x))) \to \ker (D(\hdelbar_2-s')(y))$, and $D\hat{\iota}_\cW: \text{Im} (D(\hdelbar_1-s)(x))) \to \text{Im} (D(\hdelbar_2-s')(y))$, which therefore induces a map $\coker (D(\hdelbar_1-s)(x))) \to \coker (D(\hdelbar_2-s')(y))$.
These maps induce a map between the determinant real lines
	\begin{gather*}
	\det (D(\hdelbar_1-s)(x))) = \Lambda^{\max} (\ker (D(\hdelbar_1-s)(x)))) \otimes (\Lambda^{\max} (\coker (D(\hdelbar_1-s)(x)))))^*,	\\
	\det (D(\hdelbar_2-s')(y))) = \Lambda^{\max} (\ker (D(\hdelbar_2-s')(y))) \otimes (\Lambda^{\max} (\coker (D(\hdelbar_2-s')(y)))	)^*	.
	\end{gather*}

\begin{itemize}
	\item	Assume that the induced map between the determinants
				\[\hat{\iota}_* : \det (D(\hdelbar_1-s)(x))) \to \det (D(\hdelbar_2-s')(y))\]
			is an isomorphism. Moreover, assume that this isomorphism is orientation preserving, with respect to the chosen orientations of $\hdelbar_1$ at the point $x$ and $\hdelbar_2$ at the point $y$ (see Definition~\ref{def:oriented-Fredholm}).
\end{itemize}

Returning to the main discussion, it follows from commutativity of \eqref{eq:commutative-diagram-naturality} that $\iota_\cZ$ restricts to a continuous bijection between the unperturbed solution sets,
	\[\iota_\cZ |_{\cS(\delbar_1)} : \cS(\delbar_1)\to \cS(\delbar_2).\]
In fact, this map is a homeomorphism as can be shown via point-set topology, noting that $\cS(\delbar_1)$ is compact and $\cS(\delbar_2)$ is Hausdorff (see \cite[Rmk.~3.1.15]{MWtopology}).

In order to compare the polyfold invariants, suppose we also have a commutative diagram
	\begin{equation}\label{eq:gw-invariant-pair-of-polyfolds}
	\begin{tikzcd}
	&  & \cO \\
	\cZ_1 \arrow[r, "\iota_\cZ"', hook] \arrow[rru, "f_1"] & \cZ_2 \arrow[ru, "f_2"'] & 
	\end{tikzcd}
	\end{equation}
where:
\begin{itemize}
	\item $\cO$ is a finite-dimensional orbifold.
	\item $f_i$ are $\ssc$-smooth maps for $i=1,2$.
\end{itemize}

\begin{definition}\label{def:intermediate-subbundle}
	We define an  \textbf{intermediary subbundle} as a subset $\cR \subset \cW_2$ which satisfies the following properties.
	\begin{enumerate}
		\item Let $(R,\bm{R})$ be the associated subgroupoid of $\cR$. Then for every object $x\in Z_2$ we require that the fiber $R_x : = R \cap (W_2)_x$ is a vector subspace of $(W_2)_x$. (Note that we do not require that $R_x$ is complete.)
		\item \label{property-2-intermediary-subbundle} For any point $[x] \in \cZ_2$, if $\delbar_2 ([x]) \in \cR$ then $[x] \in \iota_\cZ(\cZ_1)$. (Equivalently, for any object $x\in Z_2$, if $\hdelbar_2 (x) \in R$ then $x \in \hat{\iota}_\cZ (Z_1)$.)
		\item \label{property-3-intermediary-subbundle}
		Given $[x_0] \in \cS(\delbar_1) \simeq \cS(\delbar_2)$, let $V\subset U\subset Z_2$ be $\bm{G}(x_0)$-invariant open neighborhoods of a representative $x_0\in Z_2$ such that $\overline{V} \subset U$.
		We require that there exist $\ssc^+$-sections
			\[
			s'_i :U \to W_2, \qquad 1\leq i \leq k
			\]
		which have well-defined restrictions $s'_i|_{U\cap \hat{\iota}_\cZ (Z_1)} : U\cap \hat{\iota}_\cZ (Z_1) \to \hat{\iota}_\cW (W_1)$. These restrictions induce sections $s_i : \hat{\iota}_\cZ^{-1} (U) \to W_1$ which we require to be $\ssc^+$ with respect to the M-polyfold structures on $Z_1$ and $W_1$.
		We require that:
			\begin{itemize}
			\item $s'_i (U) \subset R$,
			\item $s'_i= 0$ on $U\setminus V$,
			\item $\text{span}\{s'_1(x_0),\ldots , s'_k(x_0)\} \oplus \text{Im}(D\hdelbar_2(x_0)) = (W_2)_{x_0},$
			\item $\text{span}\{s_1(x_0),\ldots , s_k(x_0)\} \oplus \text{Im}(D\hdelbar_1(x_0)) = (W_1)_{x_0}.$
			\end{itemize}
		\item In addition, given a pair $(N_2,\cU_2)$ which controls the compactness of $\delbar_2$, we require that these $\ssc^+$-sections satisfy the following:
			\begin{itemize}
			\item $\hat{N}_2[s'_i] \leq 1,$
			\item $\abs {\supp (s'_i)}\subset \cU_2$.
			\end{itemize}
		\begin{comment}
		\item \red{Alternative.}
		Given $[x] \in \cS(\delbar_2)$, let $U\subset Z_2$ be a $\bm{G}(x)$-invariant open neighborhood of a representative $x$.
		We require that there exists a parametrized $\ssc^+$-multisection
			\[
			\Lambda : \cW_2\times B^k_\ep \to \Q^+
			\]
		such that the composition $\Lambda \circ \iota_\cW : \cW_1 \times B^k_\ep \to \Q^+$ is also a well-defined $\ssc^+$-multisection.
		\end{comment}
	\end{enumerate}
\end{definition}

Despite the lengthy properties that a intermediary subbundle must satisfy, in practice such subbundles are easy to construct, as we demonstrate in \S~\ref{subsec:independence-sequence} and \S~\ref{subsec:independence-punctures}.
%The purpose of such an intricate definition is to give a precise checklist of what is necessary to establish the equality of the polyfold invariants for $\cZ_1$ and $\cZ_2$.

We may now prove Theorem~\ref{thm:naturality-polyfold-invariants}, which we restate in order to be consistent with our current notation.

\begin{theorem}
	\label{thm:naturality}
	Suppose there exists an intermediary subbundle $\cR \subset \cW_2$. Then the polyfold invariants for $\cZ_1$ and $\cZ_2$ defined via the branched integral are equal.
	This means that, given a de Rahm cohomology class $\ww\in H^*_{\dR} (\cO)$ the branched integrals over the perturbed solution spaces are equal,
		\[
		\int_{\cS (\delbar_1, p_1)} f_1^* \ww = \int_{\cS(\delbar_2,p_2)} f_2^* \ww,
		\]
	for any choices of regular perturbations.
\end{theorem}
\begin{proof}
We prove the theorem in six steps.

\begin{itemize}[leftmargin=0em]
	\item[]\textbf{Step 1:} \emph{We use property \ref{property-3-intermediary-subbundle} of the intermediary subbundle to construct a transversal $\ssc^+$-multisection with a well-defined transversal restriction.}
\end{itemize}

At the outset, fix pairs $(N_i,\cU_i)$ which control the compactness of $\delbar_i$ for $i=1,2$. Consider a point $[x_0]\in \cS(\delbar_1)\simeq \cS(\delbar_2)$ and let $x_0\in Z_1$ be a representative with isotropy group $\bm{G}(x_0)$. Via the inclusion map $\hat{\iota}_\cZ$, we may identify $x_0$ with its image in $Z_2$ and note that we may also identify the isotropy groups.

We may use property \ref{property-3-intermediary-subbundle} of the intermediary subbundle to construct an $\ssc^+$-mul\-ti\-sec\-tion functor $\hat{\Lambda}'_0:W_2 \times B_\ep^k \to \Q^+$ 
with local section structure given by $\left\lbrace g * \left(\sum_{i=1}^k t_i \cdot s'_i\right)\right\rbrace_{g\in \bm{G}(x_0)}$ which satisfies the following.
There exists a $\bm{G}(x_0)$-invariant open neighborhood $x_0 \subset U'_0 \subset Z_2$ such that at any object $x\in U'_0$ and for any $g\in \bm{G}(x_0)$ the linearization of the function 
	\begin{align*}
	U'_0 \times B_\ep^k 	&\to W_2\\
	(x, t_1,\ldots,t_k)		&\mapsto \hdelbar_2 (x) - g * \left(\sum_{i=1}^k t_i \cdot s'_i(x)\right)
	\end{align*}
projected to the fiber $(W_2)_x$ is surjective.

Furthermore, property \ref{property-3-intermediary-subbundle} ensures that the functor
$\hat{\Lambda}_0 :=\hat{\Lambda}_0' ( \hat{\iota}_\cW (\cdot), \cdot ) :W_1 \times B_\ep^k \to \Q^+$ is also a $\ssc^+$-multisection functor, with local section structure $\left\lbrace g * \left(\sum_{i=1}^k t_i \cdot s_i\right)\right\rbrace_{g\in \bm{G}(x_0)}$ where the $\ssc^+$-sections $s_i$ are induced by the well-defined restrictions of the sections $s'_i$. Likewise, there exists a $\bm{G}(x_0)$-invariant open neighborhood $x_0 \subset U_0 \subset Z_1$ such that at any object $x\in U_0$ and for any $g\in \bm{G}(x_0)$ the linearization of the function $\hdelbar_1(x) - g * \left(\sum_{i=1}^k t_i \cdot s_i(x)\right)$ projected to the fiber $(W_1)_x$ is surjective.

We may cover the compact topological space $\cS(\delbar_2)$ by a finite collection of such neighborhoods $\abs{U'_i}$ of points $[x_i]\in \cS(\delbar_2)$; we may also cover $\cS(\delbar_1)$ by a finite collection of such neighborhoods $\abs{U_i}$ of points $[x_i]\in \cS(\delbar_1)$.
It follows that the finite sum of $\ssc^+$-multisections
	\[
	\Lambda_2:= \bigoplus_i \Lambda'_i : \cW_2 \times B_\ep^N \to \Q^+
	\]
has the property that: for any point $[x] \in \cZ_2$ with $\Lambda_2 \circ \delbar_2 ([x])>0$, and for any parametrized local section structure $\{s'_i\}_{i\in I}$ at a representative $x$, the linearization of the function $\hdelbar_2 (x) - s'_i(x,t)$ projected to the fiber $(W_2)_x$ is surjective.
Likewise, the finite sum of $\ssc^+$-multisections
	\[
	\Lambda_1 := \bigoplus_i \Lambda_i = \bigoplus_i \Lambda'_i (\iota_\cW(\cdot),\cdot) : \cW_1 \times B_\ep^N \to \Q^+
	\]
has the property that for any point $[x] \in \cZ_1$ which satisfies $\Lambda_1 \circ \delbar_1 ([x])>0$ and for any parametrized local section structure $\{s_i\}_{i\in I}$ at a representative $x$, the linearization of the function $\hdelbar_1 (x) - s_i(x,t)$ projected to the fiber $(W_1)_x$ is surjective. Observe moreover that the multisection sum commutes with composition and thus $\Lambda_1(\cdot,\cdot) = \Lambda_2 (\iota_\cW(\cdot),\cdot)$.

Furthermore for $\ep$ sufficiently small, for any fixed $t_0 \in B_\ep^N$ the $\ssc^+$-multisection $\Lambda_2 (\cdot, t_0)$ is controlled by the pair $(N_2,\cU_2)$, i.e.,
\begin{itemize}
	\item $N_2[\Lambda_2(\cdot,t_0)] \leq 1$,
	\item $\text{dom-supp}(\Lambda_2(\cdot,t_0)) \subset \cU_2$.
\end{itemize}

In contrast, $\Lambda_1(\cdot,t_0)$ will generally not be controlled by the pair $(N_1,\cU_1)$, as in general, 	
	\[
	\text{dom-supp}(\Lambda_1(\cdot,t_0)) = \iota_\cZ^{-1} (\text{dom-supp}(\Lambda_2(\cdot,t_0))) \nsubseteq \cU_1.
	\]
%$\text{dom-supp}(\Lambda_1(\cdot,t_0)) = \iota_\cZ^{-1} (\text{dom-supp}(\Lambda_2(\cdot,t_0)))$ which will generally not be a subset of $\cU_1$.

\begin{itemize}[leftmargin=0em]
	\item[]\textbf{Step 2:} \emph{We show the thickened solution sets satisfy the hypotheses of Lemma~\ref{lem:invariance-of-domain-branched-orbifolds}, and are therefore homeomorphic.}
\end{itemize}

Consider the strong polyfold bundle $\cW_i \times B_\ep^N \to \cZ_i \times B_\ep^N$ for $i=1,2$, and let $\tdelbar_i : \cZ_i \times B_\ep^N \to \cW_i \times B_\ep^N$ denote the $\ssc$-smooth proper Fredholm section defined by $([z],t)\mapsto (\delbar_i([z]),t)$.
By construction, $(\tdelbar_i, \Lambda_i)$ are transversal pairs; hence by Theorem~\ref{thm:transversal-pairs-weighted-branched-suborbifolds} the thickened solution sets
	\[
	\cS(\tdelbar_i, \Lambda_i) = \{ ([z],t) \in \cZ_i \times B_\ep^N \mid \Lambda_i (\tdelbar_i ([z],t)) >0 \} \subset \cZ_i\times B_\ep^N
	\]
have the structure of weighted branched orbifolds.

We now claim that these thickened solution sets satisfy the hypotheses of Lemma~\ref{lem:invariance-of-domain-branched-orbifolds}.
Indeed, commutativity of the diagram \eqref{eq:commutative-diagram-naturality} together with the equation $\Lambda_1(\cdot,\cdot) = \Lambda_2 (\iota_\cW(\cdot),\cdot)$ imply that the injective continuous map $\tilde{\iota}_\cZ: \cZ_1 \times B_\ep^N \to \cZ_2 \times B_\ep^N; ([z],t)\mapsto (\iota_\cZ([z]),t)$ has a well-defined restriction to the thickened solution sets,
	\begin{equation}\label{eq:restriction-to-thickening}
	\tilde{\iota}_\cZ |_{\cS(\tdelbar_1, \Lambda_1)} : \cS(\tdelbar_1, \Lambda_1) \hookrightarrow \cS(\tdelbar_2, \Lambda_2).
	\end{equation}
Moreover, at any $(x,t)\in S_1(\hdelbar_1,\hat{\Lambda}_1)$ which maps to $(y,t)\in S_2(\hdelbar_2,\hat{\Lambda}_2)$, the local section structure $(s_i)$ for $\hat{\Lambda}_1$ at $(x,t)$ is induced by the restrictions of the local section structure $(s'_i)$ for $\hat{\Lambda}_2$ at $(y,t)$. In particular, we have the following commutative diagram.
	\[\begin{tikzcd}[column sep = large]
	W_1\times B_\ep^N \arrow[r, "{(\hat{\iota}_\cW(\cdot),\cdot)}"', hook] \arrow[d, "\hdelbar_1 - s_i \quad "'] & W_2\times B_\ep^N \arrow[d, "\quad \hdelbar_2 - s'_i"] &  \\
	O_x\times B_\ep^N \arrow[r, "{(\hat{\iota}_\cZ(\cdot),\cdot)}"', hook] \arrow[u, bend left] & O_y\times B_\ep^N \arrow[u, bend right] & 
	\end{tikzcd}\]
As noted in Remark~\ref{rmk:relationship-local-section-structures-local-branching-structures}, the local section structures and the local branching structures are related via the equations $M_i= (\hdelbar_1 - s_i)^{-1}(0)$, $M'_i = (\hdelbar_2 - s'_i)^{-1}(0)$. Thus it follow from commutativity that we have the required well-defined restriction to the individual local branches.
And now Lemma~\ref{lem:invariance-of-domain-branched-orbifolds} implies that the map \eqref{eq:restriction-to-thickening} is a local homeomorphism.

Furthermore, we may observe that by our orientation assumptions the natural induced map $\tilde{\iota}_* : \det (D(\hdelbar_1-s_i)(x))) \to \det (D(\hdelbar_2-s'_i)(y))$ is an orientation preserving isomorphism; hence the restriction $\tilde{\iota}_\cZ |_{M_i} : M_i \to M'_i$ is orientation preserving.

We now show that \eqref{eq:restriction-to-thickening} is a bijection.
Let $([y],t)\in \cS(\tdelbar_2,\Lambda_2)$, let $(y,t)$ be a representative of $([y],t)$, and consider a local section structure $(s'_i)$ for $\Lambda_2$ at $(y,t)$.
It follows that $\hdelbar_2(y) - s'_i (y,t) = 0$ for some index $i$.
Observe by construction, $s'_i$ is a finite sum of $\ssc^+$-sections with image contained in the intermediate subbundle $R$, and hence $s'_i(y,t)\in R$.
It follows that $\hdelbar_2 (y)\in R$, hence property \ref{property-2-intermediary-subbundle} of the intermediate subbundle implies that $y\in \hat{\iota}_\cZ (Z_1)$. Therefore, there exists a point $[x]\in \cZ_1$ such that $\tilde{\iota}_\cZ ([x],t) = ([y],t)$. Commutativity of \eqref{eq:commutative-diagram-naturality} implies that $\Lambda_1(\tdelbar_1([x]),t) = \Lambda_2(\tdelbar_2([y],t)) >0$, and therefore $([x],t) \in \cS_1(\tdelbar_1, \Lambda_1)$.
Thus, \eqref{eq:restriction-to-thickening} is a homeomorphism.

\begin{itemize}[leftmargin=0em]
	\item[]\textbf{Step 3:} \emph{For a common regular value $t_0$ the branched integrals of the perturbed solution spaces of $\delbar_1$ and $\delbar_2$ are equal.}
\end{itemize}

By Sard's theorem, we can find a common regular value $t_0 \in B_\ep^N$ of the projections $\cS (\tdelbar_1, \Lambda_1) \to B_\ep^N$ and $\cS (\tdelbar_2,\Lambda_2) \to B_\ep^N$.
For this common regular value, the perturbed solution sets
	\[
	\cS(\delbar_i, \Lambda_i (\cdot,t_0)) := \{	[z]\in \cZ_i \mid \delbar_i(\Lambda_i([z],t_0))>0	\} \subset \cZ_i
	\]
have the structure of weighted branched suborbifolds.

As we have already noted, $\Lambda_2(\cdot,t_0)$ is controlled by the pair $(N_2,\cU_2)$ and hence $\cS(\delbar_2, \Lambda_2 (\cdot,t_0))$ is a compact topological space.
For such a common regular value, the homeomorphism \eqref{eq:restriction-to-thickening} has a well-defined restriction to these perturbed solution sets. This restriction is a homeomorphism, and hence $\cS(\delbar_1, \Lambda_1 (\cdot,t_0))$ is also a compact topological space (even though in general $\Lambda_1(\cdot,t_0)$ will not be controlled by the pair $(N_1,\cU_1)$).

The restriction $\tilde{\iota}_\cZ |_{\cS(\delbar_1, \Lambda_1 (\cdot,t_0))}$ satisfies the necessary hypotheses for the change of variables theorem~\ref{thm:change-of-variables}. 
Therefore for a given $\ssc$-smooth differential form $\ww\in \Omega_\infty^* (\cZ_2)$ we have
	\begin{equation}
	\label{eq:change-variables}
	\int_{\cS(\delbar_2, \Lambda_2 (\cdot,t_0))} \ww
	= \int_{\cS(\delbar_1, \Lambda_1 (\cdot,t_0))} \tilde{\iota}_\cZ^* \ww.
	\end{equation}
However, since in general $\Lambda_1(\cdot,t_0)$ is not controlled by a pair, we cannot assume that it is a regular perturbation in the sense of Definition~\ref{def:regular-perturbation}.
This is problematic since Theorem~\ref{thm:cobordism-between-regular-perturbations} only implies the existence of a compact cobordism between the perturbed solution spaces of two perturbations which are both assumed to be regular perturbations (see Figure~\ref{fig:cobordism}).

\begin{figure}[ht]
	\centering
	\includegraphics{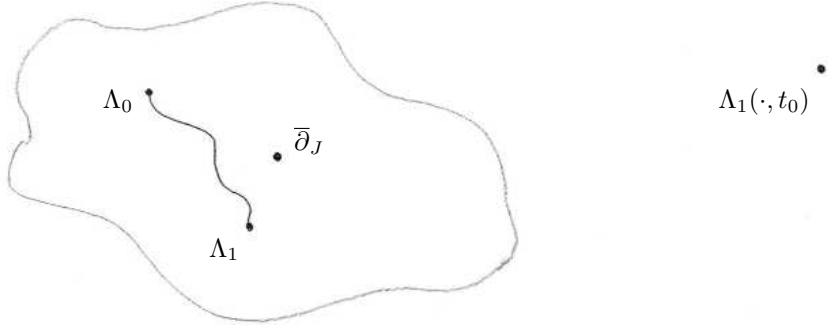}
	\caption{Compact cobordism between regular perturbations}\label{fig:cobordism}
\end{figure}

\begin{itemize}[leftmargin=0em]
	\item[]\textbf{Step 4:} \emph{We show that the set 
		\[
		\cS(\tdelbar_1, \Lambda_1 (\cdot,st_0)) = \{	([z],s) \in \cZ_1\times [0,1]	\mid	\Lambda_1(\delbar_1(z),s t_0)>0	\}
		\]		
		is compact.}
\end{itemize}

Let $\delta = \abs{t_0}$.
The auxiliary norm $\tilde{N}_2 : \cW_2[1] \times \overline{B_{\delta}^N} \to [0,\infty)$ defined by $\tilde{N}_2 ([w],t) := N_2([w])$ together with the open set $\tilde{\cU}_2 := \cU_2 \times \overline{B_{\delta}^N}$ together control the compactness of the extended $\ssc$-smooth Fredholm section $\tdelbar_2$.
By construction, $\Lambda_2$ is controlled by this pair and hence by Theorem~\ref{thm:compactness} the thickened solution set $\cS(\tdelbar_2, \Lambda_2(\cdot, t); t\in \overline{B_\delta ^N})$ is a compact topological space.
Therefore, the closed subset $\cS(\tdelbar_2, \Lambda_2(\cdot, t) ; t = s \cdot t_0, s\in[0,1])$ is also compact.

The restriction of \eqref{eq:restriction-to-thickening} yields a homeomorphism  
	\[
	\cS(\tdelbar_1, \Lambda_1(\cdot, t) ; t = s \cdot t_0, s\in[0,1]) \to \cS(\tdelbar_2, \Lambda_2(\cdot, t) ; t = s \cdot t_0, s\in[0,1]).
	\]
From this it is now clear that 
	\[
	\cS(\tdelbar_1, \Lambda_1 (\cdot,st_0)) \simeq \cS(\tdelbar_1, \Lambda_1(\cdot, t) ; t = s \cdot t_0, s\in[0,1])
	\]
is a compact topological space.

\begin{itemize}[leftmargin=0em]
	\item[]\textbf{Step 5:} \emph{We interpret the pair $(\delbar_1,\Lambda_1(\cdot,t_0))$ as a transversal Fredholm multisection, and use Proposition~\ref{prop:cobordism-multisection-regular} to obtain a compact cobordism to a regular perturbation.}
\end{itemize}

We claim that the hypotheses described in \S~\ref{subsubsec:cobordism-multisection-regular} are satisfied.
In particular, we must show that the extended Fredholm multisection $(\tdelbar_1, \tilde{\Lambda}_1 (\cdot,t_0))$ is proper.
This can be seen using step 4; indeed, the solution set $\cS(\tdelbar_1, \tilde{\Lambda}_1 (\cdot,t_0))$ described in \S~\ref{subsubsec:cobordism-multisection-regular} can be identified with the compact set $\cS(\tdelbar_1, \Lambda_1 (\cdot,st_0))$.

We may therefore use Proposition~\ref{prop:cobordism-multisection-regular} to obtain a cobordism from $(\delbar_1,\Lambda_1(\cdot,t_0))$ to a regular perturbation $\Gamma_0: \cW_1 \to \Q^+$ of $\delbar_1$.
Given a closed $\ssc$-smooth differential form $\ww\in \Omega_\infty^* (\cZ_1)$, Stokes' theorem~\ref{thm:stokes} then implies
	\begin{equation}
	\label{eq:step-5-cobordism-stokes}
	\int_{\cS(\delbar_1,\Gamma_0)}	\ww =\int_{\cS(\delbar_1,\Lambda_1(\cdot,t_0))} \ww.
	\end{equation}
	
\begin{itemize}[leftmargin=0em]
	\item[]\textbf{Step 6:} \emph{We show that the polyfold invariants are equal.}
\end{itemize}

Let $\ww\in H^*_{\dR} (O)$ be the de Rahm cohomology class fixed in the statement of the theorem, and used to define the polyfold invariants.
We can now compute relate the branched integrals as follows:
\begin{align*}
	 \int_{\cS(\delbar_2,\Lambda_2(\cdot,t_0))} f_2^* \ww
	 		& =	\int_{\cS(\delbar_1, \Lambda_1 (\cdot,t_0))} \tilde{\iota}_\cZ^* f_2^*\ww	\\
	 		& =	\int_{\cS(\delbar_1, \Lambda_1 (\cdot,t_0))} f_1^*\ww	\\
	 		& =	\int_{\cS(\delbar_1,\Gamma_0)}	f_1^*\ww,
\end{align*}
where the first equality follows from equation \eqref{eq:change-variables}, the second equality follows from the commutativity of \eqref{eq:gw-invariant-pair-of-polyfolds}, and the third equality follows from equation \eqref{eq:step-5-cobordism-stokes}.
By construction, $\Lambda_2(\cdot,t_0)$ is a regular perturbation of $\delbar_2$, while $\Gamma_0$ is a regular perturbation of $\delbar_1$.
This proves the theorem.
\end{proof}

\subsection{Gromov--Witten invariants are independent of choice of sequence \texorpdfstring{$\delta_i$}{δi}}
	\label{subsec:independence-sequence}

We now use Theorem~\ref{thm:naturality} to show that the Gromov--Witten polyfold invariants are independent of the choice of increasing sequence $(\delta_i)_{i\geq 0}\subset (0,2\pi)$.
Given two sequences $(\delta_i)\subset (0,2\pi)$ and $(\delta_i')\subset (0,2\pi)$ we can always find a third sequence $(\delta_i'')\subset (0,2\pi)$ which satisfies
	\[
	\delta_i \leq \delta_i'', \qquad \delta_i' \leq \delta_i''
	\]
for all $i$.  The GW-polyfold associated to the sequence $(\delta_i'')$ give a refinement of the GW-polyfolds associated to $(\delta_i)$ and $(\delta_i')$, in the sense that there are inclusion maps
	\[
	\cZ_{A,g,k}^{3,\delta'_0} \hookleftarrow \cZ_{A,g,k}^{3,\delta''_0} \hookrightarrow \cZ_{A,g,k}^{3,\delta_0}.
	\]
It is therefore sufficient to consider inclusion maps of the form
$\cZ^{3,\delta_0}_{A,g,k} \hookrightarrow \cZ^{3,\delta_0'}_{A,g,k}$
with $\delta_i' \leq \delta_i$ for all $i$ and demonstrate that the associated GW-invariants are equal.
%\red{careful i just switched the location of the primes to be consistent with the main proposition...}

To this end, consider the commutative diagram:
\[\begin{tikzcd}
\cW^{2,\delta_0}_{A,g,k} \arrow[r, "\iota_\cW"', hook] \arrow[d, "\delbarj\quad "'] & \cW^{2,\delta_0'}_{A,g,k} \arrow[d, "\quad \delbarj'"] &  \\
\cZ^{3,\delta_0}_{A,g,k} \arrow[r, "\iota_\cZ"', hook] \arrow[u, bend left] & \cZ^{3,\delta_0'}_{A,g,k} \arrow[u, bend right] & 
\end{tikzcd}\]
and observe that it satisfies the same properties as \eqref{eq:commutative-diagram-naturality}.
In addition, consider the commutative diagram:
\[\begin{tikzcd}
&  & Q^k\times \dmlog_{g,k} \\
\cZ^{3,\delta_0}_{A,g,k} \arrow[r, "\iota_\cZ"', hook] \arrow[rru, "ev_i \times \pi"] & \cZ^{3,\delta_0'}_{A,g,k} \arrow[ru, "ev_i\times \pi"'] & 
\end{tikzcd}\]
which satisfies the same properties as \eqref{eq:gw-invariant-pair-of-polyfolds}.

Note that if $\delta_0' < \delta_0$ the inclusion map $\iota_\cZ$ is not proper.
To see this, exploit the difference in exponential weights to produce a sequence which converges in a local M-polyfold model for $\cZ^{3,\delta_0}_{A,g,k}$ but diverges in a local M-polyfold model for $\cZ^{3,\delta_0'}_{A,g,k}$.
Note also that the pullback strong polyfold bundle is not the same as the standard strong polyfold bundle on $\cZ^{3,\delta_0}_{A,g,k}$.

\begin{proposition}
	\label{prop:existence-subbundle-naturality}
	The set	
	\[
	\cR := \{	[\Sigma,j,M,D,u,\xi] \in \cW^{2,\de_0'}_{A,g,k} \mid \supp \xi \subset K \subset \Sigma\setminus\abs{D} \text{ for some compact } K	\}
	\]
	is an intermediary subbundle of the strong polyfold bundle $\cW^{2,\de_0'}_{A,g,k}$.
\end{proposition}
\begin{proof}
	We must show that the set $\cR$ satisfies the properties of Definition~\ref{def:intermediate-subbundle}.
	The first two properties can be easily checked.
	
	We show how to construct the $\ssc^+$-sections required by property \ref{property-3-intermediary-subbundle}.
	Consider a stable curve $[\alpha]=[\Sigma,j,M,D,u] \in \cS(\delbarj') \subset \cZ^{3,\de_0'}_{A,g,k}$ and let let $\alpha = (\Sigma,j,M,D,u)$ be a stable map representative.
	Let $V_\alpha \subset U_\alpha$ be a $\bm{G}(\alpha)$-invariant M-polyfold charts centered at $\alpha$ such that $\overline{V_\alpha}\subset U_\alpha$.
	This means we have good uniformizing family 
	\[
	(a,v,\eta ) \mapsto (\Sigma_{a},j(a,v),M_{a},D_{a}, \oplus_{a} \exp_u (\eta)),\qquad (a,v,\eta) \in O_\alpha.
	\]
	Let $K\to O_\alpha$ be a local strong bundle model, with $\ssc$-coordinates given by $(a,v,\eta,\xi)$ where $\xi \in H^{2,\de_0'}(\Sigma,\Lambda^{0,1} \otimes_J u^* TQ)$.
	
	Use Corollary~\ref{cor:vectors-which-span-cokernel} to choose vectors $v_1,\ldots, v_k$ which vanish on disk-like regions of the nodal points and such that
	\begin{itemize}
		\item $\text{span}\{v_1,\ldots , v_k\} \oplus \text{Im}(D\hdelbarj'(\alpha)) = H^{2,\de_0'}(\Sigma,\Lambda^{0,1} \otimes_J u^* TQ)$,
		\item $\text{span}\{v_1,\ldots , v_k\} \oplus \text{Im}(D\hdelbarj(\alpha)) = H^{2,\de_0}(\Sigma,\Lambda^{0,1} \otimes_J u^* TQ)$.
	\end{itemize}
	
	Let $\beta: U_\alpha \to [0,1]$ be an $\ssc$-smooth cutoff function which satisfies $\beta\equiv 1$ near $(0,0,0)$, and $\beta \equiv 0$ on $U_\alpha\setminus V_\alpha$.
	In these local $\ssc$-coordinates, the desired $\ssc^+$-sections are defined as follows:
	\[
	s_i': U_\alpha\to K,\qquad (a,v,\eta) \mapsto (a,v,\eta, \beta(a,v,\eta) \cdot ( \rho_{a} (v_i)		)),
	\]
	where $\rho_{a}$ is the strong bundle projection defined using the hat gluings, see \cite[p.~117]{HWZbook} and \cite[pp.~65--67]{HWZGW}
	
	Let $(N,\cU)$ be a pair which controls the compactness of $\delbarj'$.
	By construction, these $\ssc^+$-sections satisfy $N[s_i'] \leq C$ for some constant $C\leq \infty$, and hence by rescaling the vectors $v_i$ we may assume that $N[s_i'] \leq 1$. Moreover, by shrinking the support of the cutoff function we may assume that $\abs{\supp s_i'} = \abs{\supp \beta } \subset \cU$.
	
	By construction, the $\ssc^+$-section $s_i'$ induces a well-defined restriction $s_i|_{\hat{\iota}_\cZ^{-1} (U_\alpha)}$. Locally this restriction is given by multiplying the $\ssc$-smooth cutoff $\beta \circ \hat{\iota}_\cZ$ and the locally constant vector $v_i$, hence it is $\ssc^+$.
\end{proof}

Having shown the previous proposition, we may immediately apply Theorem~\ref{thm:naturality} to see that 
the polyfold Gromov--Witten invariants do not depend on the choice of an increasing sequence $(\delta_i)_{i\geq 0} \allowbreak \subset (0,2\pi)$.
This proves Corollary~\ref{cor:naturality-polyfold-gw-invariants}.

\subsection[Gromov-Witten invariants are independent of punctures at marked points]{Gromov--Witten invariants are independent of punctures at marked points}
\label{subsec:independence-punctures}

We now recall the regularity estimates that the stable curves of the Gromov--Witten polyfolds as constructed in \cite{HWZGW} are required to satisfy.
Let $u: \Sigma \to Q$ be a continuous map, and fix a point $z\in \Sigma$.
We consider a local expression for $u$ as follows. Choose a small disk-like neighborhood $D_z\subset \Sigma$ of $z$ such that there exists a biholomorphism $ \sigma:[0,\infty)\times S^1\rightarrow D_z \setminus\{z\}$.
Let $\varphi:U\rightarrow \R^{2n}$ be a smooth chart on a neighborhood $U \subset Q$ of $u(z)$ such that to $\varphi (u(z))=0$.
The local expression 
\[
\tilde{u}: [s_0,\infty)\times S^1 \to \R^{2n}, \qquad (s,t) \mapsto \varphi\circ u\circ \sigma(s,t)
\]
is defined for $s_0$ large.
Let $m\geq 3$ be an integer, and let $\delta >0$. We say that $u$ is of \textbf{class $H^{m,\delta}$ around the point $z\in \Sigma$} %written as $u\in H^{m,\delta} (D_z\setminus\{z\})$,
if $e^{\delta s}  \tilde{u}$ belongs to the space $L^2([s_0,\infty)\times S^1,\R^{2n})$.
We say that $u$ is of \textbf{class} $H^m_{\text{loc}}$ \textbf{around the point $z\in \Sigma$} if $u$ belongs to the space $H^m_{\text{loc}}(D_z)$.
If $u$ is of class $H^{m,\delta}$ at a point $z\in \Sigma$ we will refer to that point as a \textbf{puncture}.

By definition, any stable map representative $(\Sigma,j,M,D,u)$ of a stable curve in the Gromov--Witten polyfold $\cZ_{A,g,k}$ is required to be of class $H^{3,\de_0}$ at all nodal points.
This is required in order to carry out the gluing construction at the nodes of \cite[\S~2.4]{HWZGW}.

However, in some situations we would like to treat the marked points in the same way as the nodal points.
Note that allowing a puncture with exponential decay at a specified marked point is a global condition on a Gromov--Witten polyfold.
Hence, we will need to require that the map $u$ is of class $H^{3,\delta_0}$ at a fixed subset of the marked points (in addition to the nodal points).

Given a subset $I \subset \{1,\ldots,k\}$ we can define a GW-polyfold $\cZ^I_{A,g,k}$ where we require that all stable map representatives are of class $H^{3,\delta_0}$ at the marked points $z_i$ for all $i\in I$ and of class $H^3_\text{loc}$ at the remaining marked points.
Given another subset $J \subset\{1,\ldots,k\}$ we can define a GW-polyfold $\cZ^J_{A,g,k}$ in the same manner.
On the other hand, we can consider a GW-polyfold $\cY_{A,g,k}$ where we require that all stable map representatives are of class $H^{3,\delta_0}$ and also of class $H^3_\text{loc}$ at all the marked points.
Such a GW-polyfold with strict regularity at all marked points gives a refinement of the GW-polyfolds with different choices of punctures $I,J\subset \{1,\ldots,k\}$ at the marked points, in the sense that there are inclusion maps
\[
\cZ^I_{A,g,k} \hookleftarrow \cY_{A,g,k} \hookrightarrow \cZ^J_{A,g,k}.
\]
It is sufficient to consider inclusion maps of the form $\cY_{A,g,k} \hookrightarrow \cZ^I_{A,g,k}$ and demonstrate that the associated GW-invariants are equal.

To this end, consider the commutative diagram:
\[\begin{tikzcd}
\cV_{A,g,k} \arrow[r, "\iota_\cW"', hook] \arrow[d, "\delbarj\quad "'] & \cW^I_{A,g,k} \arrow[d, "\quad \delbarj"] &  \\
\cY_{A,g,k} \arrow[r, "\iota_\cZ"', hook] \arrow[u, bend left] & \cZ^I_{A,g,k} \arrow[u, bend right] & 
\end{tikzcd}\]
and observe that it satisfies the same properties as \eqref{eq:commutative-diagram-naturality}.
In addition, consider the commutative diagram:
\[\begin{tikzcd}
&  & Q^k\times \dmlog_{g,k} \\
\cY_{A,g,k} \arrow[r, "\iota_\cZ"', hook] \arrow[rru, "ev_i \times \pi"] & \cZ^I_{A,g,k} \arrow[ru, "ev_i\times \pi"'] & 
\end{tikzcd}\]
which satisfies the same properties as \eqref{eq:gw-invariant-pair-of-polyfolds}.

There exist sequences of maps which converge in $H^3_\text{loc}$ but do not converge in $H^{3,\delta_0}$.  Consequently, in general the above inclusion map is not proper.  Furthermore, the pullback strong polyfold bundle is not the same as the standard strong polyfold bundle on $\cY_{A,g,k}$.

\begin{proposition}
	The set	
	\[
	\cR := \{	[\Sigma,j,M,D,u,\xi] \in \cW^I_{A,g,k} \mid \supp \xi \subset K \subset \Sigma\setminus M \text{ for some compact } K	\}
	\]
	is an intermediary subbundle of the strong polyfold bundle $\cW^I_{A,g,k}$.
\end{proposition}
\begin{proof}
	The proof is identical to the proof of Proposition~\ref{prop:existence-subbundle-naturality}, except here we use Corollary~\ref{cor:vectors-which-span-cokernel} to choose vectors $v_i$ which vanish on disk-like regions of the marked points instead of the nodal points.
\end{proof}

Again, combining the previous proposition and Theorem~\ref{thm:naturality} we see that the polyfold Gromov--Witten invariants do not depend on the choice of puncture at the marked points.
This proves Corollary~\ref{cor:punctures-equal}.

%%%%%%%%%%%%%%%%%%%%%%%%%%%%%%%%%%%%%%%%%%%%%%%%%%
% Pulling Back Abstract Perturbations %%%%%%%%%%%%
%%%%%%%%%%%%%%%%%%%%%%%%%%%%%%%%%%%%%%%%%%%%%%%%%%

\section{Pulling back abstract perturbations}
	\label{sec:pulling-back-abstract-perturbations}

In this section we show how to construct a regular perturbation which pulls back to a regular perturbation, culminating in Theorem~\ref{thm:compatible-pullbacks} and in Corollary~\ref{cor:pullback-via-permutation}.

\subsection{Pullbacks of strong polyfold bundles}
	\label{subsec:pullbacks-strong-polyfold-bundles}

Let $P: \cW \to \cZ$ be a strong polyfold bundle, and let $f: \cY \to \cZ$ be a $\ssc$-smooth map between polyfolds.
Consider the topological pullback
	\[f^* \cW = \{([x],[w]_{[y]})	\mid f([x])=[y]=P([w]_{[y]})	\} \subset \cY \times \cW\]
equipped with the subspace topology. Since $\cY$ and $\cW$ are second countable, paracompact, Hausdorff topological spaces, so too is the product $\cY \times \cW$ and hence $f^*\cW$ is also a second countable, paracompact, Hausdorff topological spaces.

We can take the pullback $\hat{f}^* W$ of the object strong M-polyfold bundle; by Proposition~\ref{prop:pullback-bundle} this has the structure of a strong M-polyfold bundle over the object space $Y$.
The fiber product 
	\[\bm{Y} _s \times_{\text{proj}_1} \hat{f}^* W = \{	(\phi,y,w_x) \mid s(\phi) = y, \hat{f}(y)=x	\}\]
may be viewed as the strong M-polyfold bundle via the source map $s$ over the morphism space $\bm{Y}$,
	\[\begin{tikzcd}
	\bm{Y} _s \times_{\text{proj}_1} \hat{f}^* W \arrow[d] \arrow[r] & \hat{f}^* W \arrow[d] &  \\
	\bm{Y} \arrow[r, "s"'] & Y & 
	\end{tikzcd}\]

We may define a \textbf{pullback strong polyfold bundle structure} over the polyfold structure $(Y,\bm{Y})$ as the strong M-polyfold bundle $\text{proj}_1 : \hat{f}^*W \to Y$ together with the bundle map defined as follows:
	\begin{align*}
	\lambda: \bm{Y} _s \times_{\text{proj}_1} \hat{f}^* W & \to \hat{f}^*W \\
							(\phi,y,w_x)	& \mapsto (t(\phi), \mu(f(\phi), w_x))
	\end{align*}
It is straightforward to check that this map satisfies the requirements of Definition~\ref{def:strong-polyfold-bundle}.

Given a $\ssc$-smooth section $\delbar:\cZ\to \cW$ there exists a well-defined \textbf{pullback section} $f^*\delbar:\cY \to f^*\cW$.  Between the underlying sets, it is defined by
	\[
	f^*\delbar ([x]) = ([x], \delbar \circ f ([x]) ).
	\]
It is automatically regularizing if $\delbar$ is regularizing.
We may define the pullback of a $\ssc^+$-multisection as follows.

\begin{definition}
	\label{def:pullback-multisection}
	Given a $\ssc^+$-multisection $\Lambda:\cW\to \Q^+$ there exists a well-defined \textbf{pullback $\ssc^+$-multisection} $\text{proj}_2^*\Lambda :f^*\cW \to \Q^+$.
	It consists of the following:
	\begin{enumerate}
		\item the function $\Lambda \circ \text{proj}_2 :f^*\cW\to\Q^+$
		\item the functor $\hat{\Lambda} \circ \text{proj}_2 :\hat{f}^* W\to \Q^+$
		\item at each $[x]\in \cZ_1$ there exists a ``pullback local section structure'' for $\text{proj}_2^*\Lambda$, defined below.
	\end{enumerate}
\end{definition}

Given $[x]\in\cZ_1$, the local section structure for $\text{proj}_2^*\Lambda$ at $[x]$ is described as follows.
Let $x\in Z_1$ be a representative of $[x]$.
Let $y:=\hat{f}(x)\in Z_2$, and let $U_y\subset Z_2$ be a $\bm{G}(y)$-invariant open neighborhood of $y$, and let 
$s_1,\ldots,s_k : U_y \to W$ be a local section structure for $\hat{\Lambda}$ at $y$ with associated weights $\sigma_1,\ldots ,\sigma_k$.
Let $U_x\subset Z_1$ be a $\bm{G}(x)$-invariant open neighborhood of $x$ such that $\hat{f}(U_x)\subset U_y$, and consider the restricted strong M-polyfold bundle $(\hat{f}^*W)|_{U_x} \to U_x$.
Then the pullback of the local sections $\hat{f}^*s_1,\ldots,\hat{f}^*s_k :U_x \to \hat{f}^* W$ with the associated weights $\sigma_1,\ldots ,\sigma_k$ gives the local section structure for $\text{proj}_2^*\Lambda$ at $[x]$.

Indeed, it tautologically follows from the original assumption that $s_1,\ldots,s_k$, $\sigma_1,\ldots,\sigma_k$ is a local section structure for $\Lambda$ at $[y]$ that
\begin{enumerate}
	\item $\sum_{i=1}^k \sigma_i =1$ 
	\item the local expression $\text{proj}_2^* \hat{\Lambda}:\hat{f}^*W|_{U_x}\to \Q^+$ is related to the local sections and weights via the equation 
	\[
	\text{proj}_2^* \hat{\Lambda} (x',w_{y'}) = \sum_{\{i\in \{1,\ldots, k \} \mid (x',w_{y'}) = \hat{f}^*s_i (\text{proj}_1 (x',w_{y'} ) )\}} \sigma_i
	\]
	for all $(x',w_{y'}) \in (\hat{f}^* W)|_{U_x}$ (which necessarily satisfy $\hat{f}(x')= y' = P(w_{y'})$).
\end{enumerate}

\subsection{The topological pullback condition and pullbacks of pairs which control compactness}
	\label{subsec:topological-pullback-condition-controlling-compactness}

Suppose at the outset that we have a $\ssc$-smooth map $f: \cY \to \cZ$ and a strong polyfold bundle $P:\cW \to \cZ$ with a $\ssc$-smooth proper Fredholm section $\delbar$.
Consider the pullback of this bundle and of this section via the map $f$ yielding the following commutative diagram:
	\[\begin{tikzcd}
	f^* \cW \arrow[d, "f^* \delbar\quad"'] \arrow[r, "\text{proj}_2"'] & \cW \arrow[d, "\quad \delbar"] &  \\
	\cY \arrow[r, "f"'] \arrow[u, bend left] & \cZ. \arrow[u, bend right] & 
	\end{tikzcd}\]
Assume that the $\ssc$-smooth section $f^*\delbar$ is a proper Fredholm section; such an assumption is not automatic from the above setup, however it is natural in the context of polyfold maps one might encounter.\footnote{An alternative to outright assuming that $f^*\delbar$ is a Fredholm section would be to formulate a precise notion of a ``Fredholm map'' for a map between polyfolds, and then require that $f$ is such a map. This would also be natural in the context of polyfold maps one might encounter.}

Given a pair $(N,\cU)$ which controls the compactness of $\delbar$ we show in this subsection how to obtain a pullback of this pair, which will control the compactness of $f^*\delbar$.

\begin{proposition}
	Let $N :\cW[1] \to [0,\infty)$ be an auxiliary norm. The pullback of $N$, given by
	\[
	\text{proj}_2^* N : f^* \cW[1] \to [0,\infty)
	\]
	defines an auxiliary norm on the pullback strong polyfold bundle $f^* \cW\to \cY$.
\end{proposition}
\begin{proof}
	This is immediate from the definitions.  In particular, property \ref{property-2-auxiliary-norm} of Definition~\ref{def:auxiliary-norm} can be checked as follows.  Let $(x_k, w_k)$ be a sequence in $\hat{f}^* W[1]$, such that $x_k$ converges to $x$ in $Y$, and suppose $\text{proj}_2^* \hat{N} (x_k, w_k) \to 0$.  Then $w_k$ is a sequence in $W[1]$ such that $\hat{f}(x_k)$ converges to $\hat{f}(x)$ in $Z$, and $\hat{N} (w_k) = \text{proj}_2^* \hat{N} (x_k, w_k) \to 0$, and hence $w_k\to 0_{\hat{f}(x)}$.  Thus $(x_k,w_k)\to (x,0_{\hat{f}(x)})$, as required.
\end{proof}

\begin{definition}\label{topological-pullback-condition}
	We say that $f$ satisfies the \textbf{topological pullback condition} if for all $[y] \in \cS(\delbar)\subset \cZ$ and for any open neighborhood $\cV \subset \cY$ of the fiber $f^{-1} ([y])$ there exists an open neighborhood $\cU_{[y]}\subset \cZ$ of $[y]$ such that $f^{-1} (\cU_{[y]}) \subset \cV$.
	Note that if $f^{-1} ([y])=\emptyset$, this implies that there exists an open neighborhood $\cU_{[y]}$ of $[y]$ such that $f^{-1} (\cU_{[y]})=\emptyset$.
\end{definition}

\begin{proposition}
	\label{prop:simultaneous-compactness}
	Suppose that the map $f:\cY \to \cZ$ satisfies the topological pullback condition.  
	Then there exists a pair $(N,\cU)$ which controls the compactness of $\delbar$ such that the pair $(\text{proj}_2^* N, f^{-1} (\cU) )$ controls the compactness of $f^* \delbar$.
	
	Furthermore, if a $\ssc^+$-multisection $\Lambda: \cW\to \Q^+$ satisfies $N [\Lambda ] \leq 1$ and $\text{dom-supp} (\Lambda) \subset \cU$, then its pullback $\text{proj}_2^* \Lambda :f^*\cW \to \Q^+$ satisfies $\text{proj}_2^* N [\text{proj}_2^* \Lambda ] \leq 1$ and $\text{dom-supp} (\text{proj}_2^* \Lambda) \allowbreak \subset \allowbreak f^{-1} (\cU)$.
\end{proposition}
\begin{proof}
	Let $(N, \cV)$ be a pair which controls the compactness of $\delbar$.
	By the previous proposition we know that the pullback $\text{proj}_2^* N :f^* \cW[1]\to [0,\infty)$ is an auxiliary norm.
	We may then apply \cite[Prop.~4.5]{HWZ3} to assert the existence of a neighborhood $\cU'\subset \cY$ of $\cS(f^*\delbar)$ such that the pair $(\text{proj}_2^* N, \cU')$ controls the compactness of $f^*\delbar$.
	
	At every $[y] \in \cS(\delbar)$, observe that $f^{-1} ([y]) \subset \cS(f^*\delbar) \subset \cU'$.  We can use the topological pullback condition to choose a neighborhood $\cU_{[y]}$ of $[y]$ such that $f^{-1} (\cU_{[y]}) \subset \cU'\subset \cY$ and moreover such that $\cU_{[y]} \subset \cV\subset \cZ$.  %This gives a cover of $\cS (\delbar)$ by compactness we may take a finite subcover $\{\cU_{[y]_i}\}$ for finitely many $i$.
	We define an open neighborhood by $\cU := \cup_i \ \cU_{[y]_i}$ for every $[y]_i\in \cS(\delbar)$.
	
	Then $(N,\cU)$ is the desired pair.
	Indeed, $\cU$ is an open neighborhood of the unperturbed solution set $\cS(\delbar)$.
	And $\cU\subset \cV$ since for every $[y]_i$ we have $\cU_{[y]_i}\subset \cV$.  Hence we have $\cS(\delbar) \subset \cU\subset \cV$ therefore it follows from Remark~\ref{rmk:shrink-neighborhood} that $(N,\cU)$ controls the compactness of $\delbar$.
	
	Observe that $\cS(f^*\delbar) = f^{-1} (\cS(\delbar)) \subset f^{-1}(\cU)$.  By the construction of $\cU$ we have $f^{-1} (\cU) \subset \cU'$.  Hence we have $\cS(f^*\delbar) \subset f^{-1}(\cU) \subset \cU'$ therefore it follows from Remark~\ref{rmk:shrink-neighborhood} that $(\text{proj}_2^* N, f^{-1} (\cU))$ controls the compactness of $f^* \delbar$.
	
	Finally, the claim regarding the pullback of a $\ssc^+$-multisection is immediate from the construction.
\end{proof}

\subsection{Construction of regular perturbations which pullback to regular perturbations}
	\label{subsec:construction-regular-perturbation-which-pullback}

With the same assumptions and setup as in the previous subsection (i.e., our map satisfies the topological pullback condition), we show in this subsection how to construct regular perturbations which will pullback to regular perturbations

\begin{theorem}
	\label{thm:compatible-pullbacks}
	We can construct a regular perturbation $\Lambda: \cW \to \Q^+$ which pulls back to a regular perturbation $\text{proj}_2^* \Lambda$. This means that the perturbations satisfy the following conditions.
	\begin{enumerate}
		\item $(\delbar, \Lambda)$ and $(f^* \delbar, \text{proj}_2^* \Lambda)$ are both transversal pairs.
		\item There exists a pair $(N,\cU)$ which controls the compactness of $\delbar$ such that the pair $(\text{proj}_2^* N, f^{-1} (\cU) )$ controls the compactness of $f^* \delbar$. Then:
		\begin{itemize}
			\item $N [\Lambda ] \leq 1$ and $\text{dom-supp} (\Lambda) \subset \cU$
			\item $\text{proj}_2^* N [\text{proj}_2^* \Lambda ] \leq 1$ and $\text{dom-supp} (\text{proj}_2^* \Lambda) \subset f^{-1} (\cU)$.
		\end{itemize}
	\end{enumerate}
	%\red{What about orientations?}
\end{theorem}
\begin{proof}

We now give an explicit construction of a $\ssc^+$-multisection $\Lambda : \cW \to \Q^+$ such that $(\delbar, \Lambda)$ and $(f^* \delbar, \text{proj}_2^* \Lambda)$ are both transversal pairs.  Our approach is based on the general position argument of \cite[Thm.~15.4]{HWZbook}.

\noindent\emph{Local construction.}
We construct a $\ssc^+$-multisection $\Lambda_0 : \cW \to \Q^+$ which will be transversal at a point $[y_0] \in \cS(\delbar) \subset \cZ$.
Let $U(y_0) \subset Z$ be a $\bm{G}(y_0)$-invariant open neighborhood of $y_0$ and moreover let $V(y_0)\subset U(y_0)$ be a $\bm{G}(y_0)$-invariant open neighborhood of $y_0$ such that $\overline{V(y_0)} \subset U(y_0)$.

Choose smooth vectors $v_1,\ldots, v_k \in W_{y_0}$ such that
	\[
	\text{span} \{v_1,\ldots,v_k \} \oplus D\hdelbar_{y_0} (T_{y_0} Z) = W_{y_0}.
	\]

For each smooth vector $v_1,\ldots, v_k$ we can use \cite[Lem.~5.3]{HWZbook} to define $\ssc^+$-sections $s_i : U(y_0) \to W$ such that
\begin{itemize}
	\item $s_i= 0$ on $U(y_0)\setminus V(y_0)$,
	\item $s_i(y_0) = v_i$.
\end{itemize}
Furthermore, to ensure that the resulting multisection will be controlled by the pair $(N,\cU)$ we require that
\begin{itemize}
	\item $N[s_i] \leq 1,$
	\item $\text{supp}(s_i)\subset \cU$.
\end{itemize}

We may use these locally constructed $\ssc^+$-sections to define a $\ssc^+$-multisection functor 
	\[\hat{\Lambda}'_0:W \times B_\ep^k \to \Q^+\]
with local section structure given by 
$\left\lbrace g * \left(\sum_{i=1}^k t_i \cdot s_i\right)\right\rbrace_{g\in \bm{G}(y_0)}$
which satisfies the following.
There exists a $\bm{G}(y_0)$-invariant open neighborhood $y_0 \subset U'_0 \subset Z$ such that at any object $y\in U'_0$ and for any $g\in \bm{G}(y_0)$ the linearization of the function 
	\begin{align*}
	U'_0 \times B_\ep^k 	&\to W\\
	(y, t_1,\ldots,t_k)		&\mapsto \hdelbar(y) - g * \left(\sum_{i=1}^k t_i \cdot s_i(y)\right)
	\end{align*}
projected to the fiber $W_y$ is surjective.

We now construct a $\ssc^+$-multisection whose pullback $\text{proj}_2^* \Lambda_0 : f^*\cW \to \Q^+$ will be transversal at a point $[x_0] \in \cS(f^*\delbar) \subset \cY$.
Consider a point $[x_0]\in \cS(f^*\delbar) \subset \cY$ which maps to $[y_0]:=f([x_0])\in \cS(\delbar)\subset \cZ$.
Let $x_0\in S(\hat{f}^*\hdelbar)$ be a representative of $[x_0]$, and let $U(x_0) \subset Y$ be a $\bm{G}(x_0)$-invariant open neighborhood of $x_0$.
Then $y_0:= \hat{f}(x_0) \in S(\hdelbar)$ is a representative of $[y_0]$.
Let $U(y_0) \subset Z$ be a $\bm{G}(y_0)$-invariant open neighborhood of $y_0$ and moreover let $V(y_0)\subset U(y_0)$ be a $\bm{G}(y_0)$-invariant open neighborhood of $y_0$ such that $\overline{V(y_0)} \subset U(y_0)$.
By shrinking the open set $U(x_0)$, we may assume that the local expression $\hat{f}: U(x_0)\to U(y_0)$ is well-defined.

The fibers $(\hat{f}^*W)_{x_0}$ and $W_{y_0}$ may be identified via $\text{proj}_2$.
Choose smooth vectors $v_1,\ldots, v_k \in W_{y_0}$ such that
	\[\text{span} \{\text{proj}_2^{-1} (v_1),\ldots,\text{proj}_2^{-1} (v_k) \} \oplus D\hat{f}^*\hdelbar_{x_0} (T_{x_0} Y) = (\hat{f}^*W)_{x_0}.\]

For each smooth vector $v_1,\ldots, v_k$ we may use \cite[Lem.~5.3]{HWZbook} to define $\ssc^+$-sections $s_i : U(y_0) \to W$ such that
\begin{itemize}
	\item $s_i\equiv 0$ on $U(y_0)\setminus V(y_0)$,
	\item $s_i(y_0) = v_i$.
\end{itemize}
Furthermore, to ensure that the resulting multisection will be controlled by the pair $(N,\cU)$ we require that
\begin{itemize}
	\item $N[s_i] \leq 1,$
	\item $\text{supp}(s_i)\subset \cU$.
\end{itemize}

We may use these locally defined $\ssc^+$-sections to define a $\ssc^+$-multisection functor $\hat{\Lambda}_0:	W \times B_\ep^k \to \Q^+$ 
with local section structure given as follows. 
	\[\left\lbrace g * \left(\sum_{i=1}^k t_i \cdot s_i\right)\right\rbrace_{g\in \bm{G}(y_0)}\]
By construction, the pullback $\ssc^+$-multisection functor
$\text{proj}_2^* \hat{\Lambda}_0:	\hat{f}^* W \times B_\ep^k \to \Q^+$ 
has local section structure given as follows. 
	\[\left\lbrace \hat{f}^* \left(g * \left(\sum_{i=1}^k t_i \cdot s_i\right)\right) \right\rbrace_{g\in \bm{G}(y_0)}\]
It may be observed that there exists a $\bm{G}(x_0)$-invariant open neighborhood $x_0 \subset U_0\subset Y$ such that at any object $x\in U_0$ and for any $g\in \bm{G}(y_0)$ the linearization of the function 
	\begin{align*}
	U_0 \times B_\ep^k 	&\to \hat{f}^*W\\
	(x, t_1,\ldots,t_k)	&\mapsto \hat{f}^* \hdelbar (x) - \hat{f}^* \left( g * \left(\sum_{i=1}^k t_i \cdot s_i\right)\right) (x) \\
						&\phantom{\mapsto} = \left(x, \hdelbar (\hat{f}(x)) - g * \left(\sum_{i=1}^k t_i \cdot s_i(\hat{f}(x)) \right) \right)
	\end{align*}
projected to the fiber $(\hat{f}^* W)_x$ is surjective.

\noindent\emph{Global construction.}
We may cover the compact topological space $\cS(\delbar)$ by a finite collection of such neighborhoods $\abs{U'_i}$ of points $[y_i]\in \cS(\delbar)$; we may also cover $\cS(f^*\delbar)$ by a finite collection of such neighborhoods $\abs{U_i}$ of points $[x_i]\in \cS(f^*\delbar)$.
It follows that the finite sum of $\ssc^+$-multisections
	\[
	\Lambda:= \bigoplus_i \Lambda_i : \cW \times B_\ep^N \to \Q^+
	\]
has the property that: for any point $[y] \in \cZ$ with $\Lambda \circ \delbar ([y])>0$, and for any parametrized local section structure $\{s_i\}_{i\in I}$ at a representative $y$, the linearization of the function $\hdelbar (y) - s_i(y,t)$ projected to the fiber $W_y$ is surjective.

Likewise, the pullback $\ssc^+$-multisection $\text{proj}_2^* \Lambda : f^*\cW \times B_\ep^N \to \Q^+$
has the property that for any point $[x] \in \cY$ which satisfies $\text{proj}_2^* \Lambda \circ f^*\delbar ([x])>0$ and for any parametrized local section structure $\{s_i\}_{i\in I}$ at a representative $x$, the linearization of the function $\hdelbar (x) - s_i(x,t)$ projected to the fiber $(\hat{f}^*W)_x$ is surjective. 

Furthermore for $\ep$ sufficiently small, for any fixed $t_0 \in B_\ep^N$ the $\ssc^+$-multisection $\Lambda (\cdot, t_0)$ is controlled by the pair $(N,\cU)$, i.e.,
\begin{itemize}
	\item $N[\Lambda(\cdot,t_0)] \leq 1$,
	\item $\text{dom-supp}(\Lambda(\cdot,t_0)) \subset \cU$.
\end{itemize}
It follows from Proposition~\ref{prop:simultaneous-compactness} that the pullback $\ssc^+$-multisection $\text{proj}_2^* \Lambda (\cdot, t_0)$ is controlled by the pullback $(\text{proj}_2^*N,f^{-1}(\cU))$.

\noindent\emph{Common regular value.}
Consider the strong polyfold bundle $\cW \times B_\ep^N \to \cZ \times B_\ep^N$ and the pullback strong polyfold bundle $f^*\cW \times B_\ep^N \to \cY \times B_\ep^N$.
Let $\tdelbar : \cZ \times B_\ep^N \to \cW \times B_\ep^N$ denote the $\ssc$-smooth proper Fredholm section defined by $([y],t)\mapsto (\delbar([y]),t)$, and let $\tilde{f}^*\tdelbar :\cY \times B_\ep^N \to f^*\cW\times B_\ep^N$ denote the $\ssc$-smooth proper Fredholm section defined by $([x],t) \mapsto ([x],\delbar(f([x])),t)$.
By construction, $(\tdelbar, \Lambda)$ and $(\tilde{f}^*\tdelbar, \text{proj}_2^* \Lambda)$ are transversal pairs; hence by Theorem~\ref{thm:transversal-pairs-weighted-branched-suborbifolds} the thickened solution sets
	\begin{align*}
	\cS(\tdelbar, \Lambda) &= \{ ([y],t) \in \cZ \times B_\ep^N \mid \Lambda (\tdelbar ([y],t)) >0 \} \subset \cZ \times B_\ep^N,	\\
	\cS(\tilde{f}^*\tdelbar, \text{proj}_2^* \Lambda) &= \{ ([x],t) \in \cY \times B_\ep^N \mid \text{proj}_2^*\Lambda (\tilde{f}^*\tdelbar ([x],t)) >0 \} \subset \cY \times B_\ep^N
	\end{align*}
have the structure of weighted branched orbifolds.

By Sard's theorem, we can find a common regular value $t_0 \in B_\ep^N$ of the projections $\cS(\tdelbar, \Lambda) \to B_\ep^N$ and $\cS(\tilde{f}^*\tdelbar, \text{proj}_2^* \Lambda)\to B_\ep^N$.
Then the $\ssc^+$-multisection $\Lambda(\cdot,t_0) : \cW \to \Q^+$ and its pullback $\text{proj}_2^*\Lambda(\cdot,t_0) :f^*\cW \to \Q^+$ are the desired regular perturbations.
\end{proof}

The significance of this theorem is the following.
Both perturbed solution sets $\cS(f^*\delbar, \text{proj}_2^* \Lambda )$ and $\cS(\delbar,\Lambda)$ have the structure of compact oriented weighted branched suborbifolds.  
Moreover, the restriction of $f$ gives a well-defined continuous function between these perturbed solution spaces, i.e.,
	\[
	f|_{\cS(f^*\delbar, \text{proj}_2^* \Lambda )}: \cS(f^*\delbar, \text{proj}_2^* \Lambda ) \to \cS(\delbar,\Lambda).
	\]
Furthermore, $f$ is weight preserving in the sense that the weight functions are related via pullback by the following equation $ ( \Lambda \circ \delbar) \circ f= \text{proj}_2^* \Lambda \circ f^* \delbar$.

\subsection{The permutation maps between perturbed Gromov--Witten moduli spaces}
	\label{subsec:permutation-map}
{As we have explained in the introduction, naively there does not exist a permutation map for an arbitrary choice of abstract perturbation.}

Fix a permutation $\sigma \in S_k,\  \sigma: \{1,\ldots, k\}\to \{1,\ldots,k\}$.
Associated to this permutation we can define a $\ssc$-smooth permutation map between the Gromov--Witten polyfold
\[
\sigma: \cZ_{A,g,k} \to \cZ_{A,g,k}, \qquad [\Sigma,j,M,D,u]  \mapsto [\Sigma,j,M^\sigma,D,u]
\]
where $M = \{z_1,\ldots,z_k\}$ and where $M^\sigma := \{z'_1,\ldots,z'_k\}$, $z'_i:= z_{\sigma(i)}$.

Consider the pullback via $\sigma$ of the strong bundle $\cW_{A,g,k} \to \cZ_{A,g,k}$ and the Cauchy--Riemann section $\delbarj$, as illustrated in the below commutative diagram.
	\[
	\begin{tikzcd}
	\sigma^* \cW_{A,g,k} \arrow[d, "\sigma^* \delbarj \quad"'] \arrow[r, "\text{proj}_2"'] & \cW_{A,g,k} \arrow[d, "\quad \delbarj"] &  \\
	\cZ_{A,g,k} \arrow[r, "\sigma"'] \arrow[u, bend left] & \cZ_{A,g,k} \arrow[u, bend right] & 
	\end{tikzcd}
	\]
%We may naturally identify the bundles $\sigma^* \cW_{A,g,k}$ and $\cW_{A,g,k}$ and the Cauchy--Riemann sections 
The map $\sigma$ is a homeomorphism when considered on the underlying topological spaces, and hence satisfies the topological pullback condition.
By applying Theorem~\ref{thm:compatible-pullbacks} we immediately obtain Corollary~\ref{cor:pullback-via-permutation}.

It follows that the permutation map restricts to a well-defined map between the perturbed Gromov--Witten moduli spaces, 
\[
\sigma|_{\cS_{A,g,k} (\delbarj, \text{proj}_2^* \Lambda)} : \cS_{A,g,k} (\delbarj, \text{proj}_2^* \Lambda) \to \cS_{A,g,k} (\delbarj,\Lambda).
\]
Considered on the underlying topological spaces, this map is a homeomorphism.  Considered on the branched ep-subgroupoid structures, the associated functor
\[
\hat{\sigma}|_{S_{A,g,k} (\hdelbarj, \text{proj}_2^* \hat{\Lambda})}: S_{A,g,k} (\hdelbarj, \text{proj}_2^* \hat{\Lambda}) \to S_{A,g,k} (\hdelbarj,\hat{\Lambda})
\]
is a local diffeomorphism, and moreover is injective.
The restricted permutation map $\sigma$ and its associated functor $\hat{\sigma}$ are both weight preserving, i.e.,
$(\Lambda\circ \delbarj) \circ \sigma = \text{proj}_2^*\Lambda \circ \delbarj$ and $(\hat{\Lambda}\circ \hdelbarj) \circ \hat{\sigma} = \text{proj}_2^* \hat{\Lambda} \circ \hdelbarj$.
%\red{In particular, the hypothesis of change of variables are satisfied}

\appendix

\section{Local surjectivity of the linearized Cauchy--Riemann operator}
	\label{appx:local-surjectivity}

%\red{REVISE SOME OF THE NOTATION IN THIS SECTION}

We recall some basic facts about the standard, linear Cauchy--Riemann operator.

\begin{proposition}[{\cite[Prop.~4.15]{HWZGW}}]
	\label{prop:linear-cr-cylinder}
	Let $H_c^{3,\delta_0}(\R\times S^1,{\mathbb R}^{2n})$ be the $\ssc$-Hilbert space with antipodal asymptotic constants, where the level $m$ has regularity $(m+3,\delta_m)$.  Let $J(0)$ be a constant almost complex structure on $\R^{2n}$.
	The Cauchy--Riemann operator
	\[
	\partial_s+J(0)\partial_t: H^{3,\delta_0}_c( \R\times S^1,\R^{2n} )\rightarrow H^{2,\delta_0}(\R\times S^1, \R^{2n} )
	\]
	is a $\ssc$-isomorphism.
\end{proposition}

This expression also gives the formula for the filled section for the Cauchy--Riemann operator, (see \cite[pp.~129--130]{HWZGW}).  We have a similar result for disks.

\begin{proposition}
	\label{prop:linear-cr-disk}
	Consider the $\ssc$-Hilbert space $H^3_\text{loc}(\mathbb{D},\R^{2n})$.  Let $J(0)$ be a constant almost complex structure on $\R^{2n}$.
	The Cauchy--Riemann operator
	\[
	\partial_s+J(0)\partial_t: H^3_\text{loc}(\mathbb{D},\R^{2n}) \to H^2_\text{loc}(\mathbb{D},\R^{2n})
	\]
	is surjective.
\end{proposition}
\begin{proof}
	This follows from \cite[Exer.~B.3.3]{MSbook}, where solutions can be constructed using the existence of solutions to the Laplacian.
	\begin{comment}
	We recall the argument therein for the convenience of the reader.  Identify $(\R^{2n},J(0))$ with $(\C^n,i)$.  In order to solve
	\[
	(\partial_s+ i \partial_t) w = h
	\]
	write $w = u +iv$ and $h= f+ ig$ and solve the second order equations
	\[
	\Delta u_j = \partial_s f_j + \partial_t g_j, \qquad \Delta v_j = \partial_s g_j - \partial_t f_j.
	\]
	\end{comment}
\end{proof}

Consider now the Cauchy--Riemann section, defined on the underlying sets of the polyfold $\cZ_{A,g,k}^{3,\delta_0}$ and strong polyfold bundle $\cW_{A,g,k}^{2,\delta_0}$ by the equation
	\[
	[\Sigma,j,M,D,u] \mapsto [\Sigma,j,M,D,u,\tfrac{1}{2}(du+J(u)\circ du \circ j))].
	\]
In local $\ssc$-coordinates it has the following local expression
	\[
	(a,v,\eta) \mapsto (a,v,\eta, \overline{\xi})
	\]
where $\overline{\xi}$ is the unique solution of the equations
	\begin{align}
	\Gamma (\oplus_a \exp_u \eta, \oplus_a u) \cdot \hat{\oplus}_a \overline{\xi} \circ \delta(a,v) &= \delbar_{J,j(a,v)} (\oplus_a \exp_u \eta),\nonumber \\
	\hat{\ominus}_a \overline{\xi} \cdot \frac{\partial}{\partial_s} &= 0, \label{eq:local-expression-cr-operator}
	\end{align}
where:
\begin{itemize}
	\item $\delta(a,v):(T\Sigma_a,j(a,v)) \to (T\Sigma_a,j(a,0))$ is the complex linear map given by $\delta(a,v)h=\tfrac{1}{2}(id - j(a,0)\circ j(a,v)) h$,
	\item $\Gamma$ is defined via parallel transport of a complex connection, as follows.
	Fix a complex connection on the almost complex vector bundle $(TQ, J)\to Q$, i.e., if $\nabla_X$ is the covariant derivative on $Q$ belonging to the Riemaniann metric $\omega\circ (\id \oplus J)$, the connection $\tilde{\nabla}_X$, defined by $\tilde{\nabla}_XY=\nabla_XY-\frac{1}{2}J(\nabla_XJ)Y$, defines a complex connection, in the sense that it satisfies $\tilde{\nabla}_X(JY)=J(\tilde{\nabla}_XY)$.
	If $\eta\in T_pQ$ is a tangent vector, the parallel transport of a complex connection along the path $t\mapsto \exp_p (t\eta)$ for $t\in [0,1]$, defines the linear map
	\[
	\Gamma  (\exp_p(\eta), p):(T_pQ, J(p))\to (T_{\exp_p(\eta)}Q, J( \exp_p(\eta)))
	\]
	which is complex linear, hence $\Gamma (\exp_p(\eta), p)\circ J(p)= J( \exp_p(\eta))\circ  \Gamma (\exp_p(\eta), p)$.
\end{itemize}
A full explanation of these details can be found in \cite[p.~118, p.~126]{HWZGW}.

We can simplify this expression by fixing the coordinates $a=0$, $v=0$.  Moreover, we may identify a neighborhood of a point $q\in Q$ with a neighborhood of $0\in\R^{2n}$ under which the Euclidean metric pulls back to the Riemannian metric on $Q$.  The formula \eqref{eq:local-expression-cr-operator} defining $\overline{\xi}$  now becomes
\[
\overline{\xi} = \Gamma (u+ \eta, u)^{-1} \cdot \left( \partial_s(u+\eta) + J(u+\eta ) \partial_t(u+\eta)	\right).
\]
Consider the linearization at a solution $\delbar_{J,j} u =0$ with respect to the coordinate $\eta$.  This linearization is given as follows:
	\[
	\eta \mapsto \frac{1}{2} \left(	\partial_s \eta + J(u) \partial_t \eta + \partial_\eta J(u) \partial_t u	\right).
	\]

\begin{lemma}[Local surjectivity of the linearized Cauchy--Riemann operator]
	\label{lem:local-surjectivity-cauchy-riemann}
	Let $(\Sigma,\allowbreak j,\allowbreak M,\allowbreak D,\allowbreak u)$ be a stable map which is a solution to the Cauchy--Riemann operator $\hdelbar$.  Let 
	\[
	D_{u} \hdelbar : H^{3,\delta_0} (\Sigma, u^*TQ) \to H^{2,\delta_0}(\Sigma,\Lambda^{0,1} \otimes_J u^* TQ)
	\]
	be the linearization of $\hdelbar$ at $(\Sigma,j,M,D,u)$, considered as an $\ssc$-Fredholm operator between $\ssc$-Banach spaces, where we have fixed the complex structure on $(\Sigma,j)$ and the gluing parameters.
	Then there exist open subsets of the Riemann surface $\Sigma$:
	\begin{itemize}
		\item a disk-like neighborhood $D_{z_i}$ at every marked point $z_i \in M$ (regardless of whether we require $u$ is of class $H^{3,\delta_0}$ or of class $H^3_\text{loc}$ at $z_i$)
		\item disk-like neighborhoods $D_{x_a}\sqcup D_{y_a}$ at every nodal pair $\{x_a,y_a\}\in D$
		\item (if it exists) a component $S^2\subset \Sigma$ with two punctures, on which $u$ is constant
	\end{itemize}
	such that the restriction of $D_{u} \hdelbar$ to each of these regions is a surjective operator.
\end{lemma}
\begin{proof}
	We prove the existence of the first neighborhood, assuming $u$ is of class $H^3_\text{loc}$.  Consider the operator 
	\[
	D_{u} \hdelbar : H^3_\text{loc} (\mathbb{D}, \R^{2n}) \to H^2_\text{loc} (\mathbb{D}, \R^{2n})
	\]
	which is defined by the local expression we have just discussed, i.e.,
	\[
	D_{u} \hdelbar \eta = \frac{1}{2} \left(	\partial_s \eta + J(u) \partial_t \eta + \partial_\eta J(u) \partial_t u	\right).
	\]
	Moreover, assume we have identified a neighborhood of $Q$ with a neighborhood of $\R^{2n}$ such that $u(0) =0$.
	
	For every $\epsilon>0$ there exists a $\delta>0$ such that we have the following estimates for the ball $B_\delta(0)\subset \mathbb{D}$:
	\[
	\norm{(J(u)-J(0)\partial_t \eta }_{H^2_\text{loc} (B_\delta, \R^{2n})} \leq \frac{\epsilon}{2} \cdot \norm{\eta}_{H^3_\text{loc} (B_\delta, \R^{2n})}
	\]
	and
	\[
	\norm{\partial_\eta J(u) \partial_t u}_{H^2_\text{loc} (B_\delta, \R^{2n})} \leq \frac{\epsilon}{2} \cdot \norm{\eta}_{H^3_\text{loc} (B_\delta, \R^{2n})}.
	\]
	We consider the restriction of $D_{u} \hdelbar$ to $H^3_\text{loc} (B_\delta, \R^{2n})$; we may write
	\[
	D_{u} \hdelbar \eta = \frac{1}{2} \left(	\partial_s \eta + J(0) \partial_t \eta\right)
	+ \frac{1}{2} \left((J(u)-J(0))\partial_t \eta\right)
	+ \frac{1}{2} \left(\partial_\eta J(u) \partial_t u \right).
	\]
	From Proposition~\ref{prop:linear-cr-disk} the first term on the right is surjective, while we can bound the other two terms on the right in the operator norm by $\epsilon$.  From classical functional analysis %(e.g., \cite[Theorem 3.4]{lang1993real})
	the space of surjective operators is open.  Hence there exists some $\delta$ such that
	\[
	D_{u} \hdelbar : H^3_\text{loc} (B_\delta, \R^{2n}) \to H^2_\text{loc} (B_\delta, \R^{2n})
	\]
	is surjective.
	
	We prove the existence of first neighborhood, assuming $u$ is of class $H^{3,\delta_0}$.
	By symmetry, this will also show the existence of the second neighborhood.  Consider the operator 
	\[
	D_{u} \hdelbar : H^{3,\delta_0}_c (\R^+\times S^1, \R^{2n}) \to H^{2,\delta_0} (\R^+\times S^1, \R^{2n})
	\]
	which is defined by the same expression as before.  Moreover, assume we have identified a neighborhood of $Q$ with a neighborhood of $\R^{2n}$ such that $\lim_{s\to \infty} u(s) =0$.
	We proceed the same as before.  By \cite[Lem.~4.19]{HWZGW} there exists $R\geq 0$ such that we have the following estimate for the region $[R,\infty)\times S^1 \subset \R^+\times S^1$,
	\[
	\norm{(J(u)-J(0)\partial_t \eta }_{H^{2,\delta_0} ([R,\infty)\times S^1, \R^{2n})} \leq \frac{\epsilon}{2} \norm{\eta}_{H^{3,\delta_0}_c ([R,\infty)\times S^1, \R^{2n})}.
	\]
	The same argument shows that there exists $R\geq 0$ such that
	\[
	\norm{\partial_\eta J(u) \partial_t u}_{H^{2,\delta_0} ([R,\infty)\times S^1, \R^{2n})} \leq \frac{\epsilon}{2} \norm{\eta}_{H^{3,\delta_0}_c ([R,\infty)\times S^1, \R^{2n})}.
	\]
	One should be careful to note the presence of the exponential weights in the above norms.  Using Proposition~\ref{prop:linear-cr-cylinder}, we may use the same argument to conclude that there exists some $R\geq 0$ such that 
	\[
	D_{u} \hdelbar : H^{3,\delta_0}_c ([R,\infty)\times S^1, \R^{2n}) \to H^{2,\delta_0} ([R,\infty)\times S^1, \R^{2n})
	\]
	is surjective.
	
	We prove the existence of the third neighborhood.  Noting that $u$ is constant on the component $S^2$, and assuming in our chart $u(S^2)=0$ the local expression for the linearized Cauchy--Riemann operator
	\[
	D_{u} \hdelbar :	H^{3,\delta_0}_{a,b}( \R\times S^1,\R^{2n} )\rightarrow H^{2,\delta_0}(\R\times S^1, \R^{2n} )
	\]
	is given by 
	\[
	\eta \mapsto \partial_s+J(0)\partial_t \eta
	\]
	where $H^{3,\delta_0}_{a,b}( \R\times S^1,\R^{2n} )$ is the $\ssc$-Hilbert space of maps with asymptotic constant $a$ as $s\to -\infty$ and asymptotic constant $b$ as $s\to +\infty$.  By Proposition~\ref{prop:linear-cr-cylinder}, we can observe that this is a surjective Fredholm operator, with kernel the constant maps.
\end{proof}

As a consequence of the above lemma, we obtain the following.

\begin{corollary}
	\label{cor:vectors-which-span-cokernel}
	Shrink the above small disk-like neighborhoods slightly.  If necessary, shrink further in order to assume the regions are all disjoint.  Then there exist vectors $v_1, \ldots, v_k\in H^{2,\delta_0}(\Sigma,\Lambda^{0,1} \otimes_J u^* TQ)$ such that
	\begin{itemize}
		\item $v_1, \ldots, v_k$ together with $\text{Im} (D_{u} \hdelbar)$ span $H^{2,\delta_0}(\Sigma,\Lambda^{0,1} \otimes_J u^* TQ)$
		\item $v_1, \ldots, v_k$ vanish on the above regions of $\Sigma$.
	\end{itemize}
\end{corollary}

\begin{remark}
	This is not the full linearization of the $\ssc$-smooth proper Fredholm operator $\delbarj :\cZ \to \cW$, rather the linearization restricted to the subset $a=0, v=0$.  However, the image of the full linearization contains $\text{Im} (D_{u} \hdelbar)$, so all this implies is that the number of vectors in the above corollary will be greater or equal to the dimension of the cokernel of the full linearization.
\end{remark}

%%%%%%%%%%%%%%%%%%%%%%%%%%%%%%%%%%%%%%%%%%%%%%%%%%
% BIBLIOGRAPHY %%%%%%%%%%%%%%%%%%%%%%%%%%%%%%%%%%%
%%%%%%%%%%%%%%%%%%%%%%%%%%%%%%%%%%%%%%%%%%%%%%%%%%

%%%%%%%%%%%%%%%%%%%%%%%%%%%%%%%%%%%%%%%%%%%%%%%%%%
% End %%%%%%%%%%%%%%%%%%%%%%%%%%%%%%%%%%%%%%%%%%%%
%%%%%%%%%%%%%%%%%%%%%%%%%%%%%%%%%%%%%%%%%%%%%%%%%%

\end{document}